\documentclass[11pt,reqno,upgreek]{amsart}
\usepackage{amsmath,amsthm,amssymb,bm,stmaryrd,amsxtra,mathscinet}
\usepackage{mathrsfs}
\usepackage[all]{xy}
\usepackage{tikz}
\usepackage{comment}
\usepackage{eufrak}
\usepackage{yfonts}
\usepackage{upgreek}

\UseRawInputEncoding


\makeatletter

\@addtoreset{equation}{subsection}
\makeatother

\newtheorem{Thm*}{Theorem}[section]
\newtheorem{Thm}{Theorem}[subsection]
\newtheorem{Lem}[Thm]{Lemma}

\newtheorem{Prop}[Thm]{Proposition}
\newtheorem{Prop*}[Thm*]{Proposition}
\newtheorem{Rem}[Thm]{Remark}
\newtheorem{Def*}[Thm*]{Definition}
\newtheorem{Def}[Thm]{Definition}

\newtheorem*{ThmI}{Theorem A}
\newtheorem{Exa}[Thm]{Example}

\newcommand{\C}{\mathbb{C}}           
\newcommand{\Z}{\mathbb{Z}}
\newcommand{\Q}{\mathbb{Q}}

\newcommand{\Hom}{\mathrm{Hom} \,}

\newcommand{\cl}{\mathrm{cl}}

\newcommand{\fg}{{\mathfrak g}}

\newcommand{\ga}{\alpha}
\newcommand{\gb}{\beta}

\newcommand{\gl}{\lambda}
\newcommand{\gL}{\Lambda}
\newcommand{\gd}{\delta}
\newcommand{\gD}{\Delta}
\newcommand{\gT}{\Theta}
\newcommand{\gt}{\theta}

\newcommand{\gO}{\Omega}
\newcommand{\gG}{\Gamma}

\newcommand{\gS}{\Sigma}

\newcommand{\gee}{\varepsilon}

\newcommand{\ol}{\overline}

\newcommand{\wti}{\widetilde}

\newcommand{\wt}{\mathrm{wt}}

\newcommand{\bk}{\mathbf{k}}

\newcommand{\gmod}[1]{{#1}\mathrm{\mathchar`-gmod}}

\renewcommand{\Hom}[1]{\mathrm{Hom}_{#1}}

\newcommand{\sg}{\mathsf{g}}

\newcommand{\bS}{\mathbb{S}}

\newcommand{\br}{\mathbf{r}}

\newcommand{\cF}{\mathcal{F}}
\newcommand{\cP}{\mathcal{P}}

\newcommand{\cA}{\mathcal{A}}

\newcommand{\cD}{\mathcal{D}}
\newcommand{\cL}{\mathcal{L}}

\newcommand{\scC}{\mathscr{C}}
\newcommand{\scD}{\mathscr{D}}
\newcommand{\scL}{\mathscr{L}}

\newcommand{\tfd}{\textfrak{d}}

\newcommand{\hd}{\mathrm{hd}}
\newcommand{\soc}{\mathrm{soc}}

\newcommand{\nab}{\, \nabla \,}
\newcommand{\Del}{\, \Delta \,}

\newcommand{\sL}{\mathsf{L}}

\newcommand{\Bup}{\mathbf{B}^{\mathrm{up}}}

\newcommand{\Irr}{\mathrm{Irr}}

\newcommand{\wh}[1]{\widehat{#1}}

\newcommand{\bi}{\bm{i}}

\newcommand{\scaleg}{\scalebox{0.6}{$\fg$}}

\newcommand{\up}{\mathrm{up}}

\newcommand{\HF}{\mathrm{HF}}
\newcommand{\HFt}{\mathrm{HF}^{\mathrm{tw}}}
\newcommand{\mrD}{\mathrm{D}}
\newcommand{\mrU}{\mathrm{U}}
\newcommand{\tw}{\mathrm{tw}}

\newcommand{\Vj}{V\!\langle j \rangle}
\newcommand{\aj}{\langle j \rangle}

\newcommand{\FT}{\mathrm{FT}}

\title[Strong duality Data of type $A$ and extended $T$-systems]
{Strong duality Data of type $A$ and \\ extended $T$-systems}

\author[K.~Naoi]{Katsuyuki Naoi}
\address[K.~Naoi]{%
Institute of Engineering \\
Tokyo University of Agriculture and Technology\\
2-24-16 Naka-cho, Koganei-shi, Tokyo 184-8588, JAPAN}
\email{naoik@cc.tuat.ac.jp}


\begin{document}

\begin{abstract}
 The extended $T$-systems are a number of relations in the Grothendieck ring of the category of finite-dimensional modules over the quantum affine algebras of types $A_n^{(1)}$
 and $B_n^{(1)}$,
 introduced by Mukhin and Young as a generalization of the $T$-systems.
 In this paper we establish the extended $T$-systems for more general modules, which are constructed from an arbitrary strong duality datum of type $A$.
 Our approach does not use the theory of $q$-characters, and so also provides a new proof to the original Mukhin--Young's extended $T$-systems.
\end{abstract}

\maketitle

\section{Introduction}

The $T$-systems are remarkable functional relations appearing in solvable lattice models (see \cite{MR2591898} and references therein).
Let $\scC_\fg$ denote the category of integrable finite-dimensional modules over a quantum affine algebra $U_q'(\fg)$.
It has been proved by Nakajima \cite{MR1993360} and Hernandez \cite{MR2254805,MR2576287} that the $q$-characters (or equivalently, the classes in the Grothendieck ring)
of Kirillov--Reshetikhin (KR) modules in $\scC_\fg$ satisfy
the $T$-systems.
These $T$-system relations of KR modules play an essential role in the recent developments of the theory of monoidal categorifications of cluster algebras 
(\cite{MR2682185,zbMATH06578174,kashiwara2020monoidal,kashiwara2103monoidal}).

\textit{Snake modules} are a relatively large family of simple modules in $\scC_\fg$ of types $A_n^{(1)}$ and $B_n^{(1)}$ introduced in \cite{zbMATH06084095},
which contain all minimal affinizations (\cite{MR1367675}) of these types.
Via the monomial parametrization of simple modules in $\scC_\fg$ (see \cite{MR1745260}),
each snake module is expressed as $L(\prod_{r} Y_{i_r,aq^{k_r}})$ with a sequence $\big((i_1,k_1),\ldots,(i_p,k_p)\big) \in (I_{\fg_0} \times \Z)^p$ satisfying 
some combinatorial conditions, where $I_{\fg_0}$ denotes the index set of the simple Lie subalgebra $\fg_0$ of $\fg$.

Mukhin and Young introduced in \cite{MR2960028} a number of relations satisfied by the classes of snake modules as generalizations of the $T$-systems.
They call them the \textit{extended $T$-systems}.
Let $L(\prod_{r=1}^p Y_{i_r,aq^{k_r}})$ be a prime snake module (recall that a simple module having no nontrivial tensor factorization is called prime).
The corresponding extended $T$-system is the relation in the Grothendieck ring of the form
\[ \left[L(\prod_{r=1}^{p-1} Y_{i_r,aq^{k_r}})\right]\left[L(\prod_{r=2}^p Y_{i_r,aq^{k_r}})\right]= \left[L(\prod_{r=1}^p Y_{i_r,aq^{k_r}})\right]\left[L(\prod_{r=2}^{p-1}Y_{i_r,aq^{k_r}})\right]
   +\left[M\right] \left[N\right],
\]
where $M$ and $N$ are other snake modules.
In \cite{MR2960028}, these relations were proved using the formula for the $q$-characters of snake modules established in \cite{zbMATH06084095}.

It is a natural problem to find extended $T$-systems (containing, at least, all the minimal affinizations) in other types (see \cite[Appendix A]{MR2960028}). 
Such a family of relations was found for type $G_2$ in \cite{zbMATH06289987} and for type $C_3$ in \cite{zbMATH06417500}, 
but in general types this is still open as far as the author knows.
One essential obstacle is that the $q$-characters of minimal affinizations in other types are more complicated: they are not thin or special 
in the terminology of $q$-characters, while 
the $q$-characters of all the snake modules (in types $A_n^{(1)}$ and $B_n^{(1)}$) have these properties.
To overcome this difficulty, one possible way is to establish an approach to extended $T$-systems not relying on the theory of $q$-characters.

For the $T$-systems, such an approach was developed in \cite{kashiwara2103monoidal} using a \textit{strong duality datum} and \textit{affine cuspidal modules} introduced in \cite{kashiwara2020pbw}.
Let $\sg$ be a simple Lie algebra of type $ADE$.
A strong duality datum associated with $\sg$ is a family of simple 
$U_q'(\fg)$-modules $\cD=\{\sL_i\}_{i \in I_\sg} \subseteq \scC_\fg$, characterized using the invariant $\tfd$ introduced in 
\cite{kashiwara2020monoidal}.
Given a strong duality datum $\cD$ and a reduced word $\bm{i}=(i_1,\ldots,i_N)$ of the longest element of the Weyl group of $\sg$, 
one can construct the associated affine cuspidal modules $\{S_k^{\cD,\bm{i}}\}_{k \in \Z} \subseteq \scC_\fg$.
When the pair $(\cD, \bm{i})$ is associated with a \textit{$Q$-datum} (\cite{zbMATH07355935}), the affine cuspidal modules $\{S_k^{\cD,\bm{i}}\}_{k\in\Z}$ consist of fundamental modules.
In \cite{kashiwara2103monoidal}, the authors showed for any pair $(\cD,\bm{i})$ that \textit{affine determinantial modules}, each of which is constructed as the head of 
the tensor product of some $S^{\cD,\bm{i}}_k$'s,
satisfy exact sequences corresponding to the $T$-systems.
When $(\cD,\bm{i})$ is associated with a $Q$-datum, affine determinantial modules coincide with KR modules, and  
the classical $T$-systems are recovered.
This result was obtained by applying several properties of strong duality data and affine cuspidal modules, instead of $q$-characters.

In this paper, by applying a similar approach, we generalize the extended $T$-systems of type $A_n^{(1)}$ and $B_n^{(1)}$ 
to more general modules constructed from an arbitrary strong duality datum of type $A$,
which we hope to be a first step toward extended $T$-systems of general types.

In the first part of this paper, we give a sufficient condition for the head of the tensor product of affine cuspidal modules of general type
to satisfy a short exact sequence similar to the $T$-system.
More explicitly, for an arbitrary pair $(\cD,\bm{i})$ of a strong duality datum and a reduced word and the associated affine cuspidal modules $S_k:=S_k^{\cD,\bm{i}}$ ($k \in \Z$),
the following theorem is proved.

\begin{ThmI}[Theorem \ref{Thm:Main_thm_general}]
 Let $\bm{k}=(k_1<\cdots<k_p)$ be an increasing sequence of integers with $p\geq 2$.
 For $1\leq a < b \leq p$, write $\bS_{\bm{k}}[a,b] = \hd(S_{k_a}\otimes S_{k_{a+1}} \otimes \cdots \otimes S_{k_b})$, and set $\bS_{\bm{k}} = \bS_{\bm{k}}[1,p]$.
 Assume that the following two conditions are satisfied:
 \begin{itemize}\setlength{\leftskip}{-15pt} 
  \item[(a)] for any $1\leq a < b \leq p$, we have $\tfd(S_{k_a},\bS_{\bm{k}}[a+1,b]) = 1$, and
  \item[(b)] for any $1\leq a < b \leq p$, we have $\tfd(\bS_{\bm{k}}[a,b-1],S_{k_b}) =1$.
 \end{itemize}
 Then there exists a short exact sequence 
 \begin{equation*}\label{eq:intro}
  0 \to \hd \left(\bigotimes_{a=1}^{p-1} (S_{k_a} \Del S_{k_{a+1}})\right) \to \bS_{\bm{k}}[1,p-1] \otimes \bS_{\bm{k}}[2,p] \to \bS_{\bm{k}} \otimes \bS_{\bm{k}}[2,p-1] \to 0,
 \end{equation*}
 where the tensor product in the second term is ordered from left to right, and $M \Del N$ denotes the socle of $M \otimes N$.
 Moreover, the first and third terms are both simple.
\end{ThmI}

We also give sufficient conditions for $\bS_{\bm{k}}$ to be prime or real (Propositions \ref{Prop:primeness} and \ref{Prop:reality}).

In the latter half of this paper, we focus on cuspidal modules $S_k^{\cD,\bm{i}}$ such that $\cD$ is associated with $\mathfrak{sl}_{n+1}$
and $\bm{i}$ belongs to either of two special families: reduced words \textit{adapted to a height function} and
\textit{adapted to a twisted height function} (see Subsection \ref{Subsection:quivers} for the definitions).
Let $\bm{i}^{\text{hf}}$ (resp.\ $\bm{i}^\tw$) be a reduced word adapted to a height (resp.\ twisted height) function (these notations are for this introduction only).
For an arbitrary strong duality datum $\cD$ associated with $\mathfrak{sl}_{n+1}$ and $\bm{i} \in \{\bm{i}^{\mathrm{hf}}, \bm{i}^\tw\}$,
we define an associated \textit{snake module} by the head of the tensor product of $S_k^{\cD,\bm{i}}$'s satisfying some conditions.
When $(\cD,\bm{i}^{\mathrm{hf}})$ (resp.\ $(\cD,\bm{i}^{\tw})$) is associated with a $Q$-datum corresponding to $U_q'(\fg)$ of type $A_n^{(1)}$ (resp.\ $B_{n_0}^{(1)}$ with $n=2n_0-1$),
these modules coincide with the Mukhin--Young's snake modules (recall that in both cases corresponding $Q$-data are of type $A_n$, see \cite{zbMATH07355935}).
We show that these snake modules satisfy short exact sequences of the form in Theorem A, and moreover give more concrete description to the first terms 
(Theorems \ref{Thm:MainA} and \ref{Thm:Main_twisted}).
When associated with a $Q$-datum, these recover the Mukhin--Young's extended $T$-systems.
We also show that snake modules are real, and give a necessary and sufficient condition for them to be prime.

We explain our strategy for the proof of the short exact sequences.
Given a strong duality datum $\cD$, through the \textit{quantum affine Schur--Weyl duality functor} $\cF_\cD$ (\cite{kang2018symmetric}), 
we can define a crystal structure on a subset of the isomorphism classes of simple modules in $\scC_\fg$,
which is isomorphic to the crystal base of $U_q^-(\sg)$.
Letting $\gee_i, \gee_i^*$ ($i \in I_\sg$) be the functions on this crystal,
the conditions (a) and (b) in Theorem A can be rephrased in terms of the values of these functions $\gee_i, \gee_i^*$ (see Proposition \ref{Prop:Kashiwara_Park} (c)) at
certain elements expressed in Lusztig's parametrization (\cite{lusztig1990canonical}) with respect to $\bm{i}$.
Reineke introduced in \cite{zbMATH01031515} a useful algorithm to calculate the values of $\gee_i$ and $\gee_i^*$ at an element expressed in Lusztig's parametrization with respect to
a reduced word of some type.
We can apply this algorithm to the word $\bm{i}^{\mathrm{hf}}$, and show that snake modules associated with $\bm{i}^{\mathrm{hf}}$ satisfy the conditions (a) and (b)
of Theorem A.
Hence the existence of the short exact sequences in the case $\bm{i}=\bm{i}^{\mathrm{hf}}$ is shown.

The Reineke's algorithm cannot be applied in the case $\bm{i} = \bm{i}^{\mathrm{tw}}$.
Instead, we give a detailed description of the connection between Lusztig's parametrizations with respect to $\bm{i}^{\mathrm{hf}}$ and $\bm{i}^\tw$.
A similar work was previously done in \cite[Section 12]{hernandez2019quantum}, but our result is more involved:
we give a transition formula for all elements corresponding to snake modules (Proposition \ref{Prop:change_of_words}).
Using this formula and the results of the previous case $\bm{i} = \bm{i}^{\mathrm{hf}}$, we can show the existence of the short exact sequences in the case $\bm{i}=\bm{i}^\tw$ as well.

Note that in \cite{MR2960028}, the extended $T$-systems are given in terms of the relations in the Grothendieck ring,
and therefore there are two possibilities of short exact sequences.
As another advantage of our approach, we can determine which one is correct.

Theorem A holds for a strong duality datum of a general type, not only of type $A$, and we hope that it will help us to study extended $T$-systems of other types.
One difficulty is that, not in type $A$, the Reineke's algorithm cannot be applied in full generality for any reduced word.
We hope to return this problem in the future.

This paper is organized as follows. 
In Section \ref{Section:simple_Lie_algebras}, we recall basic notions concerning a simple Lie algebra of type $ADE$, 
such as the upper global and dual PBW bases, and crystals.
In Section \ref{Section:QAA}, we recall the basic notions and several properties on quantum affine algebras, the invariants $\gL$ and $\tfd$, and affine cuspidal modules.
In Section \ref{Section:Reality,etc}, we give sufficient conditions for the head of the tensor product of affine cuspidal modules to be prime or real, and prove Theorem A.
In Section \ref{Section:Snake_modules} we give the definition of snake modules associated with $\cD$ of type $A$ and $\bm{i}^{\mathrm{hf}}$ or $\bm{i}^\tw$,
and several related notions.
In Section \ref{Section:Untwisted}, we show that snake modules associated with $\bm{i}^{\mathrm{hf}}$ satisfy the extended $T$-systems, 
and at the same time we discuss their reality and primeness.
In Section \ref{Section:twisted}, we show analogous assertions for snake modules associated with $\bm{i}^\tw$.


\section{Preliminaries on simple Lie algebras of type $ADE$}\label{Section:simple_Lie_algebras}
\noindent\textbf{Conventions.}\ 
\begin{itemize}\setlength{\leftskip}{-15pt}
 \item[(i)] For a base field $\bk$, we write $\otimes$ for $\otimes_{\bk}$ when no confusion is likely.
 \item[(ii)] For $a,b \in \Z$ such that $a\leq b$, we denote by $[a,b]$ the set $\{k \in \Z \mid a \leq k \leq b\}$. 
  We set $[a,b]=\emptyset$ if $a>b$.
\end{itemize}

\subsection{Basic notation}

Let $\sg$ be a complex simple Lie algebra of type $ADE$, with an index set $I$ and a Cartan matrix $A=(a_{ij})_{i,j \in  I}$.
Let $\ga_i$ ($i\in I$) be the simple roots, $R$ the root system, $R^+$ the set of positive roots, $P$ the weight lattice, $W$ the Weyl group with simple reflections $\{s_i \mid i \in  I \}$, and $w_0 \in W$ the longest element.
Denote by $\ell(w)$ for $w \in W$ the length of $w$, and set  $N =\ell(w_0)$.
For $i \in I$, define $i^* \in I$ by $w_0(\ga_i) = -\ga_{i^*}$.
Let 
\[ R(w_0) = \{\bm{i}= (i_1,\cdots,i_N) \in I^{N} \mid s_{i_1}\cdots s_{i_N} = w_0\}
\]
denote the set of reduced words of $w_0$.
For two words $\bm{i},\bi'$, we say $\bi$ and $\bi'$ are \textit{commutation equivalent} if $\bi'$ is obtained from $\bi$ by applying a sequence of operations which 
transform some adjacent components $(i,j)$ such that $a_{ij}=0$ into $(j,i)$.
An equivalence class for this relation is called a \textit{commutation class}.

\subsection{Dual PBW bases and crystals}\label{Subsection:dual_PBW}

Let $\bk$ be a base field containing $\Q(q)$,
and $U_q(\sg)$ the quantized enveloping algebra associated with $\sg$ over $\bk$ with 
generators $\{e_i,f_i,q^{h_i}\mid i \in I\}$.
Let $U_q^-(\sg)$ be the $\bk$-subalgebra of $U_q(\sg)$ generated by $f_i$ ($i \in  I$),
and denote by $\mathbf{B}^{\mathrm{up}}$ the \textit{upper global basis} (or \textit{dual canonical basis}) of $U_q^-(\sg)$ (see \cite{MR1115118}).
Let $\cA_0 \subseteq \Q(q)$ be the subring of rational functions that are regular at $q=0$, and set $\mathscr{L} \subseteq U_q^-(\sg)$ to be the $\cA_0$-span of $\Bup$.

We briefly recall dual PBW bases of $U_q^-(\sg)$. 
Let $T_i = T_{i,+}''$ ($i \in  I$) denote the algebra automorphism of $U_q(\sg)$ given in \cite[Chapter 37]{MR1227098},
and take a reduced word $\bm{i}=(i_1,\ldots,i_N) \in R(w_0)$.
For each $1\leq k \leq N$, set $\gb_k = s_{i_1}\cdots s_{i_{k-1}} (\ga_{i_k}) \in R^+$,
\begin{equation}\label{eq:dual_PBW}
 F^{\bm{i}}_{\mathrm{low}}(\gb_k) = T_{i_1}\cdots T_{i_{k-1}}(f_{i_k}), \ \ \ \text{and} \ \ \ F^{\bm{i}}(\gb_k) = \frac{F_{\mathrm{low}}^{\bm{i}}(\gb_k)}{\big(F_{\mathrm{low}}^{\bm{i}}(\gb_k),
 F^{\bm{i}}_{\mathrm{low}}(\gb_k)\big)},
\end{equation}
where $(\ , \ )$ is the bilinear form on $U_q^-(\sg)$ given in \cite[Section 3.4]{MR1115118}.
$F^{\bm{i}}(\gb_k)$ is called a \textit{dual root vector}.
We have $F^{\bm{i}}(\gb_k) \in \Bup$ for any $\bm{i}$ and $k$.
For $\bm{c}=(c_1,\ldots,c_N) \in \Z_{\geq 0}^N$, set
\[ F^{\bm{i}}(\bm{c}) = q^{\frac{1}{2}\sum_k c_k(c_k-1)}F^{\bm{i}}(\gb_1)^{c_1}\cdots F^{\bm{i}}(\gb_k)^{c_k}.
\]
Then $\{F^{\bm{i}}(\bm{c})\mid \bm{c} \in \Z_{\geq 0}^N\}$ forms a basis of $U_q^-(\sg)$,
and we call this the \textit{dual PBW basis} associated with $\bm{i}$.
For each $\bm{c} \in \Z_{\geq 0}^N$, there is a unique element $B^{\bm{i}}(\bm{c}) \in \Bup$ such that 
\[ F^{\bm{i}}(\bm{c}) \equiv B^{\bm{i}}(\bm{c}) \ \ \ \text{mod $q\scL$},
\]
see \cite[Theorem 4.29]{zbMATH06047791}.

Let $*$ be the $\bk$-algebra anti-involution on $U_q^-(\sg)$ defined by $*f_i = f_i$ ($i\in  I$).
This $*$ preserves $\Bup$ (\cite{MR1240605}, \cite[Lemma 3.5]{zbMATH06047791}).

\begin{Lem}[{\cite[Subsection 2.11]{lusztig1990canonical}}]\label{Lem:duality_of_B}
 For $\bm{i} = (i_1,\ldots,i_N) \in R(w_0)$ and $\bm{c} = (c_1,\ldots,c_N) \in \Z_{\geq 0}^N$, set 
 \[ \bm{i}^\vee = (i_N^*,i_{N-1}^*,\ldots,i_1^*) \in R(w_0) \ \ \text{and} \ \ \bm{c}^\vee=(c_N,c_{N-1},\ldots,c_1).
 \]
 Then we have $*F^{\bm{i}}(\bm{c}) = F^{\bm{i}^\vee}(\bm{c}^\vee)$ and $*B^{\bm{i}}(\bm{c}) = B^{\bm{i}^\vee}(\bm{c}^\vee)$.
\end{Lem}

\begin{Prop}\label{Prop:relations_of_FB}
 For any $\bm{i} \in R(w_0)$ and $\bm{c} \in \Z_{\geq 0}^N$, we have 
 \[ F^{\bm{i}}(\bm{c}) \in B^{\bm{i}}(\bm{c}) + \sum_{\bm{c}' \prec \bm{c}} q\Z[q] B^{\bm{i}}(\bm{c}'),
 \]
 where $\prec$ is the bi-lexicographic order on $\Z_{\geq 0}^N$, namely, $(a_1,\ldots,a_N)\prec (b_1,\ldots,b_N)$ 
 if and only if there are $1\leq k\leq l \leq N$ satisfying $a_k<b_k$, $a_l<b_l$, and $a_j=b_j$ for 
 all $j$ such that $j <k$ or $l<j$.
\end{Prop}

\begin{proof}
 Except for the triangularity with respect to the bi-lexicographic order, the assertion follows from \cite[Theorem 4.29]{zbMATH06047791}.
 The triangularity is proved by applying $*$ and using Lemma \ref{Lem:duality_of_B}.
\end{proof}


\begin{Prop}[{\cite[Subsection 2.3]{lusztig1990canonical}}]\label{Prop:2-move_3-move}
 Let $\bm{i},\bm{i}' \in R(w_0)$, and $\bm{c}, \bm{c}' \in \Z_{\geq 0}^N$.
 \begin{itemize}\setlength{\leftskip}{-15pt}
 \item[{\normalfont(i)}] Assume for some $1\leq k <N$ that $a_{i_ki_{k+1}}=0$, $i'_k = i_{k+1}$, $i'_{k+1} = i_k$, and $i_l = i'_l$ for all $l \neq k,k+1$.
  Then $B^{\bm{i}}(\bm{c}) = B^{\bm{i}'}(\bm{c}')$ holds if and only if 
 \[ c_k'=c_{k+1}, \ \ \ c_{k+1}' = c_k, \ \ \text{and} \ \ c_l'=c_l \ \ \text{for all } l \neq k,k+1.
 \]
 \item[{\normalfont(ii)}] Let $i,j \in I$ be such that $a_{ij}=-1$, and assume for some $1<k<N$ that $(i_{k-1},i_k,i_{k+1}) = (i,j,i)$, $(i_{k-1}',i_k',i_{k+1}')=(j,i,j)$, and
 $i_l=i_l'$ for all $l \notin \{k,k\pm 1\}$. 
 Then $B^{\bm{i}}(\bm{c}) = B^{\bm{i}'}(\bm{c}')$ holds if and only if $c_l'=c_l$ for all $l \notin \{k,k\pm 1\}$, and 
 \begin{equation}\label{eq:3-move}
  c_{k-1}'= c_k+c_{k+1}-c_0, \ \ c_k' = c_0, \ \ c_{k+1}' = c_{k-1}+c_{k}-c_0,
 \end{equation}
 where we set $c_0 = \min(c_{k-1},c_{k+1})$.
 \end{itemize}
\end{Prop}

By identifying $\Bup$ with its image under the projection $\scL \to \scL/q\scL$, we define the canonical (abstract) crystal structure on $\Bup$  (see \cite{MR1240605}).
Let
\[ \wt\colon \Bup \to P, \ \ \ \gee_i, \varphi_i\colon \Bup \to \Z \ \ \text{and} \ \ \tilde{e}_i,\tilde{f}_i\colon \Bup \to \Bup \sqcup \{0\} \ \ \text{for} \ 
   i\in  I 
\]
be the maps giving this crystal structure.
%
%
%
%
We also define maps $\gee^*_i$, $\varphi^*_i$, $\tilde{e}_i^*$, $\tilde{f}_i^*$ ($i \in  I $) on $\Bup$ by
\[ \gee^*_i=\gee_i \circ *, \ \ \ \varphi_i^* = \varphi_i \circ *, \ \ \ \tilde{e}_i^* = * \circ \tilde{e}_i \circ *, \ \ \ \tilde{f}_i^* = * \circ \tilde{f}_i \circ *.
\]
For $\bm{i}=(i_1,\ldots,i_N) \in R(w_0)$ and $\bm{c}=(c_1,\ldots,c_N) \in \Z_{\geq 0}^N$, it follows from \cite{MR1227098} and Lemma \ref{Lem:duality_of_B} that
\begin{equation}\label{eq:simple_case}
 \gee_{i_1} \big(B^{\bm{i}}(\bm{c})\big) = c_1 \ \ \ \text{and} \ \ \ \gee_{i_N^*}^*\big(B^{\bm{i}}(\bm{c})\big) = c_N.
\end{equation}


\section{Preliminaries on quantum affine algebras}\label{Section:QAA}

\subsection{Basic notation}

Let $\fg$ be an affine Kac--Moody Lie algebra with index set $I_\fg$ and simple roots $\{\ga_i^{\scalebox{0.6}{$\fg$}} \mid i \in I_\fg\}$.
Denote by $0 \in I_\fg$ the special element prescribed in \cite[Section 4]{MR1104219}, 
except $A_{2n}^{(2)}$-type in which we set $\ga_0^{\scaleg}$ to be the longest simple root.
Let $I^0_\fg=I_\fg \setminus \{0\}$,
$P^\fg$ be the weight lattice of $\fg$, and $P^\fg_{\cl} = P^\fg / (P^\fg \cap \Q \gd)$, where $\gd$ is the indivisible imaginary positive root.

Now we fix the base field $\bk$ to be the algebraic closure of $\C(q)$ in $\bigcup_{m >0} \C (\! (q^{1/m})\!)$, 
and denote by $U_q'(\fg)$ the quantized enveloping algebra over $\bk$ associated with $\fg$
with generators $\{e_i,f_i,q^h \mid i \in I_\fg, h \in P^{\fg,\vee}_{\cl}:=\Hom{\Z} (P^\fg_{\cl},\Z) \}$
(here, by abuse of notation, we use the same symbols with the generators of $U_q(\sg)$).
We call $U_q'(\fg)$ a \textit{quantum affine algebra} in the sequel.
Let $\gD \colon U_q'(\fg) \to U_q'(\fg) \otimes U_q'(\fg)$ denote the coproduct (we follow the convention in \cite[Section 7]{MR1890649}).

A $U_q'(\fg)$-module $M$ is said to be \textit{integrable} if $M = \bigoplus_{\gl \in P_\cl^\fg} M_\gl$ with $M_\gl = \{ v \in M\mid q^h v=q^{\langle h,\gl\rangle}v \ (h \in P_\cl^{\fg,\vee})\}$,
and $e_i, f_i$ ($i \in I_\fg$) act nilpotently on $M$.
We denote by $\scC_{\fg}$ the category of integrable finite-dimensional $U_q'(\fg)$-modules.
Let $\bm{1} \in \scC_\fg$ denote the trivial module.
For $M,N \in \scC_{\fg}$, the tensor product $M \otimes N$ is also an object of $\scC_{\fg}$ via the coproduct $\gD$,
and this gives a monoidal category structure on $\scC_\fg$ with unit object $\bm{1}$.
Moreover, the monoidal category $\scC_\fg$ is rigid,
namely, every object $M \in \scC_\fg$ has its right dual $\scD M$ and left dual $\scD^{-1} M$.
There are isomorphisms
\begin{align*}
 \mathrm{Hom}_{\scC_\fg}(M \otimes X,Y) &\cong \mathrm{Hom}_{\scC_\fg}(X,\scD M\otimes Y), \ \ \ \text{and}\\
 \mathrm{Hom}_{\scC_\fg}(X \otimes M, Y) &\cong \mathrm{Hom}_{\scC_\fg}(X,Y\otimes \scD^{-1}M),
\end{align*}
which are functorial in $X,Y \in \scC_\fg$.

For simple modules $M$ and $N$ in $\scC_\fg$, we say that $M$ and $N$ \textit{commute} if $M\otimes N \cong N \otimes M$,
and \textit{strongly commute} if $M\otimes N$ is simple.
Note that, if $M$ and $N$ strongly commute then they commute, since the Grothendieck ring of $\scC_\fg$ is commutative \cite{MR1745260}. 
We say $M$ is \textit{real} if $M$ strongly commutes with itself.

\begin{Prop}[\cite{zbMATH06424572}]\label{Prop:simple_heads}\
 \begin{itemize}\setlength{\leftskip}{-15pt}
  \item[{\normalfont(i)}] Let $M_j$ $(j=1,2,3)$ be a module in $\scC_\fg$, and assume that $M_2$ is simple.
   If $f\colon L \to M_2 \otimes M_3$ and $g\colon M_1 \otimes M_2 \to L'$ are nonzero homomorphisms, then the composition
   \[ M_1 \otimes L \stackrel{M_1 \otimes f}{\longrightarrow} M_1 \otimes M_2 \otimes M_3 \stackrel{g \otimes M_3}{\longrightarrow} L' \otimes M_3
   \]
   does not vanish.
  \item[{\normalfont(ii)}] 
   Let $M$ and $N$ be simple modules in $\scC_\fg$, and assume that one of them is real.
   Then both $M \otimes N$ and $N \otimes M$ have simple socles and simple heads.
 \end{itemize}
\end{Prop}

For $M \in \scC_\fg$, we denote by $\hd(M)$ (resp.\ $\soc(M)$) the head (resp.\ socle) of $M$.
For $M, N \in \scC_\fg$, we also use the notation $M\, \nabla\, N$ (resp.\ $M\, \Delta\, N$) to denote $\hd(M \otimes N)$ (resp.\ $\soc(M\otimes N)$).

\begin{Prop}[{\cite[Corollary 3.14]{zbMATH06424572}}]\label{Prop:bijective_of_head}
 Let $M,N$ be simple modules in $\scC_\fg$, and assume that $M$ is real.
 Then we have 
 \[ \scD^{-1}M \nab (N \nab M) \cong N \ \ \ \text{and} \ \ \ (M \nab N) \nab \scD M \cong N.
 \]
\end{Prop}


\subsection{R-matrices and invariants}\label{Subsection:R-matrices}

In this subsection we briefly recall the definitions and properties of some invariants on pairs of modules in $\scC_\fg$, which were introduced in \cite{kashiwara2020monoidal}.
For more details, see \cite{kashiwara2020monoidal,kashiwara2020pbw,kashiwara2103monoidal}.

For any simple module $M$ in $\scC_\fg$,
there is unique $\gl \in P^\fg_\cl$ such that $M_\gl \neq 0$ and $M_\mu =0$ unless $\mu \in \gl - \sum_{i \in I_\fg\setminus \{0\}}\Z_{\geq 0} \cl(\ga^{\scaleg}_i)$,
where $\cl\colon P^\fg \to P^\fg_{\cl}$ is the canonical projection.
We call a nonzero vector $u \in M_\gl$ an \textit{$\ell$-highest weight vector}, which is unique up to a scalar multiplication.

For $M \in \scC_\fg$ and an indeterminate $z$, denote by $M_z$ the $U_q'(\fg)$-module $\bk[z^{\pm 1}] \otimes M $ defined by
\[ e_i\big(g(z) \otimes u\big) = z^{\gd_{i0}}g(z)\otimes e_iu, \ \ \ f_i\big(g(z) \otimes u\big) = z^{-\gd_{i0}} g(z) \otimes f_iu, \ \ \ q^h\big(g(z) \otimes u\big)=g(z) \otimes q^hu
\]
for $i \in I_\fg$, $h \in P^{\fg,\vee}_\cl$, $g(z) \in \bk[z^{\pm 1}]$ and $u \in M$. 
We write $u_z = 1 \otimes u \in M_z$.

For simple modules $M,N$ in $\scC_\fg$ with $\ell$-highest weight vectors $u \in M$ and $v \in N$,
there exists a unique $\bk(z) \otimes U_q'(\fg)$-module isomorphism 
\[ R^{\mathrm{norm}}_{M,N_z}\colon \bk(z) \otimes_{\bk[z^{\pm 1}]} (M \otimes N_z) \stackrel{\sim}{\to} \bk(z) \otimes_{\bk[z^{\pm 1}]} (N_z \otimes M)
\]
satisfying $R^{\mathrm{norm}}_{M,N_z}(u\otimes v_z) = v_z \otimes u$ (\cite{MR1890649}).
$R^{\mathrm{norm}}_{M,N_z}$ is called the \textit{normalized R-matrix} of $M$ and $N$.
Let $d_{M,N}(z) \in \bk[z]$ be the monic polynomial of the smallest degree satisfying
\[ d_{M,N}(z)R_{M,N_z}^{\mathrm{norm}}(M\otimes N_z) \subseteq N_z \otimes M.
\]
The polynomial $d_{M,N}(z)$ is called the \textit{denominator} of $R_{M,N_z}^{\mathrm{norm}}$.
We denote by $\mathbf{r}_{M,N}$ the specialization at $z=1$:
\[ \mathbf{r}_{M,N} = \big(d_{M,N}(z)R_{M,N_z}^{\mathrm{norm}}\big)\Big|_{z=1}\colon M \otimes N \to N \otimes M.
\]
We call this nonzero homomorphism $\mathbf{r}_{M,N}$ the \textit{R-matrix} of $M$ and $N$.
If either $M$ or $N$ is real, then the image of $\mathbf{r}_{M,N}$ is isomorphic to $M \nab N$ and $N \Del M$ (\cite{zbMATH06424572}),
and in particular we have
\begin{equation}\label{eq:head_socle}
 M \nab N \cong N \Del M.
\end{equation}

\begin{Def}[\cite{kashiwara2020monoidal}]\normalfont
  Let $M,N$ be simple modules in $\scC_\fg$.
  \begin{itemize}\setlength{\leftskip}{-15pt}
   \item[(i)] We define $\tfd(M,N) \in \Z_{\geq 0}$ by the order of the zero of the polynomial $d_{M,N}(z)d_{N,M}(z)$ at $z=1$.
   \item[(ii)] We define $\gL(M,N) \in \Z$ by
    \begin{align*}
     \gL(M,N)&= \sum_{k \in \Z_{\geq 0}} (-1)^{k}\,\tfd(M,\scD^kN) - \sum_{k \in \Z_{<0}} (-1)^{k}\,\tfd(M,\scD^kN). 
    \end{align*}
  \end{itemize}
\end{Def}

\begin{Rem}\normalfont
 These definitions are different from \cite[Definitions 3.6 and 3.14]{kashiwara2020monoidal}, but equivalent to those by \cite[Propositions 3.16 and 3.22]{kashiwara2020monoidal}.
 In \cite{kashiwara2020monoidal}, the invariant $\gL(M,N)$ are defined for not necessarily simple $M$ and $N$ as well,
 but we do not need this since in the present paper we will only treat the cases where $M$ and $N$ are simple.
\end{Rem}

Let us list several properties of the invariants.

\begin{Prop}[{\cite[Lemma 3.7 and Corollaries 3.19 and 3.17]{kashiwara2020monoidal}}]\label{Prop:fundamental_properties_of_delta}
 Let $M$ and $N$ be simple modules in $\scC_\fg$.
 \begin{itemize}\setlength{\leftskip}{-15pt}
  \item[{\normalfont(i)}] We have
   \[ \tfd(M,N) = \frac{1}{2}\big(\gL(M,N) + \gL(N,M)\big)=\tfd(\scD M,\scD N).
   \]
  \item[{\normalfont(ii)}] Assume further that either $M$ or $N$ is real. 
   Then $M$ and $N$ strongly commute if and only if $\tfd(M,N)=0$.
 \end{itemize}
\end{Prop}

\begin{Prop}[{\cite[Proposition 4.2 and Lemma 3.10]{kashiwara2020monoidal}}]\label{Prop:weakly_decreasing}
 Let $X$, $Y$ and $Z$ be simple modules in $\scC_\fg$. 
  \begin{itemize}\setlength{\leftskip}{-15pt}
   \item[{\normalfont(i)}] For any simple subquotient $S$ of $X \otimes Y$,
    we have
    \[ \tfd(S,Z) \leq \tfd(X,Z) + \tfd(Y,Z).
    \]
   \item[{\normalfont(ii)}] Assume further that $X$ and $Y$ strongly commute.
    Then we have 
    \[ \tfd(X \otimes Y, Z) = \tfd(X,Z) + \tfd(Y,Z).
    \]
  \end{itemize}
\end{Prop}

\begin{Lem}\label{Lem:invariance_by_head}
 Let $X,Y,Z \in \scC_\fg$ be simple modules, and assume that $Z$ is real.
 \begin{itemize}\setlength{\leftskip}{-15pt}
  \item[{\normalfont (i)}] If $\tfd(X,Z) = \tfd(X,\scD^{-1} Z)=0$, then we have $\tfd(X, Y\nab Z) = \tfd(X,Y)$.
  \item[{\normalfont (ii)}] If $\tfd(X,Z) = \tfd(X,\scD Z)=0$, then we have $\tfd(X, Z \nab Y) = \tfd(X,Y)$.
 \end{itemize}
\end{Lem}

\begin{proof}
 (i) We have 
 \[ \tfd(X,Y \nab Z) \leq \tfd(X,Y) + \tfd(X,Z) =\tfd(X,Y)
 \]
 by Proposition \ref{Prop:weakly_decreasing}. On the other hand, we have 
 \[ \tfd(X,Y ) = \tfd\big(X, \scD^{-1} Z \nab (Y \nab Z)\big) \leq \tfd(X, \scD^{-1} Z) + \tfd(X, Y\nab Z) = \tfd(X,Y\nab Z)
 \]
 by Proposition \ref{Prop:bijective_of_head}. Hence the assertion is proved. 
 The proof of (ii) is similar.
\end{proof}

\begin{Prop}[{\cite[Proposition 4.7]{kashiwara2020monoidal}}]\label{Prop:length2}
 Let $M$ and $N$ be simple modules in $\scC_\fg$, and assume that one of them is real and $\tfd(M,N)=1$.
 Then the composition length of $M \otimes N$ is $2$, and we have an exact sequence
 \[ 0 \to N \nab M \to M \otimes N \to M \nab N \to 0.
 \]
\end{Prop}

\begin{Prop}[{\cite[Lemma 2.22]{kashiwara2103monoidal}}]\label{Prop:reality_of_MN}
 Let $M,N$ be real simple modules in $\scC_{\fg}$ such that $\tfd(M,N) \leq 1$. Then $M \nab N$ is real.
 
\end{Prop}

\begin{Prop}[{\cite[Proposition 2.17]{kashiwara2020pbw}}]\label{Prop:strongly_decreasing}
 Let $M,N$ be simple modules in $\scC_\fg$, and assume that $N$ is real.
 If $\tfd(M,N)>0$, then we have 
 \[ \tfd(S,N) < \tfd(M,N)
 \]
 for any simple subquotient $S$ of $M\otimes N$ and also for any simple subquotient $S$ of $N \otimes M$.
\end{Prop}

\begin{Prop}[{\cite[Proposition 2.25]{kashiwara2103monoidal}}]\label{Prop:three_terms}
 Let $X$, $Y$ and $Z$ be simple modules in $\scC_\fg$ such that $Y$ is real. Assume that
  \begin{itemize}\setlength{\leftskip}{-15pt}
   \item[{\normalfont(i)}] $\tfd(\scD X,Y)= \tfd(\scD Y, Z) =0$, and
   \item[{\normalfont(ii)}] $X \otimes Y \otimes Z$ has a simple head.
  \end{itemize}
    Then we have 
    \[ \tfd\big(Y,\hd(X\otimes Y \otimes Z)\big) = \tfd(Y,X\nab Y) + \tfd(Y, Y\nab Z).
    \]
\end{Prop}

Following \cite[Definition 4.14]{kashiwara2020monoidal} (see also \cite[Definition 2.5]{kashiwara2019laurent}), 
we say a sequence $(M_1,\ldots,M_r)$ of real simple modules in $\scC_\fg$ is a \textit{normal sequence} if the composition of the R-matrices
\begin{align}
  (\mathbf{r}_{M_{r-1},M_r}) \circ \cdots \circ (\mathbf{r}_{M_2,M_r}\circ \cdots \circ \mathbf{r}_{M_2,M_3})&\circ (\br_{M_1,M_r}\circ \cdots \circ \br_{M_1,M_2})\nonumber\\
  &\colon M_1 \otimes \cdots \otimes M_r \to M_r \otimes \cdots \otimes M_1\label{eq:composition_of_R}
\end{align}
does not vanish.

\begin{Prop}[{\cite[Lemma 4.15]{kashiwara2020monoidal}}]\label{Prop:normal_implies_simple_head}
 If $(M_1,\ldots,M_r)$ is a normal sequence of real simple modules in $\scC_\fg$, then $\hd(M_1 \otimes \cdots \otimes M_r)$ and $\soc(M_r\otimes \cdots \otimes M_1)$
 are simple and isomorphic to the image of the composition {\normalfont (\ref{eq:composition_of_R})}.
 Moreover both $(M_2,\ldots,M_r)$ and $(M_1,\ldots,M_{r-1})$ are normal sequences, and we have
  \begin{align*}
   \gL(M_1,\hd(M_2 \otimes \cdots \otimes M_r)) &= \sum_{k=2}^r \gL(M_1,M_k) \ \ \text{and}\\
   \gL(\hd(M_1 \otimes \cdots \otimes M_{r-1}), M_r) &= \sum_{k=1}^{r-1} \gL(M_k,M_r).
  \end{align*}
\end{Prop}

\begin{Prop}[{\cite[Lemma 4.17]{kashiwara2020monoidal}}]\label{Prop:normality}
 For real simple modules $X$, $Y$ and $Z$ in $\scC_\fg$, the triple $(X,Y,Z)$ is a normal sequence if
 $\scD X$ and $Z$ strongly commute. 
\end{Prop}

\begin{Lem}[{\cite[Lemma 2.24]{kashiwara2020pbw}}]\label{Lem:two_normal_sequences}
 For real simple modules $X,Y,Z$ in $\scC_\fg$, $\tfd(X, Y\nab Z) =\tfd(X,Y) + \tfd(X,Z)$ holds if and only if both $(X,Y,Z)$ and $(Y,Z,X)$ are normal sequences.
\end{Lem}

Following \cite[Definition 2.16]{kashiwara2103monoidal}, we say a sequence $(M_1,\ldots,M_r)$ of real simple modules in $\scC_{\fg}$ is 
\textit{unmixed} (resp.\ \textit{strongly unmixed}) if for all $1 \leq j < k \leq r$ we have
\[ \tfd(\scD M_j, M_k) = 0 \ \ \ \big(\text{resp.\ } \tfd(\scD^l M_j,M_k) = 0 \ \ \text{for all $l \in \Z_{>0}$}\big).
\]

\begin{Prop}[{\cite[Lemma 5.3]{kashiwara2020pbw}}]\label{Prop:unmixed_normal}
 Any unmixed sequence of real simple modules is a normal sequence.
\end{Prop}

\begin{Lem}[{\cite[Lemma 4.26]{kashiwara2103monoidal}}]\label{Lem:technical}
 Let $X,Y,Z$ be real simple modules in $\scC_\fg$, and assume the following two conditions:
  \begin{itemize}\setlength{\leftskip}{-15pt}
   \item[{\normalfont(i)}] $(X,Y,Z)$ is a normal sequence, and
   \item[{\normalfont(ii)}] $\gL(Y,X) +\gL(Y,Z) -\gL(Y,X\nab Z) = 2\,\tfd(X,Y)$.
  \end{itemize}
   Then the composition 
   \[ X \otimes Y \otimes Z \stackrel{\br_{X,Y}\otimes Z}{\longrightarrow} Y\otimes X \otimes Z \twoheadrightarrow Y \nab (X \nab Z)
   \]
   is surjective and induces an isomorphism $\hd(X\otimes Y \otimes Z) \stackrel{\sim}{\to} Y \nab (X \nab Z)$.
\end{Lem}


\subsection{Strong duality data}\label{Subsection:strong_duality_datum}

Let $\sg$ be a simple Lie algebra of type $ADE$ with a Cartan matrix $A=(a_{ij})_{i,j\in I}$ as in Section \ref{Section:simple_Lie_algebras}.
We freely use the notation in the section.
Recall that we mainly use plain symbols such as $I$, $\ga_i$, not for $\fg$ but $\sg$.

A module $L \in \scC_\fg$ is called a \textit{root module} if $L$ is a real simple module such that
\[ \tfd(L,\scD^{k}L) = \gd_{k,-1}+\gd_{k,1} \ \ \ \text{for any $k \in \Z$},
\]
see \cite[Section 3]{kashiwara2020pbw}.

\begin{Def}[{\cite[Definition 4.7]{kashiwara2020pbw}}]\normalfont
 Let $\mathcal{D}=\{\sL_i\}_{i \in  I}$ be a family of simple modules in $\scC_\fg$.
 \begin{itemize}\setlength{\leftskip}{-15pt}
  \item[(i)] We say $\mathcal{D}$ is a \textit{duality datum} associated with $\sg$ if the following conditions are satisfied: 
  \begin{enumerate}\setlength{\leftskip}{-15pt}
   \item[(a)] $\sL_i$ is real for all $i \in  I$, and
   \item[(b)] $\tfd(\sL_i,\sL_j) = -a_{ij}$ for all $i,j \in I$ such that $i\neq j$.
  \end{enumerate}
  \item[(ii)] We say $\mathcal{D}$ is a \textit{strong duality datum} associated with $\sg$ if the following conditions are satisfied:
  \begin{enumerate}\setlength{\leftskip}{-15pt}
   \item[(c)] $\sL_i$ is a root module for all $i \in I$, and
   \item[(d)]   $\tfd(\sL_i,\scD^k \sL_j) = -\gd_{k,0}a_{ij}$ for all $k \in \Z$ and $i,j \in I$ such that $i\neq j$.
  \end{enumerate}
 \end{itemize}
\end{Def}

Let $Q_+ = \sum_{i \in I}\Z_{\geq 0}\ga_i$, and for $\gb \in Q_+$ let $R(\gb) = R^\sg(\gb)$ be a \textit{symmetric quiver Hecke algebra} at $\gb$ associated with $\sg$ (see \cite{kang2018symmetric}).
If $\cD=\{\sL_i\}_{i \in I}$ is a duality datum associated with $\sg$, then there exists an exact monoidal functor (\textit{quantum affine Schur--Weyl duality functor})
\[ \cF_{\cD}\colon \bigoplus_{\gb\in Q_+} \gmod{R(\gb)} \to \scC_\fg,
\] 
satisfying $\cF_{\cD}(q M ) \cong \cF_{\cD} (M)$ for any $M \in \bigoplus_{\gb} \gmod{R(\gb)}$ and $\cF_{\cD}\big(L(i)\big) = \sL_i$ for all $i \in  I$ (\cite{kang2018symmetric}).
Here $\bigoplus_{\gb} \gmod{R(\gb)}$ is the direct sum of the categories of finite-dimensional graded $R(\gb)$-modules equipped with a monoidal structure via the convolution product,
$q$ the grading shift of degree $1$,
and $L(i)$ the $1$-dimensional simple module over $R(\ga_i)$.
In the sequel, we write $\gmod{R} = \bigoplus_{\gb} \gmod{R(\gb)}$.
By \cite{khovanov2009diagrammatic,rouquier20082}, there is a unique $\Z[q,q^{-1}]$-algebra isomorphism
$U_\Z^-(\sg)^{\mathrm{up}} \stackrel{\sim}{\to} K(\gmod{R})$
mapping $f_i$ to $[L(i)]$, where $U_\Z^-(\sg)^{\up}$ denotes the $\Z[q,q^{-1}]$-subalgebra of $U_q^-(\sg)$ spanned by $\Bup$,
and $K(\gmod{R})$ the Grothendieck ring of $\gmod{R}$.
This isomorphism induces a bijection between the upper global basis and the set of not necessarily degree preserving isomorphism classes 
of simple modules in $\gmod{R}$ (\cite{varagnolo2011canonical,rouquier2012quiver}).

Given a duality datum $\cD=\{\sL_i\}_{i \in I}$, define a ring homomorphism $\cL_{\cD}$ from $U_\Z^-(\sg)^{\up}$ to the Grothendieck ring $K(\scC_\fg)$ by the composition
 \begin{equation}\label{eq:composition}
   \mathcal{L}_{\cD}\colon U_\Z^- (\sg)^{\up} \stackrel{\sim}{\to} K(\gmod{R}) \to K(\scC_\fg),
 \end{equation}
 where the second one is induced from $\cF_{\cD}$.
By the properties of the isomorphism $U_\Z^- (\sg)^{\up} \stackrel{\sim}{\to} K(\gmod{R})$ stated above and \cite[Corollary 4.14]{kashiwara2020pbw}, we obtain the following lemma
(hereafter, we occasionally identify the isomorphism class of a simple module in $\scC_\fg$
with its class in $K(\scC_\fg)$).

\begin{Lem}\label{Lem:injection_of_irr}
 If $\cD=\{\sL_i\}_{i \in I}$ is a strong duality datum associated with $\sg$, $\cL_{\cD}$
 induces an injection, which we also denote by $\cL_\cD$, from $\Bup$ to the set $\Irr (\scC_{\fg})$ of isomorphism classes of simple modules in $\scC_\fg$.
\end{Lem}

The bicrystal structure on $\Bup$ is described in terms of $U_q'(\fg)$-modules as follows.

\begin{Prop}[{\cite[Lemma 3.2]{kashiwara2022categorical}}]\label{Prop:Kashiwara_Park}
 For a strong duality datum $\cD=\{\sL_i\}_{i \in  I }$, $b \in \Bup$ and $i \in  I $, we have
 \begin{enumerate}\setlength{\leftskip}{-15pt}
  \item[(a)] $\cL_{\cD}(\tilde{e}_i b) \cong \cL_{\cD}(b) \nab \scD \sL_i$ if $\gee_i(b) \neq 0$, $\cL_{\cD}(\tilde{e}^*_i b) \cong \scD^{-1}\sL_i \nab \cL_{\cD}(b)$ if $\gee_i^*(b)\neq 0$,
  \item[(b)] $\cL_{\cD}(\tilde{f}_i b) \cong \sL_i \nab \cL_{\cD}(b)$, $\cL_{\cD}(\tilde{f}_i^*b) \cong \cL_{\cD}(b) \nab \sL_i$,
  \item[(c)] $\gee_i(b) =\tfd\big(\scD \sL_i,\cL_{\cD}(b)\big)$ and $\gee_i^*(b) = \tfd\big(\scD^{-1}\sL_i,\cL_{\cD}(b)\big)$.
 \end{enumerate}
\end{Prop}

Assume that $\cD=\{\sL_i\}_{i \in  I}$ is a strong duality datum associated with $\sg$,
and fix a reduced word $\bm{i} = (i_1,\ldots,i_N) \in R(w_0)$.
For each $1\leq k \leq N$, setting  $\gb_k = s_{i_1}\cdots s_{i_{k-1}}(\ga_{i_k})$, we denote by $S_k = S_k^{\cD,\bm{i}}$ the simple module $\cL_{\cD}\big(F^{\bm{i}}(\gb_k)\big) \in \scC_\fg$
(see  (\ref{eq:dual_PBW})). 
Note that we have
\begin{equation}\label{eq:Li_Sk}
 S_k \cong \sL_i \text{ if and only if } \gb_k = \ga_{i} \ \ \ \text{for $1\leq k \leq N$ and $i \in I$}.
\end{equation}
It follows from the construction that 
\begin{equation}\label{eq:Image_of_F}
  \cL_{\cD}\big(F^{\bm{i}}(\bm{c})\big) = [ S_1^{\otimes c_1} \otimes \cdots \otimes S_N^{\otimes c_N}] \in K(\scC_\fg)
\end{equation}
for any $\bm{c}=(c_1,\ldots,c_N) \in \Z_{\geq 0}^N$.
Moreover, we also have the following.

\begin{Lem}\label{Lem:parametrization}
 For any $\bm{c}=(c_1,\ldots,c_N) \in \Z_{\geq 0}^N$, we have 
 \[ \cL_{\cD}\big(B^{\bm{i}}(\bm{c})\big) \cong \hd(S_1^{\otimes c_1} \otimes \cdots \otimes S_{N}^{\otimes c_N}).
 \] 
\end{Lem}

\begin{proof}
 This follows from (\ref{eq:Image_of_F}) and \cite[Corollary 4.8]{kato2014poincare}.
\end{proof}

Following \cite{kashiwara2020pbw}, we extend the above definition of $S_k = S_k^{\cD,\bm{i}}$ to all $k \in \Z$ by
\begin{equation}\label{eq:N-shift}
 S_{k-N} = \scD S_k \ \ \ \text{for all} \ k \in \Z.
\end{equation}
These modules $S_k$ ($k \in \Z$) are called the \textit{affine cuspidal modules} corresponding to $\cD$ and $\bm{i}$.

\begin{Rem}\normalfont
 Our convention of affine cuspidal modules is different from that of \cite{kashiwara2020pbw}.
 Setting $\bm{i}^\vee= (i_N^*,\ldots,i_1^*) \in R(w_0)$, our $S_k^{\cD,\bm{i}}$ coincides with 
 $``\mathsf{S}_{N+1-k}^{\cD,\bm{i}^\vee}"$ in \cite[Definition 5.6]{kashiwara2020pbw}.
\end{Rem}
\begin{Prop}[{\cite[Propositions 5.7]{kashiwara2020pbw}}]\label{Prop:basics_of_affine_cuspidal_mod}
 The modules $\{S_k\}_{k \in \Z}$ satisfy the following.
 \begin{enumerate}\setlength{\leftskip}{-15pt}
  \item[(i)] $S_k$ is a root module for all $k \in \Z$.
  \item[(ii)] For any $a,b \in \Z$ with $a<b$, the pair $(S_a,S_b)$ is strongly unmixed.
  \item[(iii)] For any increasing sequence $k_1<k_2<\cdots <k_p$ of integers and $(a_1,\ldots,a_p) \in \Z_{> 0}^p$,
   $(S_{k_1}^{\otimes a_1},\ldots, S_{k_p}^{\otimes a_p})$ is a normal sequence
   and $\hd\big(S_{k_1}^{\otimes a_1} \otimes \cdots \otimes S_{k_p}^{\otimes a_p}\big)$ is simple.
 \end{enumerate}
\end{Prop}

We have $(i_1,\ldots,i_N) \in R(w_0)$ if and only if $(i_2,\ldots,i_N,i_1^*) \in R(w_0)$.
The following lemma is proved from \cite[Propositions 5.9 and 5.10]{kashiwara2020pbw}.

\begin{Lem}\label{Lem:change_of_cuspidal}
 Let $\{S_k\}_{k \in \Z}$ be the sequence of affine cuspidal modules corresponding to a strong duality datum $\cD$ associated with $\sg$ and $(i_1,\ldots,i_N) \in R(w_0)$.
 We extend $i_k$ to all $k \in \Z$ by  
 \[ i_{k-N} = i^*_k \ \ \ (k \in \Z).
 \] 
 Fix $a \in \Z$, and set $i_k' = i_{a+k}$ and $S'_k=S_{a+k}$ for all $k \in \Z$.
 Set $\gb_k'=s_{i_1'}\cdots s_{i_{k-1}'}(\ga_{i_k'})$ for $1\leq k \leq N$, and let $k(i) \in [1,N]$ $(i \in I)$ be the unique integer satisfying $\gb_{k(i)}' = \ga_i$.
 Then the family $\cD'=\{S'_{k(i)}\}_{i \in I}$ forms a strong duality datum associated with $\sg$, and $\{S'_k\}_{k \in \Z}$ is the sequence of affine cuspidal modules corresponding to 
 $\cD'$ and $(i_1',\ldots,i_N') \in R(w_0)$.
\end{Lem}

The following lemma is easily proved from the construction and \cite[Subsection 39.2.5]{MR1227098}.

\begin{Lem}\label{Lem:independent_of_2-move}
 Let $\{S_k\}_{k \in \Z}$ be the sequence of affine cuspidal modules corresponding to a strong duality datum $\cD$ and $\bm{ i  }=( i_1,\ldots, i_N) \in R(w_0)$.
 Assume that $a_{i_l i_{l+1}}=0$ for some $1\leq l < N$, and let $\bm{ i }' \in R(w_0)$ be such that $ i_{l}'= i_{l+1}$, $ i_{l+1}'= i_{l}$ and $ i_k' =  i_k$ for $k\neq l,l+1$.
 For $k \in \Z$, set
 \[ S_k' = \begin{cases} S_{k+1} & \text{if $k \equiv l$ mod $N$},\\
                         S_{k-1} & \text{if $k \equiv l+1$ mod $N$},\\
                         S_k     & \text{otherwise.}\end{cases}
 \]
 Then $\{S_k'\}_{k \in \Z}$ is the affine cuspidal modules corresponding to $\cD$ and $\bm{ i }'$.
\end{Lem}

\subsection{simple modules in $\scC_\fg$}\label{Subsection:simple modules}

The notations in this subsection will be used later in some examples for illustrating notions or formulas.
For simplicity, we assume that $\fg$ is of untwisted type only in this subsection.

Let $\mathcal{P}^+=\big(1+u\bk[u]\big)^{I_\fg^0}$ be the abelian monoid (via coordinate-wise multiplication) of $I_\fg^0$-tuples of polynomials with indeterminate $u$ 
and constant term $1$.
By \cite{MR1300632,MR1357195}, the isomorphism classes of simple modules in $\scC_\fg$ are parametrized by $\mathcal{P}^+$.
For $m \in \mathcal{P}^+$, denote by $L(m)$ a simple module belonging to the associated isomorphism class.

Following \cite{MR1745260,MR1810773}, we denote elements of $\mathcal{P}^+$ via monomials.
Throughout the rest of this paper we pick and fix $\ga \in \bk^\times$ once and for all,
and for $i \in I_\fg^0$ and $k \in \Z$ define $Y_{i,k} =\big(Y_{i,k}^j(u)\big)_{j \in I_\fg^0}\in \cP^+$ by
\[ Y_{i,k}^j(u)= \begin{cases} 1-\ga q^k u & j = i,\\ 1 & j \neq i.\end{cases}
\]
All elements of $\cP^+$ appearing below will be expressed as monomials in $\{Y_{i,k}\mid i \in I_\fg^0, k \in \Z\}$.
We call a simple module $L(Y_{i,k})$ a \textit{fundamental module}.
For a sequence $(i_1,k_1),\ldots,(i_p,k_p)$ of elements of $I_\fg^0 \times \Z$, it follows from \cite{MR1883181,MR1890649,varagnolo2002standard} that 
\begin{align}\label{eq:monomial_of_tensor}
  \hd\big(L(Y_{i_1,k_1}) \otimes \cdots \otimes L(Y_{i_p,k_p})\big) \cong L(Y_{i_1,k_1}\cdots Y_{i_p,k_p}) \ \ \text{if} \ \ k_1 \geq \cdots \geq k_p.
\end{align}
We also have 
\begin{equation}\label{eq:dual_of_monomial}
 \scD^{\pm 1} L(Y_{i_1,k_1}\cdots Y_{i_p,k_p}) \cong L(Y_{i_1^*,k_1 \pm h^\vee}\cdots Y_{i_p^*,k_p \pm h^\vee})
\end{equation}
by \cite{MR1376937,MR1810773},
where $h^\vee$ is the dual Coxeter number of the simple Lie subalgebra $\fg_0 \subseteq \fg$ corresponding to $I_\fg^0$.


\section{Primeness, reality, and short exact sequences}\label{Section:Reality,etc}

\textit{Throughout this section, fix a strong duality datum $\cD = \{ \sL_i\}_{i \in I}\subseteq \scC_\fg$ associated with $\sg$,
a reduced word $\bm{i}=(i_1,\ldots,i_N) \in R(w_0)$, and an increasing sequence $\bm{k}=(k_1<k_2<\cdots<k_p)$ of integers}.
Let $S_k = S_k^{\cD,\bm{i}}$ ($k \in \Z$) be the affine cuspidal modules corresponding to $\cD$ and $\bm{i}$.
For $a,b \in [1,p]$ with $a \leq b$, we write
\[ \bS_{\bm{k}}[a,b] = \hd(S_{k_a} \otimes S_{k_{a+1}} \otimes \cdots \otimes S_{k_b}),
\]
which is simple by Proposition \ref{Prop:basics_of_affine_cuspidal_mod}.
Set $\bS_{\bm{k}} = \bS_{\bm{k}}[1,p]$, and $\bS_{\bm{k}}[a,b]=\bm{1}$ if $a>b$.


\subsection{Primeness and reality}

A simple module $M \in \scC_{\fg}$ is said to be \textit{prime} if any tensor decomposition $M \cong N_1 \otimes N_2$ satisfies $N_1 \cong \bm{1}$ or $N_2 \cong \bm{1}$.

\begin{Prop}\label{Prop:primeness}
 Assume that one of the following two conditions is satisfied:
 \begin{itemize}\setlength{\leftskip}{-15pt}
  \item[{\normalfont(i)}] For all $1\leq a < p$, we have $\tfd(S_{k_a}, \bS_{\bm{k}}[a+1,p]) > 0$.
  \item[{\normalfont(ii)}] For all $1< b \leq p$, we have $\tfd(\bS_{\bm{k}}[1,b-1], S_{k_b}) > 0$.
 \end{itemize}
 Then $\bS_{\bm{k}}$ is prime.
\end{Prop}

\begin{proof}
 We will show the assertion by the induction on $p$, assuming (ii) (the case (i) is similarly proved).
 In this proof, we write $k$ for $k_p$.
 First assume that $p=1$ (namely, $\bS_{\bm{k}}=S_k$), and $M,N \in \scC_{\fg}$ satisfies $S_{k} \cong M \otimes N$. 
 Recall that $S_k$ is a root module by Proposition \ref{Prop:basics_of_affine_cuspidal_mod}.
 Since $S_{k}$ is real, so are $M$ and $N$.
 By Proposition \ref{Prop:weakly_decreasing} (ii), we have
 \[ 1 = \tfd(\scD S_{k},S_{k}) = \tfd(\scD N \otimes \scD M, M \otimes N) \geq \tfd(\scD M, M) + \tfd(\scD N, N),
 \]
 which implies that either $\tfd(\scD M, M)=0$ or $\tfd(\scD N, N)=0$ holds. 
 Hence we have $M \cong \bm{1}$ or $N \cong \bm{1}$, as required.

 Let $p >1$, and assume that $\bS_{\bm{k}} \cong M \otimes N$.
 By Lemma \ref{Lem:invariance_by_head} (ii) and Proposition \ref{Prop:basics_of_affine_cuspidal_mod} (ii), we have 
 \begin{equation*}
  \tfd(\scD^{-1} S_k, \bS_{\bm{k}}) = \tfd(\scD^{-1} S_k, S_k)= 1.
 \end{equation*} 
 From this, we see that either $\tfd(\scD^{-1} S_k, M) =0$ or $\tfd(\scD^{-1} S_k, N)=0$ holds,
 and we may assume the former.
 It follows from Proposition \ref{Prop:bijective_of_head} that $\scD^{-1} S_k \nab \bS_{\bm{k}} \cong \bS_{\bm{k}}[1,p-1]$.
 Since $\tfd( \scD^{-1} S_k\nab N,M) = 0$, we have 
 \[ \bS_{\bm{k}}[1,p-1] \cong  \scD^{-1} S_k \nab (M \otimes N) \cong M \otimes (\scD^{-1} S_k\nab N). 
 \]
 Since $\bS_{\bm{k}}[1,p-1]$ is prime by the induction hypothesis, we have $M \cong \bm{1}$ or $\scD^{-1} S_k\nab N \cong \bm{1}$.
 If the latter occurs, we have $N \cong S_k$ and $M \cong \bS_{\bm{k}}[1,p-1]$ by Proposition \ref{Prop:bijective_of_head}, which contradicts the assumption $\tfd(\bS_{\bm{k}}[1,p-1],S_k ) > 0$.
 Hence $M \cong \bm{1}$ holds, and the proof is complete.
\end{proof}

\begin{Lem}\label{Lem:commutativity}
 Let $a,b,c \in [1,p]$ with $a\leq b \leq c$, and assume that 
 \begin{equation}\label{eq:assumption_of_middle}
  \tfd(S_{k_b},\bS_{\bm{k}}[a,b-1]) \leq 1 \ \ \text{and} \ \ \tfd(S_{k_b},\bS_{\bm{k}}[b+1,c]) \leq 1.
 \end{equation}
 Then $S_{k_b}$ and $\bS_{\bm{k}}[a,c]$ strongly commute.
\end{Lem}

\begin{proof}
  Set
  \[ X=\bS_{\bm{k}}[a,b-1], \ \ Y=S_{k_b}, \ \ Z=\bS_{\bm{k}}[b+1,c].
  \]
  By Proposition \ref{Prop:basics_of_affine_cuspidal_mod} (ii), we have $\tfd(\scD X, Y) = 0$ and $\tfd(\scD Y, Z)=0$.
  This and Proposition \ref{Prop:basics_of_affine_cuspidal_mod} (iii) imply that $(X,Y,Z)$ satisfies the assumptions of Proposition \ref{Prop:three_terms},
  and hence we have 
  \[\tfd(S_{k_b},\bS_{\bm{k}}[a,c]) = \tfd\big(Y,\hd(X\otimes Y \otimes Z)\big) = \tfd(Y,X\nab Y) + \tfd(Y,Y \nab Z).
  \]
  By Proposition \ref{Prop:strongly_decreasing}, our assumption (\ref{eq:assumption_of_middle}) implies that
  $\tfd(Y,X\nab Y) = \tfd(Y,Y\nab Z)=0$. 
  Hence we have $\tfd(S_{k_b},\bS_{\bm{k}}[a,c])=0$.
  Now the assertion follows from Proposition \ref{Prop:fundamental_properties_of_delta} (ii).
\end{proof}

\begin{Prop}\label{Prop:reality}
 Let $a,b \in [1,p]$ with $a\leq b$,
 and assume that 
 \begin{equation}
  \tfd(S_{k_c},\bS_{\bm{k}}[a,c-1]) \leq 1 \ \ \text{and} \ \ \tfd(S_{k_c},\bS_{\bm{k}}[c+1,b]) \leq 1 
 \end{equation}
 for all $c \in [a,b]$.
 \begin{enumerate}\setlength{\leftskip}{-15pt}
  \item[(i)] $\bS_{\bm{k}}[a,b]$ is a real simple module.
  \item[(ii)] For any $a',b' \in [a,b]$ with $a' \leq b'$, $\bS_{\bm{k}}[a,b]$ and $\bS_{\bm{k}}[a',b']$ strongly commute.
 \end{enumerate}
\end{Prop}

\begin{proof}
 The assertion (i) is proved inductively using Proposition \ref{Prop:reality_of_MN},
 and then (ii) follows from Lemma \ref{Lem:commutativity} and Propositions \ref{Prop:fundamental_properties_of_delta} (ii) and \ref{Prop:weakly_decreasing} (i). 
\end{proof}




\subsection{A sufficient condition to have a short exact sequence}


\begin{Thm}\label{Thm:Main_thm_general}
 Assume that $p \geq 2$, and both of the following two conditions are satisfied:
 \begin{itemize}\setlength{\leftskip}{-15pt}
  \item[{\normalfont(a)}] for any $1\leq a < b\leq p$, we have $\tfd(S_{k_a}, \bS_{\bm{k}}[a+1,b]) =1$, and
  \item[{\normalfont(b)}] for any $1\leq a < b \leq p$, we have $\tfd(\bS_{\bm{k}}[a,b-1], S_{k_b}) = 1$.
 \end{itemize}
 Then there exists a short exact sequence
 \begin{equation}\label{eq:short_exact}
  0 \to \hd\left(\bigotimes_{a=1}^{p-1} (S_{k_{a}} \Del S_{k_{a+1}})\right) \to \bS_{\bm{k}}[1,p-1] \otimes \bS_{\bm{k}}[2,p] \to \bS_{\bm{k}} \otimes \bS_{\bm{k}}[2,p-1]  \to 0,
 \end{equation}
 where the tensor product in the second term is ordered from left to right.
 Moreover, the first and third terms are simple.
\end{Thm}

Our proof goes along a similar line as the one of \cite[Theorem 4.25]{kashiwara2103monoidal}.
First we note the following lemma, which is a special case of \cite[Theorem 6.12]{kashiwara2020pbw}.
For the reader's convenience, we give a proof here.

\begin{Lem}\label{Lem:expression_of_socle}
 Let $l,m \in \Z$ be such that $l<m$, and assume that $\tfd(S_l,S_m) >0$. Then we have
 \[ S_l \Del S_m \cong \hd\big( S_{l+1}^{\otimes c_{l+1}} \otimes S_{l+2}^{\otimes c_{l+2}}\otimes \cdots \otimes S_{m-1}^{\otimes c_{m-1}}\big) \ \ \ \text{for some } c_{l+1},\ldots,c_{m-1} \in \Z_{\geq 0}.
 \]
\end{Lem}

\begin{proof}
 Since $\tfd(S_l,S_m) > 0$, we have $m\leq l+N$ by (\ref{eq:N-shift}) and Proposition \ref{Prop:basics_of_affine_cuspidal_mod} (ii).
 If the equality holds, we have $S_l \Del S_m \cong \bm{1}$, which implies the assertion with $c_j=0$ for all $j$. 
 Assume that $m <l+N$. 
 Extend $i_k \in I$ (fixed at the beginning of this section) to all $k \in \Z$ by $i_{k-N} = i_k^*$, and set $\bm{i}'=(i_{l},\ldots,i_{l+N-1}) \in R(w_0)$.
 By (\ref{eq:Image_of_F}) and Lemmas \ref{Lem:parametrization} and \ref{Lem:change_of_cuspidal}, 
 there is a strong duality datum $\cD'$ such that the algebra homomorphism $\cL_{\cD'}\colon U_\Z^-(\sg)^{\up} \to K(\scC_\fg)$
 satisfies
 \begin{align*}
  \cL_{\cD'}\big(F^{\bm{i}'}(\bm{d})\big) &= [S_{l}^{\otimes d_l}\otimes \cdots \otimes S_{l+N-1}^{\otimes d_{l+N-1}}], \ \ \text{and}\\
  \cL_{\cD'}\big(B^{\bm{i}'}(\bm{d})\big) &= [\hd(S_{l}^{\otimes d_l}\otimes \cdots \otimes S_{l+N-1}^{\otimes d_{l+N-1}})]
 \end{align*}
 for all $\bm{d}=(d_l,\ldots,d_{l+N-1}) \in \Z_{\geq 0}^N$.
 The assertion is now proved by applying this $\cL_{\cD'}$ to $F^{\bm{i}'}(\bm{d})$, where 
 \[ \bm{d} = (d_l,\ldots,d_{l+N-1}) \ \ \ \text{with} \ \  d_r = \gd_{r,l}+\gd_{r,m} \ \ \ (r \in [l,l+N-1]),
 \]
 and using Proposition \ref{Prop:relations_of_FB}.
\end{proof}

We devote the rest of this section to the proof of Theorem \ref{Thm:Main_thm_general}.
Until the end of the proof, we assume that the sequence $(S_{k_1},\ldots,S_{k_p})$ satisfies the assumptions (a) and (b) of the theorem.
Note that then the first and third terms of (\ref{eq:short_exact}) are both simple by Lemma \ref{Lem:expression_of_socle}, Proposition \ref{Prop:basics_of_affine_cuspidal_mod}, 
and Proposition \ref{Prop:reality}.

\begin{Lem}\label{Lem:for_Main1}\
 \begin{itemize}\setlength{\leftskip}{-15pt}
  \item[{\normalfont(i)}] We have $\tfd(\bS_{\bm{k}}[1,p-1],\bS_{\bm{k}}[2,p]) \leq 1.$
  \item[{\normalfont(ii)}] We have $\bS_{\bm{k}}[1,p-1] \nab \bS_{\bm{k}}[2,p] \cong  \bS_{\bm{k}} \otimes \bS_{\bm{k}}[2,p-1]$. 
 \end{itemize}
\end{Lem}

\begin{proof}
 (i) The assertion holds since
 \[ \tfd(\bS_{\bm{k}}[1,p-1],\bS_{\bm{k}}[2,p]) \leq \tfd(S_{k_1},\bS_{\bm{k}}[2,p])+ \tfd(\bS_{\bm{k}}[2,p-1],\bS_{\bm{k}}[2,p]) =1
 \]
 by Proposition \ref{Prop:reality} and the assumption (a).
 (ii) There are homomorphisms 
 \[ \bS_{\bm{k}}[1,p-1] \otimes \bS_{\bm{k}}[2,p] \hookrightarrow \bS_{\bm{k}}[1,p-1] \otimes S_{k_p} \otimes \bS_{\bm{k}}[2,p-1] \twoheadrightarrow \bS_{\bm{k}} \otimes \bS_{\bm{k}}[2,p-1]
 \]
 by (\ref{eq:head_socle}), and the last term is simple.
 By Proposition \ref{Prop:simple_heads} (i), we obtain a surjection $\bS_{\bm{k}}[1,p-1] \otimes \bS_{\bm{k}}[2,p] \twoheadrightarrow \bS_{\bm{k}} \otimes \bS_{\bm{k}}[2,p-1]$, and the proof is complete. 
\end{proof}

\begin{Lem}\label{Lem:for_Main2}
 We have $\displaystyle \bS_{\bm{k}}[2,p] \nab \bS_{\bm{k}}[1,p-1] \cong \hd\Big(\bigotimes_{a=1}^{p-1} (S_{k_{a}} \Del S_{k_{a+1}})\Big)$.
\end{Lem}

\begin{proof}
 We shall show the assertion by the induction on $p$. The case $p=2$ is obvious.
 Assume that $p>2$.
 In this proof, we write $k$ for $k_p$ and $k^{-}$ for $k_{p-1}$.\\[8pt]
 \textit{Claim 1.} \ We have $\hd (\bS_{\bm{k}}[1,p-2] \otimes S_{k} \otimes S_{k^-}) \cong S_{k} \nab \bS_{\bm{k}}[1,p-1]$.\\[-8pt]

 We apply Lemma \ref{Lem:technical} to $X=\bS_{\bm{k}}[1,p-2]$, $Y=S_k$, and $Z=S_{k^{-}}$.
 Since $\scD X$ and $Z$ strongly commute, $(X,Y,Z)$ is a normal sequence by Proposition \ref{Prop:normality} and the assumption (i) of the lemma follows.
 The assumption (ii) is equivalent to
 \begin{equation}\label{eq:proof_of_claim1}
  \gL(X,Y)-\gL(Y,Z)+\gL(Y,X\nab Z)=0 
 \end{equation}
 by Proposition \ref{Prop:fundamental_properties_of_delta} (i).
 It follows from our assumption (b) that
 \begin{align*}
  2 &= 2 \,\tfd(\bS_{\bm{k}}[1,p-1],S_k) = \gL(X\nab Z, Y) + \gL(Y, X\nab Z), \ \ \ \text{and}\\
  2 &= 2\,\tfd(S_{k^-},S_k) = \gL(Y,Z)+\gL(Z,Y),
 \end{align*}
 from which we have $\big(\gL(X\nab Z,Y)-\gL(Z,Y)\big)-\gL(Y,Z)+\gL(Y,X\nab Z) = 0$.
 Since $(X,Z,Y)$ is a normal sequence, we have $\gL(X \nab Z, Y) = \gL(X,Y)+\gL(Z,Y)$ by Proposition \ref{Prop:normal_implies_simple_head}.
 Hence (\ref{eq:proof_of_claim1}) holds, and Claim 1 follows from Lemma \ref{Lem:technical}.\\[-8pt]

 Therefore, we have the following homomorphisms:
 \begin{align}\label{eq:existence_of_surj}
  \bS_{\bm{k}}[2,p-1] \otimes S_k \otimes \bS_{\bm{k}}[1,p-1] &\twoheadrightarrow \bS_{\bm{k}}[2,p-1] \otimes (S_k \nab \bS_{\bm{k}}[1,p-1])\nonumber\\ 
  & \stackrel{\sim}{\to} \bS_{\bm{k}}[2,p-1] \otimes \hd (\bS_{\bm{k}}[1,p-2] \otimes S_{k} \otimes S_{k^-})\nonumber\\ 
  & \stackrel{\sim}{\to} \bS_{\bm{k}}[2,p-1] \otimes \big(\bS_{\bm{k}}[1,p-2] \nab ( S_{k^-} \Del S_{k})\big).
 \end{align}
 We see from Lemma \ref{Lem:expression_of_socle} that $\scD\, \bS_{\bm{k}}[2,p-1]$ and $S_{k^-} \Del S_k$ strongly commute,
 and thus $(\bS_{\bm{k}}[2,p-1], \bS_{\bm{k}}[1,p-2], S_{k^-} \Del S_{k})$ is a normal sequence by Proposition \ref{Prop:normality}.
 Hence we obtain a surjection 
 \[ \bS_{\bm{k}}[2,p-1] \otimes S_k \otimes \bS_{\bm{k}}[1,p-1] \twoheadrightarrow (\bS_{\bm{k}}[2,p-1] \nab \bS_{\bm{k}}[1,p-2]) \nab (S_{k^-} \Del S_{k}).
 \]
 By the induction hypothesis, we have $\displaystyle \bS_{\bm{k}}[2,p-1] \nab \bS_{\bm{k}}[1,p-2] \cong \hd\Big(\bigotimes_{a=1}^{p-2} (S_{k_a} \Del S_{k_{a+1}})\Big)$,
 and hence we obtain a surjection
 \begin{equation}\label{eq:surjection}
  \bS_{\bm{k}}[2,p-1] \otimes S_k \otimes \bS_{\bm{k}}[1,p-1] \twoheadrightarrow \hd\Big(\bigotimes_{a=1}^{p-1} (S_{k_a} \Del S_{k_{a+1}})\Big).
 \end{equation}
 Since $\tfd(\bS_{\bm{k}}[2,p-1],S_k) =1$, we have a short exact sequence
 \begin{equation}\label{eq:short_exact_seq}
  0 \to S_k \nab \bS_{\bm{k}}[2,p-1] \to \bS_{\bm{k}}[2,p-1] \otimes S_k \to \bS_{\bm{k}}[2,p] \to 0
 \end{equation}
 by Proposition \ref{Prop:length2}.\\[8pt]
 \textit{Claim 2.}  $\bS_{\bm{k}}[1,p-1]$ and $S_k \nab \bS_{\bm{k}}[2,p-1]$ strongly commute.\\[-8pt]

 In the proof of this claim, set $X=\bS_{\bm{k}}[1,p-1]$, $Y=S_k$ and $Z=\bS_{\bm{k}}[2,p-1]$.
 We have 
 \[ \tfd(X, Y \nab Z) \leq \tfd(X,Y) + \tfd(X,Z) = 1 + 0 =1,
 \]
 and hence it suffices to show that $\tfd(X, Y \nab Z) \neq \tfd(X,Y) + \tfd(X,Z)$. 
 If the equality holds, then Lemma \ref{Lem:two_normal_sequences} implies that $(X,Y,Z)$ is a normal sequence.
 Hence by Proposition \ref{Prop:normal_implies_simple_head}, we have 
 \[ \gL(X\nab Y,Z)=\gL(X,Z)+\gL(Y,Z).
 \]
 Moreover, $(Z,X,Y)$ is also a normal sequence since $\scD Z$ and $Y$ strongly commute, and hence we have 
 \[ \gL(Z, X \nab Y) = \gL(Z,X) + \gL(Z,Y).
 \] 
 On the other hand, since $\tfd(X \nab Y, Z) = 0$ and $\tfd (X,Z) = 0$, it follows from Proposition \ref{Prop:fundamental_properties_of_delta} (i) that
 \[ \gL(X \nab Y, Z) = -\gL(Z, X\nab Y) \ \ \text{and} \ \ \gL(X,Z) = -\gL(Z,X).
 \]
 Now by combining them, we have 
 \begin{align*}
  0 &= \gL(X\nab Y,Z) -\gL(X,Z)-\gL(Y,Z)= - \gL(Z, X \nab Y)+\gL(Z,X) - \gL(Y,Z) \\
    &= -\gL(Z,X) -\gL(Z,Y) +\gL(Z,X)- \gL(Y,Z) \\
    &= -\gL(Z,Y)-\gL(Y,Z) = -2\,\tfd(S_k,\bS_{\bm{k}}[2,p-1]),
 \end{align*}
 which contradicts the assumption (b). The proof is complete.\\[-8pt]

 Write $\displaystyle \mathbb{H} = \hd\Big(\bigotimes_{a=1}^{p-1} (S_{k_a} \Del S_{k_{a+1}})\Big)$.
 The simple modules $(S_k \nab \bS_{\bm{k}}[2,p-1])\otimes \bS_{\bm{k}}[1,p-1]$ and $\mathbb{H}$ are not isomorphic.
 Indeed, this follows since we have
 \[ \tfd\big(\scD S_{k_1}, (S_k \nab \bS_{\bm{k}}[2,p-1])\otimes \bS_{\bm{k}}[1,p-1]\big)\geq \tfd(\scD S_{k_1}, \bS_{\bm{k}}[1,p-1]\big) =1
 \] 
 by Lemma \ref{Lem:invariance_by_head}, and, on the other hand, 
 $\tfd(\scD S_{k_1}, \mathbb{H})=0$ by Lemma \ref{Lem:expression_of_socle}.
 Hence the composition 
 \[ (S_k \nab \bS_{\bm{k}}[2,p-1]) \otimes \bS_{\bm{k}}[1,p-1] \hookrightarrow \bS_{\bm{k}}[2,p-1] \otimes S_k \otimes \bS_{\bm{k}}[1,p-1] \twoheadrightarrow \mathbb{H}
 \]
 vanishes, where the second homomorphism is (\ref{eq:surjection}).
 Hence we obtain a surjection $\bS_{\bm{k}}[2,p] \otimes \bS_{\bm{k}}[1,p-1] \twoheadrightarrow \mathbb{H}$ by (\ref{eq:short_exact_seq}),
 which completes the proof of the lemma. 
\end{proof}

\noindent\textit{Proof of Theorem \ref{Thm:Main_thm_general}.} By Proposition \ref{Prop:length2} and Lemmas \ref{Lem:for_Main1} and \ref{Lem:for_Main2}, 
it suffices to show that $\bS_{\bm{k}} \otimes \bS_{\bm{k}}[2,p-1]$ and $\mathbb{H}$ are not isomorphic,
which follows from
\[ \tfd(\scD S_{k_1}, \bS_{\bm{k}} \otimes \bS_{\bm{k}}[2,p-1]) =1 \ \ \text{and} \ \ \tfd(\scD S_{k_1}, \mathbb{H}) = 0.
\]
\qed



\section{Snake modules associated with a strong duality datum of type $A$}\label{Section:Snake_modules}


\subsection{Quivers and reduced words}\label{Subsection:quivers}

\textit{In the remainder of this paper, we assume that $\sg=\mathfrak{sl}_{n+1}$, namely of type $A_n$}, whose index set is $I =[1,n]$ and the Dynkin diagram $\gD$ is given by
\[ \xymatrix@C=10pt@R=5pt@M=0pt{
 \circ \ar@{-}[r] & \circ \ar@{-}[r] & \circ \ar@{-}[r] & \ar@{.}[r] & \ar@{-}[r]& \circ  \ar@{-}[r] &\circ \\
 {\ \ \ 1 \ \ }& {\ \ \ 2 \ \ } &{\ \ \ 3 \ \ } & & & n-1 & \ n \
}
\]
We have 
\[ N=\frac{n(n+1)}{2},
\]
and $ i^* = n+1 -  i$ for $ i \in I $.
For $i,j \in I$ with $i\leq j$, we write $\ga_{i,j}=\ga_i+\ga_{i+1}+\cdots +\ga_{j} \in R^+$. 
If $i>j$, we set $\ga_{i,j}=0$.
We still assume that $\fg$ is an arbitrary affine Lie algebra.

\begin{Def}\normalfont \
 \begin{itemize}\setlength{\leftskip}{-15pt}
  \item[{\normalfont(i)}] A \textit{height function} (or \textit{untwisted height function}) on $I$ is a function $\xi\colon I \to \Z$ satisfying
   \[ |\xi_i-\xi_{i+1}| = 1 \ \ \text{for $1\leq i <n$},
   \]
   where we set $\xi_i=\xi(i)$ for simplicity. We denote by $\mathrm{HF}$ the set of height functions.\\
  \item[{\normalfont(ii)}] Assume that $n=2n_0-1$ for some $n_0 \in \Z_{\geq 2}$. 
   A \textit{twisted height function} on $I$ is a function $\xi\colon I \to \frac{1}{2}\Z$ satisfying
   \begin{gather*}
    \xi_i \in \Z \ \ \text{for $i \in I \setminus \{n_0\}$}, \ \ \ |\xi_i-\xi_{i+1}| =1 \ \ \text{for $i \in I \setminus\{n_0-1,n_0,n\}$}, \\
    |\xi_{n_0-1} - \xi_{n_0+1}| = 1 \ \ \ \text{and} \ \ \ |\xi_{n_0} - \min(\xi_{n_0-1},\xi_{n_0+1})| = 1/2.
   \end{gather*}
   We denote by $\mathrm{HF}^{\mathrm{tw}}$ the set of twisted height functions.
 \end{itemize} 
\end{Def}

\begin{Rem}\normalfont
 In \cite{zbMATH07355935}, a notion of a height function on a pair of a Dynkin diagram and a diagram automorphism was defined.
 In this terminology, (i) is a height function on $(\gD, \mathrm{id})$, and (ii) is that on $(\gD,( \ )^*)$, up to conventions.
\end{Rem}

Here are examples of untwisted and twisted height functions when $n=5$, where the numbers are the values of each function:
\begin{equation}\label{eq:Example_of_height}
 \xymatrix@C=10pt@R=5pt@M=0pt{
 \ 2 \ & \ 1 \ & \ 2 \ & \ 3 \ & \ 4 \  \\
 \circ \ar@{-}[r] & \circ \ar@{-}[r] & \circ \ar@{-}[r] & \circ  \ar@{-}[r] &\circ } \ \ \ \ \ \ 
   \ \ \ \xymatrix@C=10pt@R=5pt@M=0pt{
  -1 \ & \ 0 \ & -1/2 & \ 1 \ & \ 2 \ \\
 \circ \ar@{-}[r] & \circ \ar@{-}[r] & \circ \ar@{-}[r] & \circ  \ar@{-}[r] &\circ 
}
\end{equation}

Following  \cite{zbMATH07355935}, we will define several notions associated with an untwisted or a twisted height function.
Let $\xi \in \HF^{\flat}$ with $\flat \in \{\emptyset, \tw\}$, where $\HF^\emptyset := \HF$.
We say $i \in I$ is a \textit{sink} (resp.\ \textit{source}) of $\xi$ if $\xi_i < \xi_j$ (resp.\ $\xi_i-d_i > \xi_j-d_j$) for all $j \in I$ such that $|i-j|=1$,
where we set
\[ d_i = 2 \ \ \text{for all $i \in I$} \ \ \ \text{if $\xi \in \HF$}, \ \ \text{and} \ \ d_i = \begin{cases} 2 & \text{for $i \in I \setminus\{n_0\}$}\\  1 & \text{for $i =n_0$}\end{cases}
   \ \ \ \text{if $\xi \in \HFt$}.
\]
If $i \in I$ is a sink (resp.\ source) of $\xi$, we define a new function $s_i\xi \in \HF^\flat$ by
\[ (s_i\xi)_j = \xi_j +d_i\gd_{i,j} \ \ \ (\text{resp.} \ (s_i\xi)_j = \xi_j-d_i\gd_{i,j}) \ \ \ \text{for all $j \in I$}.
\]
We say a sequence $(i_1,\ldots,i_r)$ of elements of $I$ is \textit{adapted} (or \textit{sink-adapted}) to $\xi$ if $i_k$ is a sink of $s_{i_{k-1}}\cdots s_{i_1}\xi$ for all $k \in [1,r]$.
The \textit{repetition quiver} $\wh{Q}^\xi$ associated with $\xi$ is a quiver whose vertex set $\wh{Q}^\xi_0$ and arrow set $\wh{Q}^\xi_1$ are given respectively by
\begin{align*}
 \wh{Q}^{\xi}_0 &= \{(i,k) \in I \times \frac{1}{2}\Z \mid k -\xi_i \in d_i\Z\},\\
 \wh{Q}^{\xi}_1 &= \{(i,k) \to (j,l) \mid (i,k),(j,l) \in \wh{Q}^\xi_0, |i-j|=1,\ l-k = \min(d_i,d_j)/2\}.
\end{align*} 
For example, $\wh{Q}^\xi$ for $\xi$ in (\ref{eq:Example_of_height}) are, respectively, as follows:
$$
\raisebox{3mm}{
\scalebox{0.9}{\xymatrix@!C=0.5mm@R=2mm{
(i \setminus k) && -2 &-1& 0 & 1& 2 & 3& 4 & 5& 6 & \\
1&&\bullet \ar@{->}[dr]&& \bullet \ar@{->}[dr] &&\bullet\ar@{->}[dr]
&&\bullet \ar@{->}[dr] && \bullet&\\
2&   &&\bullet \ar@{->}[dr]\ar@{->}[ur]&& \bullet \ar@{->}[dr]\ar@{->}[ur] &&\bullet \ar@{->}[dr]\ar@{->}[ur]
&& \bullet\ar@{->}[ur]\ar@{->}[dr] & &  \\ 
3& \ \ \cdots \cdots \ \ &\bullet \ar@{->}[dr] \ar@{->}[ur]&& \bullet \ar@{->}[dr] \ar@{->}[ur] &&\bullet\ar@{->}[dr] \ar@{->}[ur]
&&\bullet \ar@{->}[dr] \ar@{->}[ur] && \bullet& \cdots \cdots\\
4&&& \bullet \ar@{->}[ur]\ar@{->}[dr]&&\bullet \ar@{->}[ur]\ar@{->}[dr]&&\bullet \ar@{->}[ur]\ar@{->}[dr] &&\bullet \ar@{->}[ur]\ar@{->}[dr]& & \\
5&&\bullet \ar@{->}[ur]&& \bullet \ar@{->}[ur] && \bullet \ar@{->}[ur]
&&\bullet \ar@{->}[ur] && \bullet&}}}
$$
$$\raisebox{3.1em}{\scalebox{0.65}{\xymatrix@!C=0.1ex@R=0.5ex{
\mbox{{\Large $(i\setminus k)$}}  &  & \mbox{{\Large $-3$}} &&\mbox{{\Large $-2$}} && \mbox{{\Large $-1$}} && \mbox{{\Large $0$}} & &\mbox{{\Large $1$}} & & \mbox{{\Large $2$}}&  
& \mbox{{\Large $3$}} &  & \mbox{{\Large $4$}} &  & \mbox{{\Large $5$}} &  \\
\mbox{{\Large $1$}} &&\mbox{{\Large $\bullet$}} \ar@{->}[ddrr]&&&& \mbox{{\Large $\bullet$}} \ar@{->}[ddrr]&&&&  \mbox{{\Large $\bullet$}} \ar@{->}[ddrr] &&&& \mbox{{\Large $\bullet$}}\ar@{->}[ddrr]
&&&& \mbox{{\Large $\bullet$}}& \\{\phantom 2} \\
\mbox{{\Large $2$}} &&&&\mbox{{\Large $\bullet$}}\ar@{->}[dr] \ar@{->}[uurr] &&&& \mbox{{\Large $\bullet$}} \ar@{->}[dr]\ar@{->}[uurr] &&&&  \mbox{{\Large $\bullet$}} \ar@{->}[dr] \ar@{->}[uurr]
&&&& \mbox{{\Large $\bullet$}} \ar@{->}[dr] \ar@{->}[uurr] & \\
\mbox{{\Large $3$}} & \mbox{{\Large\ \ $\cdots \cdots$}}&& \mbox{{\Large $\bullet$}} \ar@{->}[ur] &&\mbox{{\Large $\bullet$}} \ar@{->}[dr] && \mbox{{\Large $\bullet$}} \ar@{->}[ur] 
&& \mbox{{\Large $\bullet$}} \ar@{->}[dr]
&& \mbox{{\Large $\bullet$}} \ar@{->}[ur] && \mbox{{\Large $\bullet$}} \ar@{->}[dr] & &
\mbox{{\Large $\bullet$}} \ar@{->}[ur]&&\mbox{{\Large $\bullet$}} \ar@{->}[dr] & & \mbox{\Large $\cdots\cdots$ \ \ }\\
\mbox{{\Large $4$}}&&\mbox{{\Large $\bullet$}}\ar@{->}[ddrr]\ar@{->}[ur]&&&&\mbox{{\Large $\bullet$}}\ar@{->}[ddrr]\ar@{->}[ur] &&&&  \mbox{{\Large $\bullet$}} \ar@{->}[ddrr]\ar@{->}[ur] &&&&
\mbox{{\Large $\bullet$}} \ar@{->}[ddrr] \ar@{->}[ur]
&&&& \mbox{{\Large $\bullet$}} & \\  {\phantom 2} & \\
\mbox{{\Large $5$}} &&&&\mbox{{\Large $\bullet$}} \ar@{->}[uurr] &&&& \mbox{{\Large $\bullet$}} \ar@{->}[uurr]
&&&& \mbox{{\Large $\bullet$}} \ar@{->}[uurr]  &&&& \mbox{{\Large $\bullet$}} \ar@{->}[uurr]& }}}
$$
For $\xi, \xi' \in \HF^{\flat}$ with $\flat\in \{\emptyset,\tw\}$, we have $\wh{Q}^\xi = \wh{Q}^{\xi'}$ if $\xi_i-\xi_i' \in 2\Z$ for some (or any) $i \in I$ such that $d_i=2$.
Let $\scD_\xi \colon \wh{Q}^\xi \to \wh{Q}^\xi$ denote the unique quiver automorphism satisfying 
\[ \scD_\xi(i,k)=(i^*,k-\wti{n}) \ \ \ \text{for all $(i,k) \in \wh{Q}^\xi_0$}, \ \ \ \text{where we set} \ \ \wti{n} = \begin{cases} n+1 & \text{if $\xi \in \HF$}, \\ n & \text{if $\xi \in \HFt$}.\end{cases}
\]
We also write $\scD$ for $\scD_\xi$ when $\xi$ is obvious.
Define a partial order $\preceq$ on $\wh{Q}^\xi_0$ by $(i,k) \preceq (j,l)$ if and only if there is an oriented path from $(i,k)$ to $(j,l)$ in $\wh{Q}^\xi$.

For $\xi \in \HF^\flat$ with $\flat \in \{\emptyset,\tw\}$, let $\gG^{\xi}$ denote the full subquiver of $\wh{Q}^\xi$ whose vertex set $\gG^\xi_0$ is given by
\[ \gG^\xi_0 = \{ (i,k) \in \wh{Q}^\xi_0 \mid \xi_i \leq k \leq n-1+\xi_{i^*}\}.
\]
We easily see that the number of the vertices of $\gG^\xi$ is $N$.
We say a total ordering $\gG^\xi_0=\{(i_1,k_1),\ldots,(i_N,k_N)\}$ of vertices of $\gG^\xi$ is a \textit{compatible reading} of $\gG^\xi$ 
if $r<s$ holds whenever there is an arrow $(i_r,k_r) \to (i_s,k_s)$ in $\gG^\xi$.

\begin{Prop}[{\cite[Section 3]{zbMATH07355935}}] \label{Prop:Bedard}
 Assume that $\xi \in \HF^{\flat}$ with $\flat \in \{\emptyset, \tw\}$.
 \begin{itemize}\setlength{\leftskip}{-15pt}
  \item[{\normalfont(i)}] The set of reduced words of $w_0$ adapted to $\xi$ forms a single commutation class in $R(w_0)$.
  \item[{\normalfont(ii)}] If $\{(i_1,k_1),\ldots,(i_N,k_N)\}$ is a compatible reading of $\gG^\xi$, then $(i_1,\ldots,i_N)$ is a reduced word of $w_0$ adapted to $\xi$.
   Conversely, any reduced word of $w_0$ adapted to $\xi$ is obtained from a compatible reading of $\gG^\xi$ in this way.
 \end{itemize}
\end{Prop}

Let $\xi \in \HF^{\flat}$ with $\flat \in \{\emptyset, \tw\}$, and take a compatible reading $\{(i_1,k_1),\ldots,(i_N,k_N)\}$ of $\gG^\xi$.
Define a bijection $\phi_\xi\colon \gG^\xi_0 \to R^+$ by 
\[ \phi_{\xi}(i_l,k_l) = \gb_l \ \ \text{with} \ \ \gb_l = s_{i_1}\cdots s_{i_{l-1}}(\ga_{i_l}) \ \ \text{for $l \in [1,N]$}.
\]
Here and below we write $\phi_{\xi}(i,k)$ instead of $\phi_{\xi}\big((i,k)\big)$ for simplicity.
The map $\phi_{\xi}$ does not depend on the choice of the compatible reading by Proposition \ref{Prop:Bedard}.
If we take $\xi$ as in (\ref{eq:Example_of_height}), $\gG^\xi$ and $\phi_\xi$ are given, respectively, as follows:
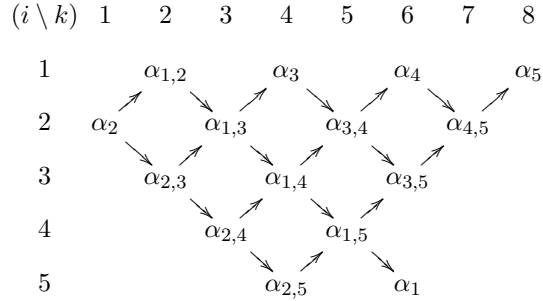
\begin{equation*} 
\raisebox{3mm}{
\scalebox{0.9}{\xymatrix@!C=0.5mm@R=2mm{
(i \setminus k) &1& 2 &3& 4 & 5& 6 & 7 & 8 \\
1 & & \ga_{1,2} \ar@{->}[dr] & & \ga_{3} \ar@{->}[dr] & & \ga_4\ar@{->}[dr] &&\ga_5&&\\
2  &\ga_2 \ar@{->}[dr] \ar@{->}[ur] && \ga_{1,3} \ar@{->}[dr]\ar@{->}[ur] && \ga_{3,4} \ar@{->}[dr] \ar@{->}[ur] && \ga_{4,5}  \ar@{->}[ur] & & \\ 
3&& \ga_{2,3} \ar@{->}[dr] \ar@{->}[ur] &&\ga_{1,4}\ar@{->}[dr] \ar@{->}[ur]
&& \ga_{3,5} \ar@{->}[ur]&& & \\
4& &&\ga_{2,4} \ar@{->}[ur]\ar@{->}[dr]&&\ga_{1,5} \ar@{->}[ur] \ar@{->}[dr]&&& & \\
5 & && & \ga_{2,5}\ar@{->}[ur] & & \ga_1 & &  }}}
\end{equation*}
\begin{equation*}
 \raisebox{3.1em}{\scalebox{0.8}
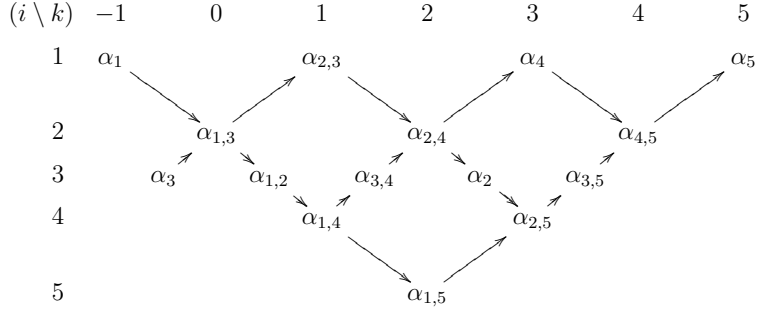
{\xymatrix@!C=0.2ex@R=0.5ex{
\mbox{\large$(i\setminus k)$ \ \ \ }    & \mbox{\large$-1$} &&\mbox{\large$ 0$} && \mbox{\large$1$} && \mbox{\large$2$} & &\mbox{\large$3$} & & \mbox{\large$4$}&  
& \mbox{\large$5$} & \\
\mbox{\large$1$}&\mbox{\large$\ga_1$} \ar@{->}[ddrr]&&&& \mbox{\large$\ga_{2,3}$} \ar@{->}[ddrr]&&&&  \mbox{\large$\ga_{4}$} \ar@{->}[ddrr] &&&& \mbox{\large$\ga_5$} \\{\phantom 2} \\
\mbox{\large$2$} &&&\mbox{\large$\ga_{1,3}$}\ar@{->}[dr] \ar@{->}[uurr] &&&& \mbox{\large$\ga_{2,4}$} \ar@{->}[dr]\ar@{->}[uurr] &&&&  \mbox{\large$\ga_{4,5}$} \ar@{->}[uurr]
&&&&  & \\
\mbox{\large$3$}  && \mbox{\large$\ga_3$} \ar@{->}[ur] &&\mbox{\large$\ga_{1,2}$} \ar@{->}[dr] && \mbox{\large$\ga_{3,4}$} \ar@{->}[ur] && \mbox{\large$\ga_2$} \ar@{->}[dr]
&& \mbox{\large$\ga_{3,5}$} \ar@{->}[ur] &&   \\
\mbox{\large$4$}&&&&&\mbox{\large$\ga_{1,4}$}\ar@{->}[ddrr]\ar@{->}[ur] &&&&  \mbox{\large$\ga_{2,5}$} \ar@{->}[ur] &&&&
 \\  {\phantom 2} & \\
\mbox{\large$5$} &&& &&&& \mbox{\large$\ga_{1,5}$} \ar@{->}[uurr] &&&&  }}}
\end{equation*}
If $\xi$ is an untwisted height function, $\gG^\xi$ is isomorphic to the Auslander-Reiten quiver of the category of finite-dimensional modules 
over the path algebra of type $A_n$ associated with $\xi$  (see \cite{zbMATH03695413}).


\subsection{Snake modules}\label{Subsection:Snake_modules}

Let $\cD=\{\sL_i \}_{ i  \in I} \subseteq \scC_{\fg}$ be a strong duality datum associated with $\mathfrak{sl}_{n+1}$,
and assume that $\xi \in \HF^{\flat}$ with $\flat \in \{\emptyset,\tw\}$.
Choose an arbitrary compatible reading $\gG^\xi_0=\{(i_1,k_1),\ldots,(i_N,k_N)\}$, and set $\bm{i}=(i_1,\ldots,i_N) \in R(w_0)$.
Let $S_k^{\cD,\bm{ i}}$ ($k \in \Z$) be the corresponding affine cuspidal modules.
These modules are labeled by $\Z$, and we shall relabel them by $\wh{Q}_0^\xi$ as follows; set
\[ S_{i_r,k_r}^{\cD,\xi} := S_r^{\cD,\bm{i}} \ \ \text{for $r \in [1,N]$},
\]
and extend this to all $(i,k) \in \wh{Q}^\xi_0$ by $S_{\scD_\xi(i,k)}^{\cD,\xi} =\scD S_{i,k}^{\cD,\xi}$.
It follows that 
\begin{equation}\label{eq:correspondence_of_S}
  S_{\scD_\xi^t(i_r,k_r)}^{\cD,\xi} = S_{r-tN}^{\cD,\bm{i}} \ \ \ \text{for $1\leq r \leq N$, $t\in \Z$}.
\end{equation} 
It is easily seen from Lemma \ref{Lem:independent_of_2-move} that these $S_{i,k}^{\cD,\xi}$ do not depend on the choice of the compatible reading (though $S_k^{\cD,\bm{i}}$ do).
We will write $S_{i,k}$ or $S_{i,k}^{\xi}$ for $S_{i,k}^{\cD,\xi}$ when no confusion is likely.

\begin{Rem}\label{Rem:fundamental_module_case}\normalfont
 Assume that $\xi \in \HF$ (resp.\ $\xi \in \HFt$) and $\fg$ is of type $A_n^{(1)}$ (resp.\ $B_{n_0}^{(1)}$).
 For each $i \in I$, set 
 \[ \sL_i = L(Y_{j,-l}) \ \ \ \text{(resp.\ $\sL_i=L(Y_{\hat{j},-2l})$ \ with \ $\hat{j} = \min(j,j^*)$)}
 \]
 (see Subsection \ref{Subsection:simple modules} for notation), where we put $(j,l) = \phi_\xi^{-1}(\ga_i) \in \gG^\xi_0$.
 Then $\cD = \{\sL_i\}_{i \in I}$ forms a strong duality datum associated with $\mathfrak{sl}_{n+1}$,
 and we have 
 \[ S_{i,k}^{\cD,\xi} \cong L(Y_{i,-k}) \ \ \ (\text{resp.\ } S_{i,k}^{\cD,\xi} \cong L(Y_{\hat{i},-2k})) \ \ \ \text{for all $(i,k) \in \wh{Q}^\xi_0$},
 \]
 see \cite{hernandez2015quantum,kang2015symmetric,zbMATH07355935,kashiwara2020pbw} (note that source-adapted reduced words are used by convention in these papers).
\end{Rem}


\begin{Lem}\label{Lem:change_of_height_function}
 Let $\cD=\{\sL_i\}_{i \in I}$ be a strong duality datum associated with $\mathfrak{sl}_{n+1}$, and assume that $\xi,\xi' \in \HF^{\flat}$ with $\flat \in \{\emptyset,\tw\}$.
 \begin{itemize}\setlength{\leftskip}{-15pt}
  \item[{\normalfont(i)}] If $\xi' - \xi$ is a constant function whose value is $p \in \Z$, then we have $S_{i,k}^{\cD,\xi'} \cong S_{i,k-p}^{\cD,\xi}$ for all $(i,k) \in \wh{Q}^{\xi'}_0$.
  \item[{\normalfont(ii)}] Assume that $\wh{Q}^\xi = \wh{Q}^{\xi'}$, and set $\sL_i' = S_{\phi_{\xi'}^{-1}(\ga_i)}^{\cD,\xi}$ for $i\in I$.
   Then $\cD'=\{\sL'_i\}_{i \in I}$ forms a strong duality datum associated with $\mathfrak{sl}_{n+1}$, and we have
   \begin{equation}\label{eq:SikcD}
    S_{i,k}^{\cD,\xi} \cong S_{i,k}^{\cD',\xi'} \ \ \ \text{for all $(i,k) \in \wh{Q}^\xi_0$}.
   \end{equation}
 \end{itemize}
\end{Lem}

\begin{proof}
 (i) This is obvious from the construction.
 (ii) It is easily seen that there is a sequence $j_1,j_2,\ldots,j_r$ of elements of $I$ such that $j_t$ is a sink or a source of $s_{j_{t-1}}\cdots s_{j_{1}}\xi$ for all $1\leq t \leq r$,
 and $s_{j_r}\cdots s_{j_1}\xi = \xi'$.
 Hence it suffices to show the assertion for $\xi'=s_j\xi$, with $j$ being a sink or a source of $\xi$.
 We show this for a sink $j$ (the other case is proved similarly).
 Take a compatible reading $\{(i_1,k_1),\ldots,(i_N,k_N)\}$ of $\gG^\xi$ such that $(i_1,k_1) = (j,\xi_j)$,
 and set  
 \[ (i_r',k_r')=(i_{r+1},k_{r+1}) \ \ \text{for $1\leq r <N$} \ \ \  and \ \ \ (i_N',k_N') = \scD_\xi^{-1}(i_1,k_1).
 \]
 We easily check that $\{(i_1',k_1'),\ldots,(i_N',k_N')\}$ is a compatible reading of $\gG^{\xi'}$.
 Set $\bm{i}=(i_1,\ldots,i_N)$, $\bm{i}'=(i_1',\ldots,i_N')$, and $S_k'= S_{k+1}^{\cD,\bm{i}}$ for all $k\in \Z$.
 For $i \in I$ and $r \in [1,N]$, 
 \begin{align*}
  s_{i_1'}\cdots s_{i_{r-1}'}(\ga_{i_r'}) = \ga_i \ \ \text{holds if and only if} \ \ S_{\phi_{\xi'}^{-1}(\ga_i)}^{\cD,\xi} = S_{i_r',k_r'}^{\cD,\xi} = S_r'.
 \end{align*}
 Therefore by Lemma \ref{Lem:change_of_cuspidal}, $\cD'= \big\{ S_{\phi_{\xi'}^{-1}(\ga_i)}^{\cD,\xi}\big\}_{i \in I} $ 
 forms a strong duality datum, and $\{S_k'\}_{k \in \Z}$ are the affine cuspidal modules corresponding to $\cD'$ and $\bm{i}'$.
 Now (\ref{eq:SikcD}) for $(i,k) \in \gG^\xi_0$ is proved by
 \begin{align*}
  S_{i_1,k_1}^{\cD,\xi} &\cong S_1^{\cD,\bm{i}} \cong S_0' \cong \scD S_N' \cong \scD S_{i_N',k_N'}^{\cD',\xi'}\cong S_{i_1,k_1}^{\cD',\xi'}, \ \text{and} \\  
  S_{i_r,k_r}^{\cD,\xi} &\cong S_r^{\cD,\bm{i}} \cong S_{r-1}' \cong S_{i_{r-1}',k_{r-1}'}^{\cD',\xi'}\cong S_{i_r,k_r}^{\cD',\xi'}  \ \ \text{for $r \in [2,N]$}.
 \end{align*}
 Then (\ref{eq:SikcD}) for general $(i,k) \in \wh{Q}^\xi_0$ also follows from the construction.
\end{proof}

\begin{Lem}\label{Lem:fundamentals_of_T}
 Let $\cD=\{\sL_i\}$ be a strong duality datum associated with $\mathfrak{sl}_{n+1}$, and $\xi$ a height function or a twisted height function.
 \begin{itemize}\setlength{\leftskip}{-15pt}
  \item[{\normalfont(a)}] Each $S_{i,k}$ is a root module.
  \item[{\normalfont(b)}] Let $(i,k),(i',k') \in \wh{Q}_0^\xi$ be such that $(i,k) \not\succeq (i',k')$.
   Then the pair $(S_{i,k}, S_{i',k'})$ is strongly unmixed.
  \item[{\normalfont(c)}]  Let $(i_1,k_1),\ldots,(i_p,k_p)$ be a sequence of elements of $\wh{Q}_0^\xi$, and assume that $(i_r,k_{r}) \not\succeq (i_s,k_s)$ for all $r,s \in [1,p]$ such that $r<s$.
   Then for any sequence $a_1,\ldots,a_p$ of positive integers, the head of $S_{i_1,k_1}^{\otimes a_1} \otimes \cdots \otimes S_{i_p,k_p}^{\otimes a_p}$ is simple.
  \item[{\normalfont(d)}] If $(i,k),(i',k') \in \wh{Q}_0^\xi$ are incomparable, then $S_{i,k}$ and $S_{i',k'}$ strongly commute.
 \end{itemize}
\end{Lem}

\begin{proof}
 (a) This follows from Proposition \ref{Prop:basics_of_affine_cuspidal_mod} (i).
 (b) Let $r,r'$ be the integers satisfying
  \[ S_{i,k} = S_r^{\cD,\bm{i}} \ \ \ \text{and} \ \ \ S_{i',k'} = S_{r'}^{\cD,\bm{i}},
  \]
 where $\bm{i} \in R(w_0)$ is an element coming from a compatible reading of $\gG^\xi$.
 By replacing $(\cD,\xi)$ and $\bm{i}$ using Lemma \ref{Lem:change_of_height_function} (ii), if necessary,
 we may assume $r<r'$.
 Then the assertion follows from Proposition \ref{Prop:basics_of_affine_cuspidal_mod} (ii).
 (c) This follows from (b) and Proposition \ref{Prop:unmixed_normal}. 
 (d) As above, we may assume that $S_{i,k}=S_{r}^{\cD,\bm{i}}$ and $S_{i',k'} = S_{r+1}^{\cD,\bm{i}}$ for some $\bm{i}$ and $r \in \Z$.
  If $S_{i,k}$ and $S_{i',k'}$ do not strongly commute, it follows from Lemma \ref{Lem:expression_of_socle} that $S_{i,k} \Del S_{i',k'} \cong \bm{1}$,
  which implies $S_{i,k} \cong \scD S_{i',k'}$, or equivalently $(i',k') = \scD^{-1}(i,k)$.
  This obviously contradicts the assumption, and the assertion is proved.
\end{proof}

Let $\cD$ be a strong duality datum associated with $\mathfrak{sl}_{n+1}$ and $\xi$ an untwisted or a twisted height function.
For a sequence $\bm{P} = \big((i_1,k_1),\ldots,(i_p,k_p)\big)$ of elements of $\wh{Q}^{\xi}_0$ satisfying $(i_r,k_r) \not\succeq (i_s,k_s)$ for all $1\leq r < s \leq p$,
set
\[ \bS^{\cD,\xi}(\bm{P})=\hd(S_{i_1,k_1}^{\cD,\xi} \otimes \cdots \otimes S_{i_p,k_p}^{\cD,\xi}),
\]
which is simple by Lemma \ref{Lem:fundamentals_of_T}.
We often write $\bS(\bm{P})$ for $\bS^{\cD,\xi}(\bm{P})$.
 

When $\xi$ is a twisted height function, we define four subsets $\wh{Q}^{\xi,\gtrless}_0$, $\wh{Q}^{\xi,\mathrm{D}}_0$ and $\wh{Q}^{\xi,\mathrm{U}}_0$ of $\wh{Q}^{\xi}_0$ by
\begin{gather}
 \wh{Q}^{\xi,\gtrless}_0 = \{ (i,k) \in \wh{Q}_0^{\xi} \mid i \gtrless n_0\}, \ \wh{Q}^{\xi,\mathrm{D}}_0 = \Big\{ (n_0,k) \in \wh{Q}^{\xi}_0 \Bigm| (n_0,k) \to 
   \Big(n_0+1,k+\frac{1}{2}\Big)  \in \wh{Q}^{\xi}_1\Big\},\nonumber\\
 \wh{Q}^{\xi,\mathrm{U}}_0 = \Big\{ (n_0,k) \in \wh{Q}^{\xi}_0 \Bigm| (n_0,k) \to 
   \Big(n_0-1,k+\frac{1}{2}\Big)  \in \wh{Q}^{\xi}_1\Big\}.\label{eq:def_of_UD}
\end{gather}
Here ``$\mathrm{D}$" (resp.\ ``$\mathrm{U}$") stands for ``downward"  (resp.\ ``upward"),  the direction of all the arrows incident to the vertices belonging to 
the subset.
When $n=5$, these are illustrated as follows, where $\circ$ (resp.\ $\heartsuit$, $\spadesuit$, $\bullet$) belong to $\wh{Q}^{\xi,<}_0$ (resp.\ $\wh{Q}^{\xi,\mathrm{U}}_0$,
$\wh{Q}^{\xi,\mathrm{D}}_0$, $\wh{Q}^{\xi,>}_0$):

$$\raisebox{3.1em}{\scalebox{0.6}
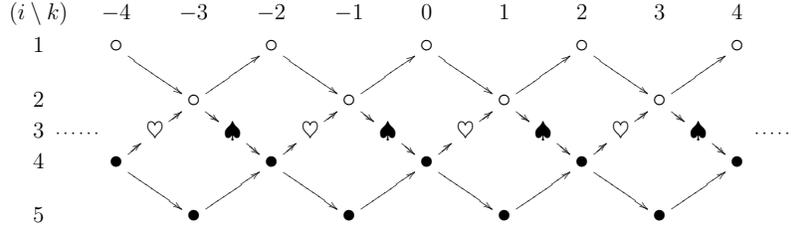
{\xymatrix@!C=0.1ex@R=0.5ex{
\mbox{\Large $(i\setminus k)$}  &  & \mbox{\Large $-4$} &&\mbox{\Large $-3$} && \mbox{\Large $-2$} && \mbox{\Large $-1$} & &\mbox{\Large $0$} & & \mbox{\Large $1$}&  
& \mbox{\Large $2$} &  & \mbox{\Large $3$} &  & \mbox{\Large $4$} &  \\
\mbox{\Large $1$}&&\mbox{\LARGE $\circ$} \ar@{->}[ddrr]&&&& \mbox{\LARGE $\circ$} \ar@{->}[ddrr]&&&&  \mbox{\LARGE $\circ$} \ar@{->}[ddrr] &&&& \mbox{\LARGE $\circ$}\ar@{->}[ddrr]
&&&& \mbox{\LARGE$\circ$}& \\{\phantom 2} \\
\mbox{\Large $2$} &&&&\mbox{\LARGE $\circ$}\ar@{->}[dr] \ar@{->}[uurr] &&&& \mbox{\LARGE $\circ$} \ar@{->}[dr]\ar@{->}[uurr] &&&&  \mbox{\LARGE $\circ$} \ar@{->}[dr] \ar@{->}[uurr]
&&&& \mbox{\LARGE $\circ$} \ar@{->}[dr] \ar@{->}[uurr] & \\
\mbox{\Large $3$} & \cdots \cdots&& \mbox{\Large$\heartsuit$} \ar@{->}[ur] &&\mbox{\Large$\spadesuit$} \ar@{->}[dr] && \mbox{\Large$\heartsuit$} \ar@{->}[ur] && \mbox{\Large$\spadesuit$} \ar@{->}[dr]
&& \mbox{\Large$\heartsuit$} \ar@{->}[ur] && \mbox{\Large$\spadesuit$} \ar@{->}[dr] & &
\mbox{\Large$\heartsuit$} \ar@{->}[ur]&&\mbox{\Large$\spadesuit$} \ar@{->}[dr] & & \cdots\cdots\\
\mbox{\Large $4$}&&\mbox{\LARGE$\bullet$}\ar@{->}[ddrr]\ar@{->}[ur]&&&&\mbox{\LARGE$\bullet$}\ar@{->}[ddrr]\ar@{->}[ur] &&&&  \mbox{\LARGE$\bullet$} \ar@{->}[ddrr]\ar@{->}[ur] &&&&
\mbox{\LARGE$\bullet$}\ar@{->}[ddrr] \ar@{->}[ur]
&&&& \mbox{\LARGE$\bullet$} & \\  {\phantom 2} & \\
\mbox{\Large $5$} &&&&\mbox{\LARGE$\bullet$} \ar@{->}[uurr] &&&& \mbox{\LARGE$\bullet$} \ar@{->}[uurr]
&&&& \mbox{\LARGE$\bullet$} \ar@{->}[uurr]  &&&& \mbox{\LARGE$\bullet$} \ar@{->}[uurr]& }}}
$$
Note that $\scD\colon \wh{Q}^\xi \to \wh{Q}^\xi$ maps $\wh{Q}_0^{\xi,<}$ (resp.\ $\wh{Q}_0^{\xi,\mrD}$) bijectively onto $\wh{Q}_0^{\xi,>}$ (resp.\ $\wh{Q}_0^{\xi,\mrU}$), and vice versa.

\begin{Def}[\cite{zbMATH06084095}]\label{Def:snake position}\normalfont \ 
 \begin{itemize}\setlength{\leftskip}{-15pt}
  \item[(1)] Assume that $\xi$ is a height function, and $(i,k), (i',k') \in \wh{Q}_0^\xi$.
  \begin{itemize}\setlength{\leftskip}{-28pt}
   \item[(i)] We say $(i',k')$ is \textit{in snake position} with respect to $(i,k)$ if $(i,k+2) \preceq (i',k')$.
   \item[(ii)] We say $(i',k')$ is \textit{in prime snake position} with respect to $(i,k)$ if 
    \[ (i,k+2) \preceq (i',k') \preceq \scD^{-1}(i,k).
    \]
  \end{itemize}
  \item[(2)] Assume that $\xi$ is a twisted height function, and $(i,k),(i',k') \in \wh{Q}_0^{\xi}$.
   \begin{itemize}\setlength{\leftskip}{-28pt}
   \item[(i)] We say $(i',k')$ is \textit{in snake position} with respect to $(i,k)$ if $(i,k+2-\gd_{i,n_0}) \preceq (i',k')$,
    \begin{align*}
     (i',k') &\in \wh{Q}^{\xi, <} \sqcup \wh{Q}^{\xi,\mathrm{D}}\ \ \  \text{when $(i,k) \in \wh{Q}^{\xi,<} \sqcup \wh{Q}^{\xi,\mrU}$}, \ \ \ \text{and} \\   
     (i',k') &\in \wh{Q}^{\xi,>}  \sqcup \wh{Q}^{\xi,\mrU}      \ \ \  \text{when $(i,k) \in \wh{Q}^{\xi,>} \sqcup \wh{Q}^{\xi,\mrD}$}.
    \end{align*}    
   \item[(ii)] We say $(i',k')$ is \textit{in prime snake position} with respect to $(i,k)$ if $(i',k')$ is in snake position with respect to $(i,k)$, and 
    $(i',k') \preceq \scD^{-1}(i,k)$.
   \end{itemize}
 \end{itemize}
 When we would like to emphasize that $\xi$ is untwisted (resp.\ twisted), we say $(i',k')$ is in snake position \textit{of untwisted type} (resp.\ \textit{of twisted type}) with respect to $(i,k)$.
\end{Def}

When $n=5$ these are illustrated as follows, where $\bullet$ and $\circ$ are in snake position with respect to $\star$, 
and $\bullet$ are in prime snake position with respect to $\star$.
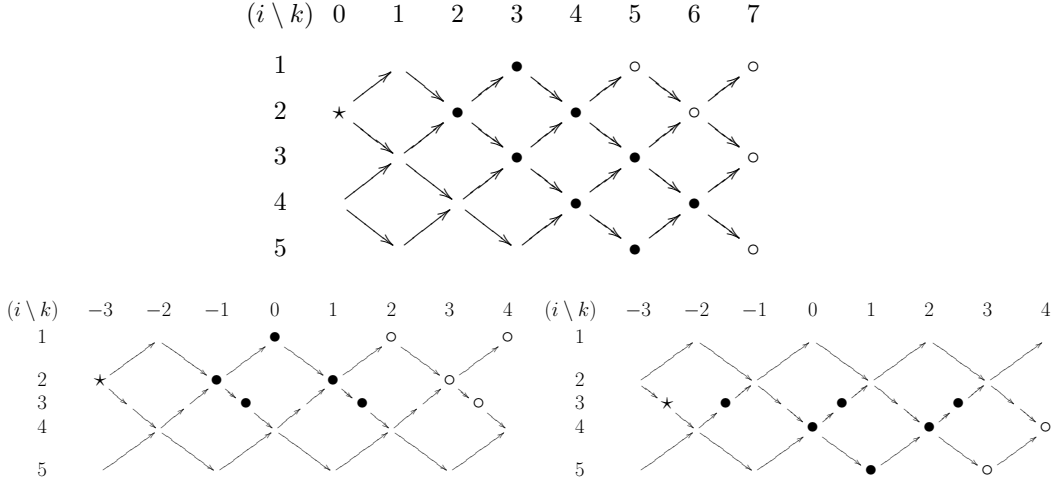
\begin{gather*}
\scalebox{0.9}{\xymatrix@!C=0.3mm@R=2mm{
(i \setminus k) &0 & 1& 2 & 3& 4&  5 & 6 & 7  \\
1& & \ar@{->}[rd]&  &\bullet\ar@{->}[rd]&  &\circ \ar@{->}[rd]&  & \circ  \\
2& \star \ar@{->}[ru] \ar@{->}[rd] && \bullet \ar@{->}[ru] \ar@{->}[rd] & & \bullet \ar@{->}[ru] \ar@{->}[rd]& &\circ \ar@{->}[ru] \ar@{->}[rd] &  \\
3& & \ar@{->}[ru] \ar@{->}[rd] && \bullet \ar@{->}[ru] \ar@{->}[rd] & & \bullet\ar@{->}[ru] \ar@{->}[rd] && \circ &  \\
4& \ar@{->}[ru] \ar@{->}[rd] & & \ar@{->}[ru] \ar@{->}[rd] && \bullet\ar@{->}[ru] \ar@{->}[rd] &  & \bullet\ar@{->}[ru] \ar@{->}[rd]& & \\
5& &\ar@{->}[ru]&  &\ar@{->}[ru]&  &\bullet\ar@{->}[ru] & & \circ \\
}}
\end{gather*}
\begin{gather*}
\raisebox{3.1em}{\scalebox{0.45}{\xymatrix@!C=0.1ex@R=0.5ex{
\mbox{\LARGE$(i\setminus k)$\ \ \ }  &&\mbox{\LARGE$-3$} && \mbox{\LARGE$-2$} && \mbox{\LARGE$-1$} & &\mbox{\LARGE$0$} & & \mbox{\LARGE$1$}&  
& \mbox{\LARGE$2$} &  & \mbox{\LARGE$3$} &  & \mbox{\LARGE$4$} &  \\
\mbox{\LARGE$1$}&&&& \ar@{->}[ddrr]&&&&  \mbox{\huge $\bullet$} \ar@{->}[ddrr] &&&& \mbox{\huge $\circ$}\ar@{->}[ddrr]
&&&& \mbox{\huge $\circ$}& \\{\phantom 2} \\
\mbox{\LARGE$2$} && \mbox{\huge $\star$} \ar@{->}[dr] \ar@{->}[uurr] &&&& \mbox{\huge $\bullet$} \ar@{->}[dr]\ar@{->}[uurr] &&&& \mbox{\huge $\bullet$} \ar@{->}[dr] \ar@{->}[uurr]
&&&& \mbox{\huge$\circ$} \ar@{->}[dr] \ar@{->}[uurr] & \\
\mbox{\LARGE$3$}  & && \ar@{->}[dr] &&  \ar@{->}[ur] &&\mbox{\huge $\bullet$} \ar@{->}[dr]
&&  \ar@{->}[ur] && \mbox{\huge$\bullet$} \ar@{->}[dr] & &
 \ar@{->}[ur]&&\mbox{\huge $\circ$} \ar@{->}[dr] & & \\
\mbox{\LARGE$4$}&&&&\ar@{->}[ddrr]\ar@{->}[ur] &&&&   \ar@{->}[ddrr]\ar@{->}[ur] &&&&
 \ar@{->}[ddrr] \ar@{->}[ur]
&&&&  & \\  {\phantom 2} & \\
\mbox{\LARGE$5$}&& \ar@{->}[uurr] &&&&  \ar@{->}[uurr]
&&&&  \ar@{->}[uurr]  &&&&  \ar@{->}[uurr]& }}} 
\raisebox{3.1em}{\scalebox{0.45}{\xymatrix@!C=0.1ex@R=0.5ex{
\mbox{\LARGE$(i\setminus k)$\ \ \ } &&\mbox{\LARGE$-3$} && \mbox{\LARGE$-2$} && \mbox{\LARGE$-1$} & &\mbox{\LARGE$0$} & & \mbox{\LARGE$1$}&  
& \mbox{\LARGE$2$} &  & \mbox{\LARGE$3$} &  & \mbox{\LARGE$4$}   \\
\mbox{\LARGE$1$}&&&& \ar@{->}[ddrr]&&&&   \ar@{->}[ddrr] &&&& \ar@{->}[ddrr]
&&&&  \\{\phantom 2} \\
\mbox{\LARGE$2$} &&  \ar@{->}[dr] \ar@{->}[uurr] &&&&  \ar@{->}[dr]\ar@{->}[uurr] &&&&  \ar@{->}[dr] \ar@{->}[uurr]
&&&&  \ar@{->}[dr] \ar@{->}[uurr]  \\
\mbox{\LARGE$3$} & && \mbox{\huge $\star$}\ar@{->}[dr] && \mbox{\huge $\bullet$} \ar@{->}[ur] && \ar@{->}[dr]
&& \mbox{\huge$\bullet$} \ar@{->}[ur] && \ar@{->}[dr] & &
\mbox{\huge$\bullet$} \ar@{->}[ur]&& \ar@{->}[dr] & \\
\mbox{\LARGE$4$}&&&&\ar@{->}[ddrr]\ar@{->}[ur] &&&& \mbox{\huge $\bullet$}  \ar@{->}[ddrr]\ar@{->}[ur] &&&&
 \mbox{\huge $\bullet$}\ar@{->}[ddrr] \ar@{->}[ur]
&&&& \mbox{\huge $\circ$}  \\  {\phantom 2} & \\
\mbox{\LARGE$5$} && \ar@{->}[uurr] &&&&  \ar@{->}[uurr]
&&&& \mbox{\huge $\bullet$} \ar@{->}[uurr]  &&&& \mbox{\huge $\circ$} \ar@{->}[uurr] }}}
\end{gather*}
\begin{Rem}\normalfont
 The definition of prime snake position for $\xi \in \HF$ can be rephrased in terms of the denominators of normalized $R$-matrices between fundamental modules of type $A_n^{(1)}$
 as follows: for $(i,k),(i',k') \in \wh{Q}_0^\xi$, $(i',k')$ is in prime snake position with respect to $(i,k)$ if and only if $d_{L(Y_{i',-k'}),L(Y_{i,-k})}(1)=0$. 
 This follows from the denominator formula in \cite{date1994calculation}.

 Similar assertion does not hold for $\xi \in \HF^{\tw}$. Assume that $\fg$ is of type $B_{n_0}^{(1)}$.
 For $(i,k), (i',k') \in \wh{Q}_0^\xi$, if $(i',k')$ is in prime snake position with respect to $(i,k)$, then $d_{L(Y_{\hat{i}',-2k'}),L(Y_{\hat{i},-2k})}(1)=0$ holds
 (see Remark \ref{Rem:fundamental_module_case} for the notation), 
 but the converse is not true (see \cite{oh2015denominators}).
\end{Rem}
The following definition (with $(\cD,\xi)$ taken as in Remark \ref{Rem:fundamental_module_case}) was introduced in \cite{zbMATH06084095}.

\begin{Def}\normalfont \label{Def:snake_module}
 Let $\xi$ be a height function or a twisted height function, and $\bm{P} = \big((i_1,k_1),\ldots,(i_p,k_p)\big)$ a sequence of elements of $\wh{Q}^\xi_0$.
 \begin{itemize}\setlength{\leftskip}{-15pt}
  \item[{\normalfont(i)}] We say $\bm{P}$ is a \textit{snake} (resp.\ \textit{prime snake}) if $(i_{s+1},k_{s+1})$ is in snake position (resp.\ in prime snake position) 
   with respect to $(i_{s},k_{s})$ for all $1 \leq  s < p$.
   We also say $\bm{P}$ is a snake \textit{of untwisted} or \textit{twisted type}, when we would like to emphasize the type of $\xi$.
   For a subset $\Omega$ of $\wh{Q}^\xi_0$, we say $\bm{P}$ is a snake in $\gO$ if $\bm{P}$ is a snake and all the elements of $\bm{P}$ belong to $\gO$.
  \item[{\normalfont(ii)}]
   Let $\cD$ be a strong duality datum associated with $\mathfrak{sl}_{n+1}$.
   If $\bm{P}$ is a snake, we call $\bS(\bm{P})=\bS^{\cD,\xi}(\bm{P})$ a \textit{snake module} (\textit{of untwisted} or \textit{twisted type}) associated with $\cD$ and $\xi$.
 \end{itemize}
\end{Def}

\begin{Lem}\label{Lem:Dualizing_label}
 If $\bm{P}=\big((i_1,k_1),\ldots,(i_p,k_p)\big)$ is a snake,
 then so are 
 \[ \scD^{\pm 1}\bm{P}=\big(\scD^{\pm 1}(i_1,k_1),\ldots,\scD^{\pm 1}(i_p,k_p)\big),
 \]
 and we have $\scD^{\pm 1} \bS(\bm{P}) \cong \bS(\scD^{\pm 1} \bm{P})$.
 Moreover, if $\bm{P}$ is prime then so are $\scD^{\pm 1} \bm{P}$.
\end{Lem}

\begin{proof}
 The first and last assertions are easily checked from the definition, and the second follows from Proposition \ref{Prop:normal_implies_simple_head}, since
 \begin{align*}
  \scD^{\pm 1} \bS(\bm{P}) &\cong \soc( S_{\scD^{\pm 1}(i_p,k_p)} \otimes \cdots \otimes S_{\scD^{\pm 1}(i_1,k_1)})\\
                           &\cong \hd (S_{\scD^{\pm 1}(i_1,k_1)} \otimes \cdots \otimes S_{\scD^{\pm 1}(i_p,k_p)}) = \bS(\scD^{\pm 1} \bm{P}).
 \end{align*}
\end{proof}
For a sequence $\bm{P} = \big((i_1,k_1),\ldots,(i_p,k_p)\big)$ of elements of $\wh{Q}^\xi_0$ and $a,b \in [1,p]$ with $a\leq b$, set $\bm{P}_{[a,b]} 
= \big((i_a,k_a),(i_{a+1},k_{a+1}),\ldots,(i_b,k_b)\big)$.

\begin{Lem}\label{Lem:for_prime}\
 \begin{itemize}\setlength{\leftskip}{-15pt}
 \item[{\normalfont(i)}] If $(i,k),(i',k') \in \wh{Q}^\xi_0$ satisfy $\scD^{-1}(i,k) \not\succeq (i',k')$,
  then $\scD^{\ell}S_{i,k}$ and $S_{i',k'}$ strongly commute for all $\ell \in \Z_{\geq 0}$.
 \item[{\normalfont(ii)}] Let $\bm{P}=\big((i_1,k_1),\ldots,(i_p,k_p)\big) \in (\wh{Q}^\xi_0)^p$ be a snake.
  If $(i_{r+1},k_{r+1})$ is not in prime snake position  with respect to $(i_r,k_r)$ for some $r$, then we have
  \[ \bS(\bm{P}) \cong \bS(\bm{P}_{[1,r]}) \otimes \bS(\bm{P}_{[r+1,p]}).
  \]
 \end{itemize}
\end{Lem}

\begin{proof}
 Since $\scD^{\ell+1} S_{\scD^{-1}(i,k)} \cong \scD^{\ell}S_{i,k}$, the assertion (i) follows from Lemma \ref{Lem:fundamentals_of_T} (b).
  The assertion (ii) easily follows from (i).
\end{proof}

Later, we will prove that $\bS(\bm{P})$ is prime if $\bm{P}$ is a prime snake (Theorems \ref{Thm:MainA} (ii) and \ref{Thm:Main_twisted} (ii)).


\section{The case of untwisted height functions}\label{Section:Untwisted}


\subsection{Reineke's algorithm}\label{Subsection:Reineke}

In this section, we focus on (untwisted) height functions.
We will prove that, 
if $\xi$ is a height function and $\bm{P}$ is a prime snake in $\wh{Q}^\xi_0$, the module $\bS(\bm{P})$ satisfies the assumptions of Theorem \ref{Thm:Main_thm_general}, 
using Proposition \ref{Prop:Kashiwara_Park} (c).
For that, we need to calculate $\gee_i(b)$ and $\gee_i^*(b)$ for some $b \in \Bup$.
An algorithm for this was introduced by Reineke in \cite{zbMATH01031515}, which we recall in this subsection.

For $i \in I$, let $\bar{i} \in \{0,1\}$ be such that $\bar{i} \equiv i$ mod $2$. 
For $\gd \in \{0,1\}$, let $\xi^{(\gd)} \in \HF$ denote the unique height function satisfying $\xi^{(\gd)}_i \in \{0,1\}$ for all $i \in I$ and $\xi^{(\gd)}_1 = \gd$.
We have $\xi^{(\bar{i})}_i = 1$ for all $i \in I$.
We write $\gG^{(\gd)}$ for $\gG^{\xi^{(\gd)}}$, and $\phi_{(\gd)}$ for $\phi_{\xi^{(\gd)}}\colon \gG^{(\gd)}_0 \to R^+$.
For each $i \in I$, define a subset $\gO_i \subseteq \gG^{(\bar{i})}_0$ by
\[ \gO_i =\{(j,k) \in \gG^{\left(\bar{i}\right)}_0 \mid (i,1) \preceq (j,k) \preceq (i^*,n)\} = \{\phi_{(\bar{i})}^{-1}(\ga_{j,l})\mid 1\leq j\leq i \leq l \leq n\}.
\]
When $n=5$, $\gG^{(\gd)}$ and $\phi_{(\gd)}$ are as follows:
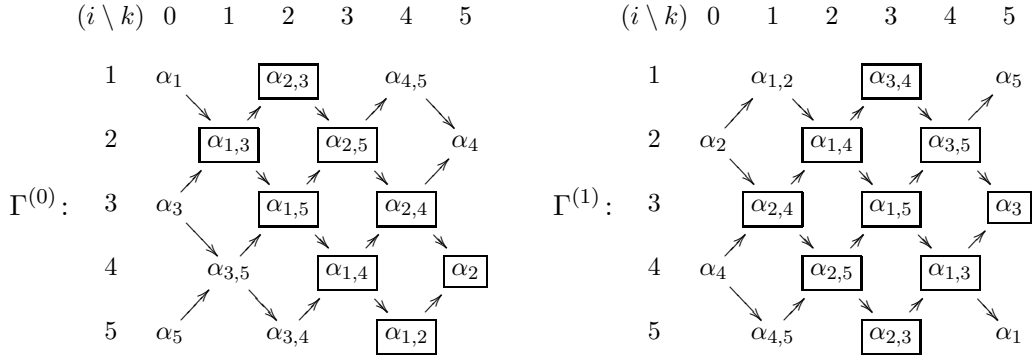
\begin{equation*}
\gG^{(0)}\colon \!\!\!\!\raisebox{75pt}{
\scalebox{0.9}{\xymatrix@!C=0.3mm@R=2mm{
(i \setminus k) &0 & 1& 2 & 3& 4&  5 \\
1& \ga_1 \ar@{->}[rd]&  &\fbox{$\ga_{2,3}$}\ar@{->}[rd]&  &\ga_{4,5} \ar@{->}[rd]&   \\
2 && \fbox{$\ga_{1,3}$} \ar@{->}[ru] \ar@{->}[rd] & & \fbox{$\ga_{2,5}$} \ar@{->}[ru] \ar@{->}[rd]&  & \ga_4 \\
3 & \ga_3\ar@{->}[ru] \ar@{->}[rd] && \fbox{$\ga_{1,5}$} \ar@{->}[ru] \ar@{->}[rd] & & \fbox{$\ga_{2,4}$} \ar@{->}[ru] \ar@{->}[rd]&  \\
4 & & \ga_{3,5} \ar@{->}[ru] \ar@{->}[rd] && \fbox{$\ga_{1,4}$}\ar@{->}[ru] \ar@{->}[rd] &  &\fbox{$\ga_2$}  \\
5 &\ga_5 \ar@{->}[ru]&  &\ga_{3,4} \ar@{->}[ru]&  &\fbox{$\ga_{1,2}$} \ar@{->}[ru] &  \\
}}}\ \ \ \ \ \ \gG^{(1)}\colon \!\!\!\!\raisebox{75pt}{
\scalebox{0.9}{\xymatrix@!C=0.3mm@R=2mm{
(i \setminus k) &0 & 1& 2 & 3& 4&  5 \\
1& & \ga_{1,2} \ar@{->}[rd]&  &\fbox{$\ga_{3,4}$}\ar@{->}[rd]&  &\ga_5  \\
2& \ga_2 \ar@{->}[ru] \ar@{->}[rd] && \fbox{$\ga_{1,4}$} \ar@{->}[ru] \ar@{->}[rd] & & \fbox{$\ga_{3,5}$} \ar@{->}[ru] \ar@{->}[rd]\\
3& & \fbox{$\ga_{2,4}$}\ar@{->}[ru] \ar@{->}[rd] && \fbox{$\ga_{1,5}$} \ar@{->}[ru] \ar@{->}[rd] & & \fbox{$\ga_3$}  \\
4& \ga_{4} \ar@{->}[ru] \ar@{->}[rd] & & \fbox{$\ga_{2,5}$} \ar@{->}[ru] \ar@{->}[rd] && \fbox{$\ga_{1,3}$}\ar@{->}[ru] \ar@{->}[rd] &  \\
5& &\ga_{4,5} \ar@{->}[ru]&  &\fbox{$\ga_{2,3}$} \ar@{->}[ru]&  &\ga_1  \\
}}}
\end{equation*}
Here the vertices belonging to $\gO_2$ in $\gG^{(0)}_0$ and $\gO_3$ in $\gG^{(1)}_0$ are boxed.
The proof of the following lemma is straightforward.

\begin{Lem}\label{Lem:Image_of_phi}
 Let $\gd \in \{0,1\}$.
 For $(i,k) \in \gG^{(\gd)}_0$, we have $\phi_{(\gd)}(i,k) = \ga_{x,y}$, where 
 \[ x= \begin{cases} i-k & \text{if} \ i-k > 0,\\ 
                           k-i+1   & \text{if} \ i-k \leq 0,\\\end{cases} \ \ \text{and} \ \ 
    y = \begin{cases} i+k & \text{if} \ i+k \leq n,\\ 
                           2n+1-i-k   & \text{if} \ i+k > n.\\\end{cases}
 \]
\end{Lem}

Let $\gd \in \{0,1\}$.
Take a compatible reading $\{(i_1,k_1),\ldots,(i_N,k_N)\}$ of $\gG^{(\gd)}$, and set $\bm{i} = (i_1,\ldots,i_N) \in R(w_0)$.
For a $\gG^{(\gd)}_0$-tuple $\bm{c}=(c_{i,k})_{(i,k) \in \gG^{(\gd)}_0}$ of nonnegative integers, we set
\begin{equation}\label{eq:natural_identification}
 B^{(\gd)}(\bm{c}) := B^{\bm{i}}(\bm{c})\in \Bup,
\end{equation}
where in the right-hand side $\bm{c}$ is regarded as an element of $\Z_{\geq 0}^N$ via the bijection $[1,N] \to \gG_0^{(\gd)}\colon r \mapsto (i_r,k_r)$.
We easily see from Propositions \ref{Prop:2-move_3-move} (i) and \ref{Prop:Bedard} that $B^{(\gd)}(\bm{c})$ does not depend on the choice of the compatible reading.
If $\bar{i} \neq \gd$, it follows from (\ref{eq:simple_case}) that $\gee_i(B^{(\gd)}(\bm{c})) = c_{i,0}$.

The other case is described as follows.
For $i \in I$, let $\mathcal{U}_i$ be the set of lower closed subsets of $\gO_i$.
That is, a subset $\Sigma \subseteq \gO_i$ belongs to $\mathcal{U}_i$ if and only if 
for any $P,Q \in \gO_i$, $P \in \Sigma$ and $Q \preceq P$ imply $Q \in \Sigma$.

\begin{Thm}[{\cite[Theorem 7.1]{zbMATH01031515}}]\label{Thm:Reineke}
 Let $j \in I$, and set $\gd=\bar{j}\in \{0,1\}$. 
 For any $\bm{c} =(c_{i,k}) \in \Z_{\geq 0}^{\gG^{(\gd)}_0}$, we have 
   \[ \gee_j\big(B^{(\gd)}(\bm{c})\big) = \max_{\gS \in \,\mathcal{U}_j} \Big(\sum_{ (i,k) \in \gS}(c_{i,k} -c_{i,k-2})\Big),
   \]
   where we set $c_{i,k} = 0$ if $k<0$.
\end{Thm}

For $\gd \in \{0,1\}$, let $\gd^\vee =\ol{\gd+n} \in \{0,1\}$.
For $\bm{c} \in \Z_{\geq 0}^{\gG^{(\gd)}_0}$, define $\bm{c}^\vee \in \Z_{\geq 0}^{\gG^{(\gd^\vee)}_0}$ by $\bm{c}^\vee=(c_{i^*,n-k})_{(i,k) \in \gG^{(\gd^\vee)}_0}$.
By Lemma \ref{Lem:duality_of_B} we have $*B^{(\gd)}(\bm{c}) = B^{(\gd^\vee)}(\bm{c}^\vee)$, and hence 
\begin{equation}\label{eq:duality_of_gee}
  \gee_j^*\big(B^{(\gd)}(\bm{c})\big) = \gee_j\big(B^{(\gd^\vee)}(\bm{c}^\vee)\big)
\end{equation}
holds.
Using this, we can also calculate the values of $\gee_i^*$'s.


\subsection{Snake modules associated with height functions}

Fix a strong duality datum $\cD=\{\sL_i\}_{i \in I}$ associated with $\mathfrak{sl}_{n+1}$, and a height function $\xi$.
We write $\wh{Q}$ for $\wh{Q}^\xi$, and $S_{i,k}$ for $S_{i,k}^{\cD,\xi}$ ($(i,k) \in \wh{Q}_0$).

\begin{Lem}\label{Lem:calculation_of_std}
 Assume that a sequence $\bm{P}=\big((i_1,k_1),\ldots,(i_p,k_p)\big)$ of elements of $\wh{Q}_0$ is a snake, and $(j,l) \in \wh{Q}_0$.
 \begin{itemize}\setlength{\leftskip}{-15pt}
  \item[{\normalfont(1)}] Suppose that $(j,l) \prec (i_1,k_1)$.
  \begin{itemize}\setlength{\leftskip}{-25pt}
   \item[{\normalfont(a)}] If $(i_1,k_1)$ is in prime snake position with respect to $(j,l)$, we have $\tfd\big(S_{j,l},\bS(\bm{P})\big)=1$.
   \item[{\normalfont(b)}] If $(i_1,k_1)$ is not in snake position with respect to $(j,l)$ {\normalfont(}that is, $i_1\in\{j \pm r\}$ and $k_1=l+r$ hold for some $r \in \Z_{>0}${\normalfont)},
    then we have $\tfd\big(S_{j,l},\bS(\bm{P})\big)=0$.
  \end{itemize}
  \item[{\normalfont(2)}] Suppose that $(i_p,k_p)\prec (j,l)$.
  \begin{itemize}\setlength{\leftskip}{-25pt}
   \item[{\normalfont(c)}] If $(j,l)$ is in prime snake position with respect to $(i_p,k_p)$, we have $\tfd\big(\bS(\bm{P}),S_{j,l}\big)=1$.
   \item[{\normalfont(d)}] If $(j,l)$ is not in snake position with respect to $(i_p,k_p)$,
    we have $\tfd\big(\bS(\bm{P}),S_{j,l}\big)=0$.
  \end{itemize}
 \end{itemize}
\end{Lem}

\begin{proof}
 We will use the notations in the previous subsection freely.
 (1) By replacing the pair $(\cD,\xi)$ using Lemma \ref{Lem:change_of_height_function}, we may assume that 
 \[ l=-1,
 \]
 and $(\cD,\xi)$ is the following specific one:
 \[ \xi = \xi^{(\gd)} \ \ \ \text{and} \ \ \ \cD = \{\sL_i\}_{i \in I} \ \ \text{with} \ \ 
    \sL_i = \begin{cases} S_{i,0} & \text{if $\bar{i} \neq \gd$} \\ 
                          S_{i^*,n} & \text{if $\bar{i} = \gd$}\end{cases} \ \ \text{for $i \in I$},
 \]
 where we set $\gd = \bar{j} \in \{0,1\}$.
 For any $(i,k) \in \wh{Q}_0$ satisfying $\scD^{-1}(j,l) \not\succeq (i,k)$, we have $\tfd(S_{j,l}, \scD^{-r} S_{i,k}) = 0$ for all $r\geq 0$ by Lemma \ref{Lem:for_prime} (i).
 Hence by Lemma \ref{Lem:invariance_by_head} (i), we may further assume that $(i_p,k_p) \preceq \scD^{-1}(j,l)$,
 which implies  $(i_s,k_s) \in \gG^{(\gd)}_0$ for all $s \in [1,p]$.
 We have $S_{j,l} =\scD S_{j^*,n} = \scD \sL_{j}$.
 Let $\{ \bm{e}_{i,k}\mid (i,k) \in\gG^{(\gd)}_0\}$ be the standard basis of $\Z^{\gG_0^{(\gd)}}$,
 and set $\bm{c} = \sum_{s=1}^p \bm{e}_{i_s,k_s}$.
 It follows from Lemma \ref{Lem:parametrization} that
 \begin{equation}\label{eq:cLcD}
  \cL_{\cD}\big(B^{(\gd)}(\bm{c})\big) \cong \bS(\bm{P}),
 \end{equation}
 and hence we have 
 \begin{equation}\label{eq:tfd=gee}
  \tfd\big(S_{j,l},\bS(\bm{P})\big)= \tfd\big(\scD \sL_j, \bS(\bm{P})\big) = \gee_j\big(B^{(\gd)}(\bm{c})\big)
 \end{equation}
 by Proposition \ref{Prop:Kashiwara_Park} (c).  
 We see from Theorem \ref{Thm:Reineke} that 
 \begin{equation}\label{gee_j}
  \gee_j\big(B^{(\gd)}(\bm{c})\big) = \max_{\gS \in \,\mathcal{U}_j} \big(\,\sharp \{s \in [1,p] \mid (i_s,k_s) \in \gS\} -\sharp\{s \in [1,p] \mid (i_s,k_s+2) \in \gS\}\big).
 \end{equation}
 Define $P_1,P_2,\ldots,P_{2p} \in \wh{Q}_0$ by 
 \[ P_{2s-1} = (i_s,k_s) \ \ \ \text{and} \ \ \ P_{2s} = (i_s,k_s+2) \ \ \ \text{for $s \in [1,p]$}.
 \]
 Note that $P_1 \prec P_2 \preceq P_3 \prec \cdots \preceq P_{2p-1} \prec P_{2p}$ holds by the definition of the snake position.
 Now we show the assertion (a).
 In this case, we have $P_1 \in \gO_j$.
 For any $\gS \in \mathcal{U}_j$,
 we easily see from the lower closedness that there is some $t \in [0,2p]$ such that $P_s \in \gS$ if and only if $s \in [1,t]$,
 and then we have
 \[ \sharp \{s \in [1,p] \mid P_{2s-1} \in \gS\} -\sharp\{s \in [1,p] \mid P_{2s} \in \gS\} = \begin{cases} 0 & \text{if $t \in 2\Z$}, \\ 1 & \text{if $t \in 2\Z-1$}.\end{cases}
 \]
 Hence (a) follows from (\ref{eq:tfd=gee}) and (\ref{gee_j}).
 In the case of (b), we have $P_1 \notin \gO_j$ and $P_2 \in \gO_j$, and the assertion is proved similarly.

 (2)  Let $\gd = \ol{n-j} \in \{0,1\}$.
 Similarly as above, we may assume that 
 \[ l=n+1, \ \ \ \xi=\xi^{(\gd)},  \ \ \ and \ \ \ \cD=\{\sL_i\}_{i\in I} \ \ \text{with} \ \ \sL_i = \begin{cases} S_{i,0} & \text{if $\bar{i} \neq \gd$}\\
                                                                                                    S_{i^*,n} & \text{if $\bar{i} = \gd$}\end{cases} \ \ \text{for $i \in I$}.
 \]
 By a similar argument as above, we may further assume that $(i_s,k_s) \in \gG_0^{(\gd)}$ for all $s \in [1,p]$.
 It follows from Proposition \ref{Prop:Kashiwara_Park} that
 \[ \tfd\big(\bS(\bm{P}),S_{j,l}\big) = \tfd\big(\bS(\bm{P}), \scD^{-1} \sL_{j^*}\big)=\gee_{j^*}^*\big(B^{(\gd)}(\bm{c})\big),
 \]
 where we set $\bm{c} = \sum_{s=1}^p \bm{e}_{i_s,k_s} \in \Z_{\geq 0}^{\gG^{(\gd)}_0}$.
 We easily see that if $\bm{P}$ and $(j,l)$ satisfy the assumption of (c) (resp.\ (d)), then 
 \[ \bm{P}^\vee=\big((i^*_p,n-k_p),(i^*_{p-1},n-k_{p-1}),\ldots,(i^*_1,n-k_1)\big)
 \]
 and $(j^*,n-l)$ do that of (a) (resp.\ (b)).
 Hence the assertions (c) and (d) are proved from (the proof of) (a) and (b) by using (\ref{eq:duality_of_gee}).
\end{proof}

Let $(i,k),(i',k') \in \wh{Q}_0$, and suppose that $(i',k')$ is in prime snake position with respect to $(i,k)$.
We define $Q_{i,k}^{i',k'}$ and $R_{i,k}^{i',k'}$, each of which is an element of $\wh{Q}_0$ or the empty set, by
\begin{align*}
 Q^{i',k'}_{i,k} &= \begin{cases} \big(\frac{1}{2}(i+i'+k-k'),\frac{1}{2}(i-i'+k+k')\big) & \text{if} \ k' -k < i+i',\\
                                  \emptyset                                                & \text{if} \ k'-k = i+i',\end{cases}\\
 R_{i,k}^{i',k'} &= \begin{cases} \big(\frac{1}{2}(i+i'-k+k'),\frac{1}{2}(-i+i'+k+k')\big) & \text{if} \ k'-k < 2n+2-i-i',\\
                                  \emptyset                                                & \text{if} \ k'-k= 2n+2-i-i'.\end{cases}
\end{align*}
When $n=5$, these are illustrated as follows, where $(i,k)$ (resp.\ $(i'.k')$, $Q_{i,k}^{i',k'}$, $R_{i,k}^{i',k'}$) are shown as $\circ$ (resp.\ $\bullet$, $\ast$, $\star$):
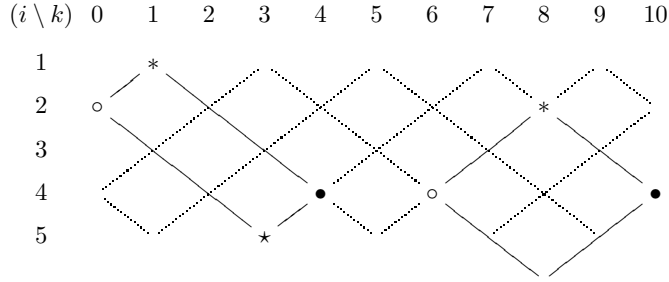
\begin{equation*}
\scalebox{0.85}{\xymatrix@!C=0.3mm@R=2mm{
(i \setminus k) &0 & 1& 2 & 3& 4& 5& 6 & 7& 8&  9 &10\\
1& & \ast \ar@{-}[rrrddd]&  & \ar@{.}[rrrddd]& &\ar@{.}[rrrrdddd]&&\ar@{.}[rd] && \ar@{.}[rd] \\
2& \circ \ar@{-}[ru] \ar@{-}[rrrddd] && & & & & && \ast \ar@{-}[rrdd] \ar@{.}[ru]&& \\
3& & && & &  &&  && \\
4& \ar@{.}[rd] \ar@{.}[rrruuu]& &  & & \bullet \ar@{.}[rrruuu] \ar@{.}[rd]&  &\circ\ar@{-}[rrdd] \ar@{-}[rruu] && && \bullet \ar@{-}[lldd] \\
5& &\ar@{.}[rrrruuuu]&  &\star\ar@{-}[ru]&  & \ar@{.}[ru]& &\ar@{.}[rrruuu]&&& &\\
                &  &&& &&&&&&&& {\phantom  2} \\
                 &&&&&&&&&&&&&}}
\end{equation*}

\begin{Lem}\label{Lem:product_of_two}
 Let $(j,l),(j',l') \in \wh{Q}_0$, and assume that $(j',l')$ is in prime snake position with respect to $(j,l)$.
 Then we have $S_{j,l} \Del S_{j',l'} \cong S_{Q_{j,l}^{j',l'}} \otimes S_{R_{j,l}^{j',l'}}$, where we set $S_\emptyset = \bm{1}$.
\end{Lem}

\begin{proof}
 Essentially, this is a formula for the product of two dual root vectors, which has previously been known (see \cite{zbMATH07051972}).
 For the reader's convenience, we give a proof.

 As above, we may assume that 
 \[ l=-1, \ \ \ \xi=\xi^{(\gd)} \ \ \text{with} \ \ \gd = \bar{j}, \ \ \ \text{and} \ \ \ \cD=\{\sL_i\}_{i \in I} \ \ \text{with} \ \ 
    \sL_i=\begin{cases} S_{i,0} & \text{if $\bar{i}\neq\gd$},\\ S_{i^*,n} & \text{if $\bar{i}=\gd$}.
    \end{cases}
 \] 
 Set 
 \[ \Pi := \{ (i,k) \in \wh{Q}_0 \mid (j,l) \prec (i,k) \prec (j',l')\} \subseteq \gG^{(\gd)}_0.
 \]
 It follows from Lemma \ref{Lem:calculation_of_std} that $\tfd(S_{j,l},S_{j',l'}) =1$, and thus by Lemma \ref{Lem:expression_of_socle} we have 
 \begin{equation}\label{eq:SjlSjl}
  S_{j,l} \Del S_{j',l'} \cong \hd\Big(\bigotimes_{(i,k) \in \Pi} S_{i,k}^{\otimes a_{i,k}}\Big)
 \end{equation}
 for some $a_{i,k} \in \Z_{\geq 0}$, where the factors are ordered compatibly with $\preceq$. 
 Set $\bm{a} = (a_{i,k}) \in \Z_{\geq 0}^{\gG_0^{(\gd)}}$, where $a_{i,k} = 0$ if $(i,k) \notin \Pi$.
 As in (\ref{eq:cLcD}), the image of $B^{(\gd)}(\bm{a})$ under $\cL_\cD\colon \Bup \to \Irr(\scC_\fg)$ is isomorphic to the right-hand side of (\ref{eq:SjlSjl}),
 and we have
 \[ \wt\big(B^{(\gd)}(\bm{a})\big) = -\sum_{(i,k) \in \Pi} a_{i,k} \phi_{(\gd)}(i,k).
 \]
 On the other hand, we have 
 \[ \sL_j \nab (S_{j,l} \Del S_{j',l'}) \cong \sL_j \nab (S_{j',l'} \nab \scD \sL_j) \cong S_{j',l'},
 \] 
 and since $\sL_j$ and $S_{j',l'}$ are the images under $\cL_{\cD}$ of the dual root vectors of weight $-\ga_j$ and $-\phi_{(\gd)}(j',l')$ respectively,
 we see that the weight of $\cL_\cD^{-1}(S_{j,l} \Del S_{j',l'}) \in \Bup$ is $-\phi_{(\gd)}(j',l') +\ga_j$. 
 Hence it follows from (\ref{eq:SjlSjl}) that 
 \begin{equation}\label{eq:sumikT}
  \sum_{(i,k) \in \Pi} a_{i,k} \phi_{(\gd)}(i,k)= \ga_{x,y} -\ga_j = \ga_{x,j-1} + \ga_{j+1,y},
 \end{equation}
 where we set $\phi_{(\gd)}(j',l') =\ga_{x,y}$.
 We easily check from Lemma \ref{Lem:Image_of_phi} that 
 \[ \phi_{(\gd)}(\Pi) \subseteq \{\ga_{r,s} \mid r\leq j \leq s\} \sqcup \{\ga_{r,j-1} \mid r\leq j-1 \} \sqcup \{\ga_{j+1,s}\mid j+1 \leq s\},
 \]
 and from this we see at once that (\ref{eq:sumikT}) holds only when $\bm{a}$ is given as follows:
 \[ a_{i,k} = 1 \ \ \text{if $\phi_{(\gd)}(i,k) \in \{\ga_{x,j-1}, \ga_{j+1,y}\}$} \ \ \ and \ \ \ a_{i,k}=0 \ \ \text{otherwise}.
 \]
 It is easily seen from Lemma \ref{Lem:Image_of_phi} that $x=j$ if $Q_{j,l}^{j',l'} =\emptyset$ and $\phi_{(\gd)}(Q_{j,l}^{j',l'}) = \ga_{x,j-1}$ otherwise.
 Similarly, we see that $y=j$ if $R_{j,l}^{j',l'} =\emptyset$ and $\phi_{(\gd)}(R_{j,l}^{j',l'})=\ga_{j+1,y}$ otherwise.
 Now, since $Q_{j,l}^{j',l'}$ and $R_{j,l}^{j',l'}$ are incomparable when they are nonempty, (\ref{eq:SjlSjl}), together with Lemma \ref{Lem:fundamentals_of_T} (d), completes the proof.
\end{proof}

The following lemma is proved by inspection.

\begin{Lem}[{\cite[Proposition 3.2]{MR2960028}}]\label{Lem:being_snake}
 Let $\bm{P} = \big((i_1,k_1),\ldots,(i_p,k_p)\big) \in (\wh{Q}_0)^p$ be a prime snake with $p \geq 2$, and set
 \[ \bm{Q}= (Q_{i_1,k_1}^{i_2,k_2},\ldots,Q_{i_{p-1},k_{p-1}}^{i_p,k_p}) \ \ \ \text{and} \ \ \ \bm{R} = (R_{i_1,k_1}^{i_2,k_2},\ldots,R_{i_{p-1},k_{p-1}}^{i_p,k_p}),
 \]
 where $\emptyset$ are ignored. 
 Then $\bm{Q}$ and $\bm{R}$ are snakes with no elements in common.
\end{Lem}

Now we give the main theorem of this section, which is a generalization of \cite[Theorem 3.4]{zbMATH06988693} and \cite[Proposition 3.1, Theorem 4.1]{MR2960028} in type $A$.

\begin{Thm}\label{Thm:MainA} 
 Let $\bm{P} = \big((i_1,k_1),\ldots,(i_p,k_p)\big) \in (\wh{Q}_0)^p$ be a snake.
 \begin{itemize}\setlength{\leftskip}{-15pt}
  \item[{\normalfont(i)}] The simple module $\bS(\bm{P})$ is real.
  \item[{\normalfont(ii)}] If $\bm{P}$ is prime, then $\bS(\bm{P})$ is prime.
  \item[{\normalfont(iii)}] Assume that $\bm{P}$ is prime with $p \geq 2$, and set
   \[ \bm{Q}= (Q_{i_1,k_1}^{i_2,k_2},\ldots,Q_{i_{p-1},k_{p-1}}^{i_p,k_p}) \ \ \ \text{and} \ \ \ \bm{R}=(R_{i_1,k_1}^{i_2,k_2},\ldots,
      R_{i_{p-1},k_{p-1}}^{i_p.k_p}).
   \]
   Then $\bS(\bm{Q})$ and $\bS(\bm{R})$ strongly commute, and there is a short exact sequence 
   \begin{equation}\label{eq:Main_ses}
    0 \to \bS(\bm{Q}) \otimes \bS(\bm{R}) \to \bS(\bm{P}_{[1,p-1]}) \otimes \bS(\bm{P}_{[2,p]}) \to \bS(\bm{P}) \otimes \bS(\bm{P}_{[2,p-1]}) \to 0.
   \end{equation}
  \end{itemize}
\end{Thm}

\begin{proof}
 Using Lemma \ref{Lem:calculation_of_std}, the assertion (i) follows from Proposition \ref{Prop:reality} and Lemma \ref{Lem:for_prime} (i), and (ii) from Proposition \ref{Prop:primeness}.
 For (iii), it suffices to show by Theorem \ref{Thm:Main_thm_general} that 
 \begin{equation}\label{eq:hdleft}
  \hd\Big(\bigotimes_{a=1}^{p-1} (S_{i_a,k_a} \Del S_{i_{a+1},k_{a+1}})\Big) \cong \bS(\bm{Q}) \otimes \bS(\bm{R}),
 \end{equation}
 which we prove by the induction on $p$.
 The case $p=2$ is just Lemma \ref{Lem:product_of_two}.
 Assume that $p>2$, and write $Q$ for $Q_{i_{p-1},k_{p-1}}^{i_p,k_p}$ and $R$ for $R_{i_{p-1},k_{p-1}}^{i_p,k_p}$. 
 By the induction hypothesis, we have
 \[ \hd\Big(\bigotimes_{a=1}^{p-1} (S_{i_a,k_a} \Del S_{i_{a+1},k_{a+1}})\Big) \cong \big(\bS(\bm{Q}') \otimes \bS(\bm{R}')\big) \nab (S_{Q} \otimes S_{R}),
 \]
 where we set
 \[ \bm{Q}'= (Q_{i_1,k_1}^{i_2,k_2},\ldots,Q_{i_{p-2},k_{p-2}}^{i_{p-1},k_{p-1}}) \ \ \ \text{and} \ \ \ \bm{R}'=(R_{i_1,k_1}^{i_2,k_2},\ldots,
      R_{i_{p-2},k_{p-2}}^{i_{p-1}.k_{p-1}}).
 \]
 To prove (\ref{eq:hdleft}), it is enough to show that $\tfd(\bS(\bm{Q}'), S_{R})= \tfd(\bS(\bm{R}'),S_Q) = 0$.
 Indeed if this holds, then we have 
 \[ \tfd\big(\bS(\bm{Q}), \bS(\bm{R})\big) = \tfd\big(\bS(\bm{Q}') \nab S_Q, \bS(\bm{R}')\nab S_R\big) = 0,
 \]
 and there is a surjection
 \begin{align*}
  \big(\bS(\bm{Q}') \otimes \bS(\bm{R}')\big) \otimes (S_{Q} \otimes S_{R}) &\cong \big(\bS(\bm{Q}') \otimes S_Q\big) \otimes \big(\bS(\bm{R}') \otimes S_R\big)\\
  & \twoheadrightarrow \bS(\bm{Q}) \otimes \bS(\bm{R}),
 \end{align*}
 which implies (\ref{eq:hdleft}).
 Let us prove $\tfd(\bS(\bm{Q}'), S_{R})=0$ (the other is proved similarly).
 We may assume that $R\neq \emptyset$.
 Set $Q'=Q_{i_{p-2},k_{p-2}}^{i_{p-1},k_{p-1}}$.
 If $Q' \neq \emptyset$, then $\tfd\big(\bS(\bm{Q}'), S_R\big)=0$ follows from Lemma \ref{Lem:calculation_of_std} (d).
 Assume that $Q' = \emptyset$, which implies that $ k_{p-1}-k_{p-2}=i_{p-1}+i_{p-2}$.
 We easily see that every element $Q''$ appearing in $\bm{Q}'$ satisfies 
 \begin{equation}\label{eq:Q''}
   Q'' \preceq (i_{p-2}-1,k_{p-2}-1) \ \ \text{if $i_{p-2} \neq 1$} \ \ \ \text{and} \ \ \ Q'' \preceq (i_{p-2}, k_{p-2}-2) \ \ \text{if $i_{p-2}=1$}. 
 \end{equation}
 By Lemma \ref{Lem:for_prime}, it suffices to show that $\scD^{-1}Q'' \not\succeq R$ holds if $Q''$ satisfies (\ref{eq:Q''}).
 We show this in the case $i_{p-2} \neq 1$ (the case $i_{p-2}=1$ is proved similarly).
 We have 
 \[ \scD^{-1} Q'' \preceq (-i_{p-2}+n+2, k_{p-2}+n) = (i_{p-1}+s,k_{p-1}-2+s)
 \]
 for suitable $s \in \Z$.
 On the other hand, $R = (i_{p-1}+r,k_{p-1}+r)$ holds for some $r \in \Z_{>0}$, and therefore we have $\scD^{-1}Q'' \not\succeq R$, as required.
 The proof is complete.
\end{proof}

\begin{Exa}\normalfont
 Here we give an example of the short exact sequence (\ref{eq:Main_ses}), not corresponding to any of the Mukhin--Young's extended $T$-systems in \cite{MR2960028}.

 Let $\fg$ be of type $A_3^{(1)}$, and set $\cD=\{L_i\}_{i \in [1,3]}$ with $L_i = L(Y_{1,2i-1})$,
 which forms a strong duality datum associated with $\mathfrak{sl}_{4}$.
 Define $\xi \in \HF$ by $\xi_i = i$ ($i \in [1,3]$).
 Then it can be proved using Lemma \ref{Lem:product_of_two} and (\ref{eq:monomial_of_tensor}) that
 \begin{align*}
  S_{1,2i-1} &= L(Y_{1,2i-1}) \ \ (i \in [1,3]), \ \ \ S_{2,2} = L(Y_{1,3}Y_{1,1}),\\
  S_{2,4} &= L(Y_{1,5}Y_{1,3}), \ \ \text{and} \ \ S_{3,3} = L(Y_{1,5}Y_{1,3}Y_{1,1}).
 \end{align*}
 By applying $\scD^{\pm 1}$, all $S_{i,k}$ ($(i,k) \in \wh{Q}_0^\xi$) are obtained (see (\ref{eq:dual_of_monomial})).
 Now the sequence (\ref{eq:Main_ses}) for a snake $\bm{P}=\big((2,0),(2,2),(1,5)\big)$ is as follows:
 \begin{align*}
  0 &\to L(Y_{1,1}) \otimes \hd\big(L(Y_{1,9}) \otimes L(Y_{1,5}Y_{1,3}Y_{1,1})\big)\\ &\to \hd\big(L(Y_{3,9}Y_{3,7}) \otimes L(Y_{1,3}Y_{1,1})\big) \otimes 
  \hd\big(L(Y_{1,3}Y_{1,1})\otimes L(Y_{1,5})\big)\\
    &\to \hd\big(L(Y_{3,9}Y_{3,7}) \otimes L(Y_{1,3}Y_{1,1})\otimes L(Y_{1,5})\big) \otimes L(Y_{1,3}Y_{1,1}) \to 0,
 \end{align*}
 or more explictily,
 \[ 0 \!\to\! L(Y_{1,9}Y_{1,5}Y_{1,3}Y_{1,1}^2)\! \to\! L(Y_{3,9}Y_{3,7}Y_{1,3}Y_{1,1}) \otimes L(Y_{2,4}Y_{1,1})\! \to\! L(Y_{3,9}Y_{3,7}Y_{2,4}Y_{1,3}Y_{1,1}^2)\! \to\! 0.
 \]
\end{Exa}


\section{The case of twisted height functions}\label{Section:twisted}


\subsection{Change of a reduced word}

Throughout this section, we assume that $n = 2n_0 -1$ with $n_0 \in \Z_{\geq 2}$.
Unlike the untwisted case, the results in \cite{zbMATH01031515} cannot be applied to reduced expressions adapted to twisted height functions.
In this subsection, we will describe the connection between untwisted and twisted cases,
which enables us to apply the results in the previous section to twisted cases.

Let $\theta$ be the untwisted height function defined by
\[ \theta_i = \begin{cases} i & \text{for $i \in [1,n_0]$},\\ i-2 & \text{for $i \in [n_0+1,n]$},\end{cases}
\]
and $\Theta$ the twisted height function defined by
\[ \Theta_i = \begin{cases} i & \text{for $i \in [1,n_0-1]$},\\ n_0-1/2 & \text{for $i = n_0$}, \\ i-1 & \text{for $i \in [n_0+1,n]$}.\end{cases}
\]
In this section, we will treat these specific functions only.
When $n=7$, $\gG^\theta$ and $\gG^\Theta$ are given as follows:
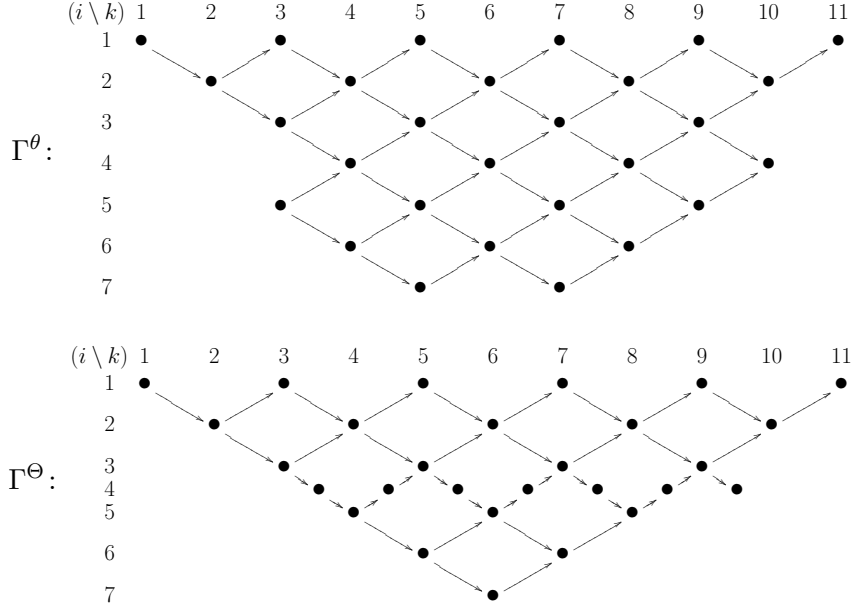
\begin{gather*}
\gG^\theta\colon
\raisebox{5em}{\scalebox{0.5}{\xymatrix@!C=0.5ex@R=0ex{
\mbox{\LARGE$(i\setminus k)$\ \ \ }  & \mbox{\LARGE $1$} && \mbox{\LARGE $2$} && \mbox{\LARGE $3$} & & \mbox{\LARGE $4$} & & \mbox{\LARGE$5$} &  & \mbox{\LARGE$6$} &  & \mbox{\LARGE$7$} & & 
\mbox{\LARGE$8$} && \mbox{\LARGE$9$} &&\mbox{\LARGE$10$} &&\mbox{\LARGE$11$} \\
\mbox{\LARGE$1$} &\mbox{\huge $\bullet$}\ar@{->}[ddrr]&&&& \mbox{\huge $\bullet$} \ar@{->}[ddrr]  &&&& \mbox{\huge $\bullet$} \ar@{->}[ddrr] &&&& \mbox{\huge $\bullet$} \ar@{->}[ddrr]
&&&& \mbox{\huge $\bullet$}\ar@{->}[ddrr]&&&&\mbox{\huge $\bullet$} \\ {\phantom 2} \\
\mbox{\LARGE$2$}&&&\mbox{\huge $\bullet$}\ar@{->}[ddrr]\ar@{->}[uurr]&& && \mbox{\huge $\bullet$}\ar@{->}[ddrr]\ar@{->}[uurr]&&&&  \mbox{\huge $\bullet$} \ar@{->}[ddrr]\ar@{->}[uurr] &&&& \mbox{\huge $\bullet$}\ar@{->}[ddrr]
\ar@{->}[uurr]& & & & \mbox{\huge $\bullet$}\ar@{->}[uurr]\\ {\phantom 2}  \\ 
\mbox{\LARGE$3$} &&&&& \mbox{\huge $\bullet$} \ar@{->}[ddrr] \ar@{->}[uurr] &&&& \mbox{\huge $\bullet$} \ar@{->}[ddrr]\ar@{->}[uurr] &&&& \mbox{\huge $\bullet$} \ar@{->}[ddrr] \ar@{->}[uurr]
&&&& \mbox{\huge $\bullet$} \ar@{->}[uurr]\ar@{->}[ddrr]& \\ {\phantom 2} \\
\mbox{\LARGE$4$}&&&&&&&\mbox{\huge $\bullet$}\ar@{->}[ddrr]\ar@{->}[uurr] &&&& \mbox{\huge $\bullet$}  \ar@{->}[ddrr]\ar@{->}[uurr] &&&&\mbox{\huge $\bullet$}\ar@{->}[uurr]
\ar@{->}[ddrr]&&&&\mbox{\huge $\bullet$}\\ {\phantom 2} \\
\mbox{\LARGE$5$} &&&&& \mbox{\huge $\bullet$}\ar@{->}[uurr]\ar@{->}[ddrr] &&&& \mbox{\huge $\bullet$} \ar@{->}[uurr]\ar@{->}[ddrr]
&&&& \mbox{\huge $\bullet$} \ar@{->}[uurr] \ar@{->}[ddrr] &&&&  \mbox{\huge $\bullet$}\ar@{->}[uurr]\\ {\phantom 2} \\ 
\mbox{\LARGE$6$}&&&&&&& \mbox{\huge $\bullet$}\ar@{->}[uurr]\ar@{->}[ddrr] &&&&  \mbox{\huge $\bullet$} \ar@{->}[uurr] \ar@{->}[ddrr]&&&& \mbox{\huge $\bullet$}\ar@{->}[uurr] &&&& \\
 &&  &&&&  && & && & {\phantom 2} \\
\mbox{\LARGE$7$}&&&&&  &&&& \mbox{\huge $\bullet$}\ar@{->}[uurr] &&&&\mbox{\huge $\bullet$}\ar@{->}[uurr]&&&&
}}}
\end{gather*}
\begin{gather*}
\gG^\Theta\colon
\raisebox{4.5em}{\scalebox{0.5}{\xymatrix@!C=0.5ex@R=0ex{
\mbox{\LARGE$(i\setminus k)$\ \ \ }  & \mbox{\LARGE $1$} && \mbox{\LARGE $2$} && \mbox{\LARGE $3$} & & \mbox{\LARGE $4$} & & \mbox{\LARGE$5$} &  & \mbox{\LARGE$6$} &  & \mbox{\LARGE$7$} & & 
\mbox{\LARGE$8$} && \mbox{\LARGE$9$} &&\mbox{\LARGE$10$} &&\mbox{\LARGE$11$} \\
\mbox{\LARGE$1$} &\mbox{\huge $\bullet$}\ar@{->}[ddrr]&&&& \mbox{\huge $\bullet$} \ar@{->}[ddrr]  &&&& \mbox{\huge $\bullet$} \ar@{->}[ddrr] &&&& \mbox{\huge $\bullet$} \ar@{->}[ddrr]
&&&& \mbox{\huge $\bullet$}\ar@{->}[ddrr]&&&&\mbox{\huge $\bullet$} \\ {\phantom 2} \\
\mbox{\LARGE$2$}&&&\mbox{\huge $\bullet$}\ar@{->}[ddrr]\ar@{->}[uurr]&& && \mbox{\huge $\bullet$}\ar@{->}[ddrr]\ar@{->}[uurr]&&&&  \mbox{\huge $\bullet$} \ar@{->}[ddrr]\ar@{->}[uurr] &&&& \mbox{\huge $\bullet$}\ar@{->}[ddrr]
\ar@{->}[uurr]& & & & \mbox{\huge $\bullet$}\ar@{->}[uurr]\\ {\phantom 2}  \\ 
\mbox{\LARGE$3$} &&&&& \mbox{\huge $\bullet$} \ar@{->}[dr] \ar@{->}[uurr] &&&& \mbox{\huge $\bullet$} \ar@{->}[dr]\ar@{->}[uurr] &&&& \mbox{\huge $\bullet$} \ar@{->}[dr] \ar@{->}[uurr]
&&&& \mbox{\huge $\bullet$} \ar@{->}[uurr]\ar@{->}[dr]&\\
\mbox{\LARGE$4$}  &&&&&& \mbox{\huge $\bullet$}\ar@{->}[dr]  && \mbox{\huge $\bullet$} \ar@{->}[ur] &&\mbox{\huge $\bullet$} \ar@{->}[dr] 
&& \mbox{\huge $\bullet$} \ar@{->}[ur] && \mbox{\huge$\bullet$} \ar@{->}[dr] & &
\mbox{\huge $\bullet$} \ar@{->}[ur] & &\mbox{\huge $\bullet$} \\
\mbox{\LARGE$5$}&&&&&&&\mbox{\huge $\bullet$}\ar@{->}[ur]\ar@{->}[ddrr]  &&&& \mbox{\huge $\bullet$}  \ar@{->}[ur] \ar@{->}[ddrr] 
&&&&\mbox{\huge $\bullet$}\ar@{->}[ur]  \\  {\phantom 2} & \\
\mbox{\LARGE$6$} &&&&&  &&&& \mbox{\huge $\bullet$} \ar@{->}[uurr]\ar@{->}[ddrr]
&&&& \mbox{\huge $\bullet$} \ar@{->}[uurr] &&&&  \\ {\phantom 2} \\
\mbox{\LARGE$7$}&&&&& && &&&&  \mbox{\huge $\bullet$} \ar@{->}[uurr] &&&& }}}
\end{gather*}
Similarly as in (\ref{eq:natural_identification}), given a $\gG_0^\gt$-tuple $\bm{c} =(c_{i,k})_{(i,k) \in \gG_0^\gt}$ of nonnegative integers,
define $B^{\theta}(\bm{c}) \in \Bup$ by taking a compatible reading $\gG_0^\gt=\{(i_1,k_1),\ldots,(i_N,k_N)\}$, putting $\bm{i}=(i_1,\ldots,i_N) \in R(w_0)$,
and setting $B^{\theta}(\bm{c}) = B^{\bm{i}}(\bm{c})$.
We also define $B^{\Theta}(\bm{c}') \in \Bup$ similarly for a $\gG^\gT_0$-tuple $\bm{c}'=(c_{i,k}')_{(i,k) \in \gG^\gT_0}$ of nonnegative integers.
Given $\bm{c}'$, the goal of this subsection is to give a formula for $\bm{c}$ satisfying $B^\gt(\bm{c}) = B^{\gT}(\bm{c}')$ using Proposition \ref{Prop:2-move_3-move} (ii), under certain conditions.

We prepare several notations. 
For $j \in [n_0+1,n]$, let $V\!\langle j\rangle$ be the subset of $I \times \frac{1}{2}\Z$ defined by 
\begin{align*}
  V\!\langle j\rangle  = \big(\gG_0^{\theta} \cap ([1,j-1] \times \Z)\big) &\sqcup \{(j,j-3/2+k) \mid k \in [0,2n-2j+1]\}\\ &\sqcup \{(i,i-1+2k) \mid i>j, k \in [0,n-i]\}.     
\end{align*}
We also set $V\!\langle n_0\rangle = \gG^\Theta_0$ and $V\!\langle n+1\rangle = \gG^\gt_0$.
When $n=7$, these are given as follows, where 
$\star$ in $\Vj$ denote the points $(j,j+2r-1/2),(j+1,j+2r),(j,j+2r+1/2)$ with $r \in [0,n-j-1]$ appearing in Lemma \ref{Lem:3-move_path} below
 (the dotted lines connecting points are for illustrative purposes only):
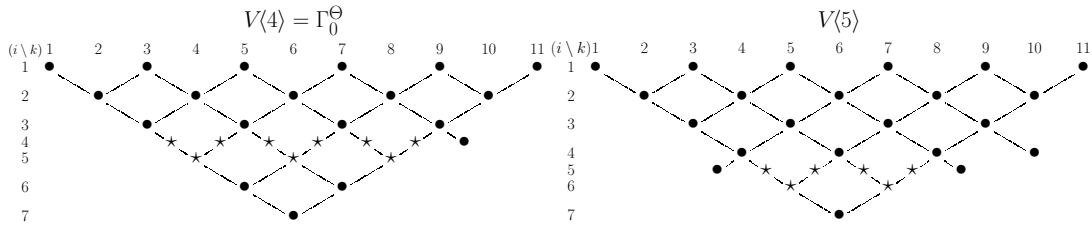
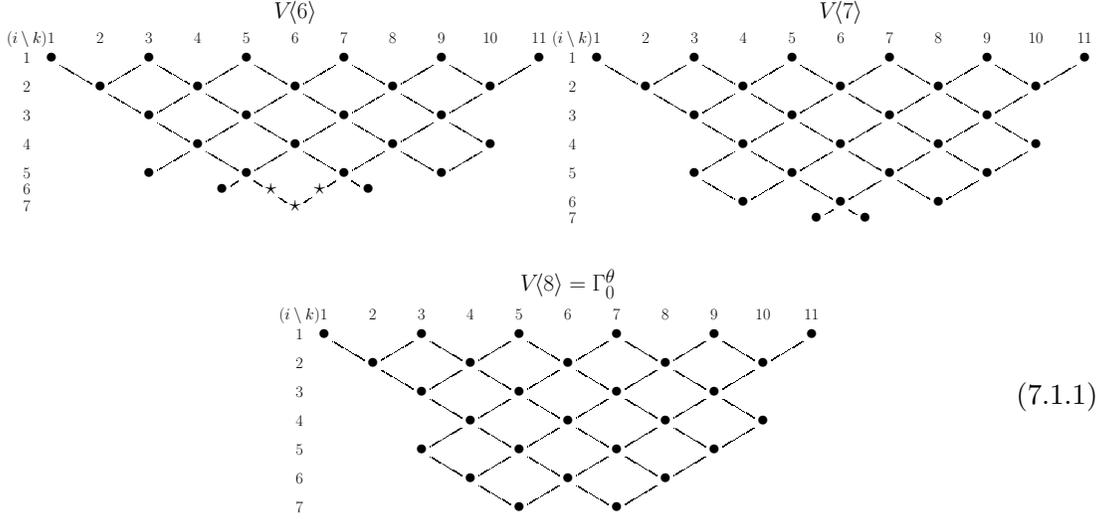
\begin{gather*}
\raisebox{4em}{\scalebox{0.35}{\xymatrix@!C=0.5ex@R=0ex{&&&&&&&&&&& \mbox{\Huge $V\!\langle 4 \rangle = \gG_0^\gT$} \\
\mbox{\Large$(i\setminus k)$}  & \mbox{\LARGE $1$} && \mbox{\LARGE $2$} && \mbox{\LARGE $3$} & & \mbox{\LARGE $4$} & & \mbox{\LARGE$5$} &  & \mbox{\LARGE$6$} &  & \mbox{\LARGE$7$} & & 
\mbox{\LARGE$8$} && \mbox{\LARGE$9$} &&\mbox{\LARGE$10$} &&\mbox{\LARGE$11$} \\
\mbox{\LARGE$1$} &\mbox{\Huge $\bullet$}\ar@{.}[ddrr]&&&& \mbox{\Huge $\bullet$} \ar@{.}[ddrr]  &&&& \mbox{\Huge $\bullet$} \ar@{.}[ddrr] &&&& \mbox{\Huge $\bullet$} \ar@{.}[ddrr]
&&&& \mbox{\Huge $\bullet$}\ar@{.}[ddrr]&&&&\mbox{\Huge $\bullet$} \\ {\phantom 2} \\
\mbox{\LARGE$2$}&&&\mbox{\Huge $\bullet$}\ar@{.}[ddrr]\ar@{.}[uurr]&& && \mbox{\Huge $\bullet$}\ar@{.}[ddrr]\ar@{.}[uurr]&&&&  \mbox{\Huge $\bullet$} \ar@{.}[ddrr]\ar@{.}[uurr] &&&& \mbox{\Huge $\bullet$}\ar@{.}[ddrr]
\ar@{.}[uurr]& & & & \mbox{\Huge $\bullet$}\ar@{.}[uurr]\\ {\phantom 2}  \\ 
\mbox{\LARGE$3$} &&&&& \mbox{\Huge $\bullet$} \ar@{.}[dr] \ar@{.}[uurr] &&&& \mbox{\Huge $\bullet$} \ar@{.}[dr]\ar@{.}[uurr] &&&& \mbox{\Huge $\bullet$} \ar@{.}[dr] \ar@{.}[uurr]
&&&& \mbox{\Huge $\bullet$} \ar@{.}[uurr]\ar@{.}[dr]&\\
\mbox{\LARGE$4$}  &&&&&& \mbox{\Huge $\star$}\ar@{.}[dr]  && \mbox{\Huge $\star$} \ar@{.}[ur] &&\mbox{\Huge $\star$} \ar@{.}[dr] 
&& \mbox{\Huge $\star$} \ar@{.}[ur] && \mbox{\Huge$\star$} \ar@{.}[dr] & &
\mbox{\Huge $\star$} \ar@{.}[ur] & &\mbox{\Huge $\bullet$} \\
\mbox{\LARGE$5$}&&&&&&&\mbox{\Huge $\star$}\ar@{.}[ur]\ar@{.}[ddrr]  &&&& \mbox{\Huge $\star$}  \ar@{.}[ur] \ar@{.}[ddrr] 
&&&&\mbox{\Huge $\star$}\ar@{.}[ur]  \\  {\phantom 2} & \\
\mbox{\LARGE$6$} &&&&&  &&&& \mbox{\Huge $\bullet$} \ar@{.}[uurr]\ar@{.}[ddrr]
&&&& \mbox{\Huge $\bullet$} \ar@{.}[uurr] &&&&  \\ {\phantom 2} \\
\mbox{\LARGE$7$}&&&&& && &&&&  \mbox{\Huge $\bullet$} \ar@{.}[uurr] &&&& }}} 
\raisebox{4em}{\scalebox{0.35}{\xymatrix@!C=0.5ex@R=0ex{&&&&&&&&&&& \mbox{\Huge $V\!\langle 5 \rangle$} \\
\mbox{\LARGE$(i\setminus k)$}  & \mbox{\LARGE $1$} && \mbox{\LARGE $2$} && \mbox{\LARGE $3$} & & \mbox{\LARGE $4$} & & \mbox{\LARGE$5$} &  & \mbox{\LARGE$6$} &  & \mbox{\LARGE$7$} & & 
\mbox{\LARGE$8$} && \mbox{\LARGE$9$} &&\mbox{\LARGE$10$} &&\mbox{\LARGE$11$} \\
\mbox{\LARGE$1$} &\mbox{\Huge $\bullet$}\ar@{.}[ddrr]&&&& \mbox{\Huge $\bullet$} \ar@{.}[ddrr]  &&&& \mbox{\Huge $\bullet$} \ar@{.}[ddrr] &&&& \mbox{\Huge $\bullet$} \ar@{.}[ddrr]
&&&& \mbox{\Huge $\bullet$}\ar@{.}[ddrr]&&&&\mbox{\Huge $\bullet$} \\ {\phantom 2} \\
\mbox{\LARGE$2$}&&&\mbox{\Huge $\bullet$}\ar@{.}[ddrr]\ar@{.}[uurr]&& && \mbox{\Huge $\bullet$}\ar@{.}[ddrr]\ar@{.}[uurr]&&&&  \mbox{\Huge $\bullet$} \ar@{.}[ddrr]\ar@{.}[uurr] &&&& \mbox{\Huge $\bullet$}\ar@{.}[ddrr]
\ar@{.}[uurr]& & & & \mbox{\Huge $\bullet$}\ar@{.}[uurr]\\ {\phantom 2}  \\ 
\mbox{\LARGE$3$} &&&&& \mbox{\Huge $\bullet$} \ar@{.}[ddrr] \ar@{.}[uurr] &&&& \mbox{\Huge $\bullet$} \ar@{.}[ddrr]\ar@{.}[uurr] &&&& \mbox{\Huge $\bullet$} \ar@{.}[ddrr] \ar@{.}[uurr]
&&&& \mbox{\Huge $\bullet$}\ar@{.}[uurr]\ar@{.}[ddrr]& &\\ {\phantom 2} \\
\mbox{\LARGE$4$}&&&&&&&\mbox{\Huge $\bullet$}\ar@{.}[dr]\ar@{.}[uurr] &&&& \mbox{\Huge $\bullet$}  \ar@{.}[dr]\ar@{.}[uurr] &&&&\mbox{\Huge $\bullet$}\ar@{.}[uurr]\ar@{.}[dr]
 &&&&\mbox{\Huge $\bullet$} \\  
\mbox{\LARGE$5$}  &&&&&& \mbox{\Huge $\bullet$} \ar@{.}[ur] &&\mbox{\Huge $\star$}\ar@{.}[dr]
&& \mbox{\Huge $\star$} \ar@{.}[ur] && \mbox{\Huge$\star$} \ar@{.}[dr] & &
\mbox{\Huge $\star$} \ar@{.}[ur] & & \mbox{\Huge $\bullet$} \\
\mbox{\LARGE$6$} &&&&& &&&& \mbox{\Huge $\star$} \ar@{.}[ur]\ar@{.}[ddrr]
&&&& \mbox{\Huge $\star$} \ar@{.}[ur] &&&&  \\ {\phantom 2} \\
\mbox{\LARGE$7$}&&&&& && &&&&  \mbox{\Huge $\bullet$} \ar@{.}[uurr] &&&&  
  }}} 
\end{gather*}
\begin{gather*}
\raisebox{4em}{\scalebox{0.35}{\xymatrix@!C=0.5ex@R=0ex{&&&&&&&&&&& \mbox{\Huge $V\!\langle 6 \rangle$} \\
\mbox{\LARGE$(i\setminus k)$}  & \mbox{\LARGE $1$} && \mbox{\LARGE $2$} && \mbox{\LARGE $3$} & & \mbox{\LARGE $4$} & & \mbox{\LARGE$5$} &  & \mbox{\LARGE$6$} &  & \mbox{\LARGE$7$} & & 
\mbox{\LARGE$8$} && \mbox{\LARGE$9$} &&\mbox{\LARGE$10$} &&\mbox{\LARGE$11$} \\
\mbox{\LARGE$1$} &\mbox{\Huge $\bullet$}\ar@{.}[ddrr]&&&& \mbox{\Huge $\bullet$} \ar@{.}[ddrr]  &&&& \mbox{\Huge $\bullet$} \ar@{.}[ddrr] &&&& \mbox{\Huge $\bullet$} \ar@{.}[ddrr]
&&&& \mbox{\Huge $\bullet$}\ar@{.}[ddrr]&&&&\mbox{\Huge $\bullet$} \\ {\phantom 2} \\
\mbox{\LARGE$2$}&&&\mbox{\Huge $\bullet$}\ar@{.}[ddrr]\ar@{.}[uurr]&& && \mbox{\Huge $\bullet$}\ar@{.}[ddrr]\ar@{.}[uurr]&&&&  \mbox{\Huge $\bullet$} \ar@{.}[ddrr]\ar@{.}[uurr] &&&& \mbox{\Huge $\bullet$}\ar@{.}[ddrr]
\ar@{.}[uurr]& & & & \mbox{\Huge $\bullet$}\ar@{.}[uurr]\\ {\phantom 2}  \\ 
\mbox{\LARGE$3$} &&&&& \mbox{\Huge $\bullet$} \ar@{.}[ddrr] \ar@{.}[uurr] &&&& \mbox{\Huge $\bullet$} \ar@{.}[ddrr]\ar@{.}[uurr] &&&& \mbox{\Huge $\bullet$} \ar@{.}[ddrr] \ar@{.}[uurr]
&&&& \mbox{\Huge $\bullet$} \ar@{.}[uurr]\ar@{.}[ddrr]& \\ {\phantom 2} \\
\mbox{\LARGE$4$}&&&&&&&\mbox{\Huge $\bullet$}\ar@{.}[ddrr]\ar@{.}[uurr] &&&& \mbox{\Huge $\bullet$}  \ar@{.}[ddrr]\ar@{.}[uurr] &&&&\mbox{\Huge $\bullet$}\ar@{.}[uurr]
\ar@{.}[ddrr]&&&&\mbox{\Huge $\bullet$}\\ {\phantom 2} \\
\mbox{\LARGE$5$} &&&&& \mbox{\Huge $\bullet$}\ar@{.}[uurr]&&&& \mbox{\Huge $\bullet$} \ar@{.}[uurr]\ar@{.}[dr]
&&&& \mbox{\Huge $\bullet$} \ar@{.}[uurr] \ar@{.}[dr] &&&&  \mbox{\Huge $\bullet$} \ar@{.}[uurr]\\ 
\mbox{\LARGE$6$}  &&&&&&   &&\mbox{\Huge $\bullet$}\ar@{.}[ur]
&& \mbox{\Huge $\star$} \ar@{.}[dr] && \mbox{\Huge$\star$} \ar@{.}[ur] & &
\mbox{\Huge $\bullet$}& & &&  \\
\mbox{\LARGE$7$}& &&&&&& &&&&  \mbox{\Huge $\star$} \ar@{.}[ur] &&&& }}} 
\raisebox{4em}{\scalebox{0.35}{\xymatrix@!C=0.5ex@R=0ex{&&&&&&&&&&& \mbox{\Huge $V\!\langle 7 \rangle$} \\
\mbox{\LARGE$(i\setminus k)$}  & \mbox{\LARGE $1$} && \mbox{\LARGE $2$} && \mbox{\LARGE $3$} & & \mbox{\LARGE $4$} & & \mbox{\LARGE$5$} &  & \mbox{\LARGE$6$} &  & \mbox{\LARGE$7$} & & 
\mbox{\LARGE$8$} && \mbox{\LARGE$9$} &&\mbox{\LARGE$10$} &&\mbox{\LARGE$11$} \\
\mbox{\LARGE$1$} &\mbox{\Huge $\bullet$}\ar@{.}[ddrr]&&&& \mbox{\Huge $\bullet$} \ar@{.}[ddrr]  &&&& \mbox{\Huge $\bullet$} \ar@{.}[ddrr] &&&& \mbox{\Huge $\bullet$} \ar@{.}[ddrr]
&&&& \mbox{\Huge $\bullet$}\ar@{.}[ddrr]&&&&\mbox{\Huge $\bullet$} \\ {\phantom 2} \\
\mbox{\LARGE$2$}&&&\mbox{\Huge $\bullet$}\ar@{.}[ddrr]\ar@{.}[uurr]&& && \mbox{\Huge $\bullet$}\ar@{.}[ddrr]\ar@{.}[uurr]&&&&  \mbox{\Huge $\bullet$} \ar@{.}[ddrr]\ar@{.}[uurr] &&&& \mbox{\Huge $\bullet$}\ar@{.}[ddrr]
\ar@{.}[uurr]& & & & \mbox{\Huge $\bullet$}\ar@{.}[uurr]\\ {\phantom 2}  \\ 
\mbox{\LARGE$3$} &&&&& \mbox{\Huge $\bullet$} \ar@{.}[ddrr] \ar@{.}[uurr] &&&& \mbox{\Huge $\bullet$} \ar@{.}[ddrr]\ar@{.}[uurr] &&&& \mbox{\Huge $\bullet$} \ar@{.}[ddrr] \ar@{.}[uurr]
&&&& \mbox{\Huge $\bullet$} \ar@{.}[uurr]\ar@{.}[ddrr]& \\ {\phantom 2} \\
\mbox{\LARGE$4$}&&&&&&&\mbox{\Huge $\bullet$}\ar@{.}[ddrr]\ar@{.}[uurr] &&&& \mbox{\Huge $\bullet$}  \ar@{.}[ddrr]\ar@{.}[uurr] &&&&\mbox{\Huge $\bullet$}\ar@{.}[uurr]
\ar@{.}[ddrr]&&&&\mbox{\Huge $\bullet$}\\ {\phantom 2} \\
\mbox{\LARGE$5$} &&&&& \mbox{\Huge $\bullet$}\ar@{.}[uurr]\ar@{.}[ddrr] &&&& \mbox{\Huge $\bullet$} \ar@{.}[uurr]\ar@{.}[ddrr]
&&&& \mbox{\Huge $\bullet$} \ar@{.}[uurr] \ar@{.}[ddrr] &&&&  \mbox{\Huge $\bullet$}\ar@{.}[uurr]\\ {\phantom 2} \\ 
\mbox{\LARGE$6$}&&&&&&& \mbox{\Huge $\bullet$}\ar@{.}[uurr]&&&&  \mbox{\Huge $\bullet$} \ar@{.}[uurr]\ar@{.}[dr] &&&& \mbox{\Huge $\bullet$}\ar@{.}[uurr] &&&& \\
\mbox{\LARGE$7$} &&&&&& &&&& \mbox{\Huge $\bullet$} \ar@{.}[ur] && \mbox{\Huge$\bullet$}& & & &
}}}
\end{gather*}
\begin{gather}\label{Graph}
\raisebox{4em}{\scalebox{0.35}{\xymatrix@!C=0.5ex@R=0ex{&&&&&&&&&&& \mbox{\Huge $V\!\langle 8 \rangle = \gG_0^\gt$} \\
\mbox{\LARGE$(i\setminus k)$}  & \mbox{\LARGE $1$} && \mbox{\LARGE $2$} && \mbox{\LARGE $3$} & & \mbox{\LARGE $4$} & & \mbox{\LARGE$5$} &  & \mbox{\LARGE$6$} &  & \mbox{\LARGE$7$} & & 
\mbox{\LARGE$8$} && \mbox{\LARGE$9$} &&\mbox{\LARGE$10$} &&\mbox{\LARGE$11$} \\
\mbox{\LARGE$1$} &\mbox{\Huge $\bullet$}\ar@{.}[ddrr]&&&& \mbox{\Huge $\bullet$} \ar@{.}[ddrr]  &&&& \mbox{\Huge $\bullet$} \ar@{.}[ddrr] &&&& \mbox{\Huge $\bullet$} \ar@{.}[ddrr]
&&&& \mbox{\Huge $\bullet$}\ar@{.}[ddrr]&&&&\mbox{\Huge $\bullet$} \\ {\phantom 2} \\
\mbox{\LARGE$2$}&&&\mbox{\Huge $\bullet$}\ar@{.}[ddrr]\ar@{.}[uurr]&& && \mbox{\Huge $\bullet$}\ar@{.}[ddrr]\ar@{.}[uurr]&&&&  \mbox{\Huge $\bullet$} \ar@{.}[ddrr]\ar@{.}[uurr] &&&& \mbox{\Huge $\bullet$}\ar@{.}[ddrr]
\ar@{.}[uurr]& & & & \mbox{\Huge $\bullet$}\ar@{.}[uurr]\\ {\phantom 2}  \\ 
\mbox{\LARGE$3$} &&&&& \mbox{\Huge $\bullet$} \ar@{.}[ddrr] \ar@{.}[uurr] &&&& \mbox{\Huge $\bullet$} \ar@{.}[ddrr]\ar@{.}[uurr] &&&& \mbox{\Huge $\bullet$} \ar@{.}[ddrr] \ar@{.}[uurr]
&&&& \mbox{\Huge $\bullet$} \ar@{.}[uurr]\ar@{.}[ddrr]& \\ {\phantom 2} \\
\mbox{\LARGE$4$}&&&&&&&\mbox{\Huge $\bullet$}\ar@{.}[ddrr]\ar@{.}[uurr] &&&& \mbox{\Huge $\bullet$}  \ar@{.}[ddrr]\ar@{.}[uurr] &&&&\mbox{\Huge $\bullet$}\ar@{.}[uurr]
\ar@{.}[ddrr]&&&&\mbox{\Huge $\bullet$}\\ {\phantom 2} \\
\mbox{\LARGE$5$} &&&&& \mbox{\Huge $\bullet$}\ar@{.}[uurr]\ar@{.}[ddrr] &&&& \mbox{\Huge $\bullet$} \ar@{.}[uurr]\ar@{.}[ddrr]
&&&& \mbox{\Huge $\bullet$} \ar@{.}[uurr] \ar@{.}[ddrr] &&&&  \mbox{\Huge $\bullet$}\ar@{.}[uurr]\\ {\phantom 2} \\ 
\mbox{\LARGE$6$}&&&&&&& \mbox{\Huge $\bullet$}\ar@{.}[uurr]\ar@{.}[ddrr] &&&&  \mbox{\Huge $\bullet$} \ar@{.}[uurr] \ar@{.}[ddrr]&&&& \mbox{\Huge $\bullet$}\ar@{.}[uurr] &&&& \\
 &&  &&&&  && & && & {\phantom 2} \\
\mbox{\LARGE$7$}&&&&&  &&&& \mbox{\Huge $\bullet$}\ar@{.}[uurr] &&&&\mbox{\Huge $\bullet$}\ar@{.}[uurr]&&&&
}}} 
\end{gather}
We have $\sharp\, \Vj = N$ for all $j \in [n_0,n+1]$.
For each $j \in [n_0,n+1]$, take and fix a total ordering $\Vj = \{(i_1,k_1),\ldots,(i_N,k_N)\}$ such that $r<s$ holds whenever $k_r<k_s$,
and set $\bm{i}\langle j\rangle = (i_1,\ldots,i_N) \in I^N$.
For $(i,k) \in V\!\langle j \rangle$, denote by $i^{(k)}$ the letter $i_r$ in $\bm{i}\langle j \rangle$,
where $r \in [1,N]$ is such that $(i_r,k_r)=(i,k)$.
Obviously, $\bm{i}\langle n_0\rangle$ (resp.\ $\bm{i}\langle n+1\rangle$) is adapted to the twisted height function $\Theta$ (resp.\ the height function $\theta$),
and we easily see that $\bm{i}\langle n\rangle$ and $\bm{i}\langle n+1\rangle$ are commutation equivalent.
Let $[\bm{i}\langle j\rangle]$ denote the commutation class containing $\bm{i}\langle j\rangle$.

\begin{Lem}\label{Lem:3-move_path}\ 
 \begin{itemize}\setlength{\leftskip}{-15pt}
  \item[{\normalfont(i)}] For each $j \in [n_0,n-1]$, 
   there is a word $\bm{i} \in [\bm{i}\langle j\rangle]$ containing consecutive subwords 
   \begin{equation*}\label{eq:Coxeter_relation}
    \big(j^{(j+2r-1/2)}, (j+1)^{(j+2r)},j^{(j+2r+1/2)}\big) \ \ \ \text{for all } \ \  r \in [0,n-j-1].
   \end{equation*}
   Let $\bm{i}'$ be the word obtained from $\bm{i}$ by transforming all these subwords into $(j+1,j,j+1)$. 
   Then we have $\bm{i}' \in [\bm{i}\langle j+1 \rangle]$.
  \item[{\normalfont(ii)}] We have $\bm{i}\aj \in R(w_0)$ for all $j \in [n_0,n+1]$.
 \end{itemize}

 \begin{proof}
  The first assertion of (i) is easily checked,
  and the second is also checked directly by noting that the transformation $(j,j+1,j) \to (j+1,j,j+1)$ is expressed pictorially as follows:
  \[ \raisebox{3.1em}{\scalebox{0.7}
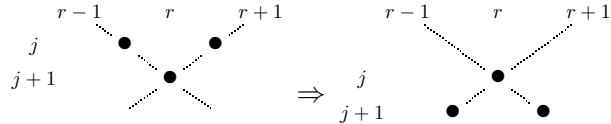
{\xymatrix@!C=0.1ex@R=0.5ex{
       &  r-1 & & r & & r+1\\
    j  &      & \mbox{\LARGE $\bullet$}\ar@{.}[ul] \ar@{.}[dr]  && \mbox{\LARGE $\bullet$} \ar@{.}[ur]\\
    j+1&      & & \mbox{\LARGE $\bullet$} \ar@{.}[ld] \ar@{.}[rd]\ar@{.}[ur]&& \\
       &      & & && {\phantom 2}}}} \Rightarrow
    \raisebox{3.1em}{\scalebox{0.7}{\xymatrix@!C=0.1ex@R=0.5ex{
       &  r-1\ar@{.}[ddrr] & & r & & r+1 \\
       & &  & & & & {\phantom j}  \\
    j  &       &  & \mbox{\LARGE $\bullet$} \ar@{.}[uurr] \ar@{.}[dr] && \\
    j+1& &\mbox{\LARGE $\bullet$}\ar@{.}[ur] &  &\mbox{\LARGE $\bullet$}& }}} 
  \]
  See (\ref{Graph}). Now (ii) is obvious since $\bm{i}\langle n_0\rangle \in R(w_0)$.
 \end{proof}
\end{Lem}

For each $j \in [n_0,n]$, define the map 
\[ \rho_{\langle j\rangle}\colon \Z_{\geq 0}^{\Vj} \to \Z_{\geq 0}^{V\!\langle j+1\rangle}, \ \ \ \bm{c}=(c_{i,k})_{(i,k) \in \Vj} \mapsto  \bm{c}'=(c_{i,k}')_{(i,k) \in V\!\langle j+1\rangle}
\]
as follows:
for $r \in [0,n-j-1]$, $(c_{j+1,j+2r-1/2}',c_{j,j+2r}',c_{j+1,j+2r+1/2}')$ is obtained by applying the transformation (\ref{eq:3-move}) to
$(c_{j,j+2r-1/2},c_{j+1,j+2r},c_{j,j+2r+1/2})$,
and the other $c_{i,k}'$'s are determined by $c_{i,k}' = c_{i,k+t}$, where
\[ t = \begin{cases} 1/2 & \text{if} \ (i,k)=(j,j-2),\\ -1/2 & \text{if} \ (i,k)=(j,2n-j),\\
                     0 & \text{if} \ i \neq j,j+1.\end{cases}
\]
Set
\[ \rho =\rho_{\langle n\rangle}\circ\rho_{\langle n-1\rangle }\circ \cdots \circ \rho_{\langle n_0\rangle} \colon \Z_{\geq 0}^{\gG^\gT_0} \to \Z_{\geq 0}^{\gG^\gt_0}.
\]
For any $\bm{c} \in \Z_{\geq 0}^{\gG^\gT_0}$, it follows from Lemma \ref{Lem:3-move_path} and Proposition \ref{Prop:2-move_3-move} (ii) that 
\begin{equation}\label{eq:coincidence_of_two_B}
 B^\gT(\bm{c}) = B^\gt\big(\rho(\bm{c})\big).
\end{equation}
Let $\{\bm{e}_{i,k}\mid (i,k) \in \gG_0^\Theta\}$ denote the standard basis of $\Z^{\gG_0^\gT}$.
We would like to give an explicit formula for $\rho(\bm{c})$ when $\bm{c}$ is expressed as $\sum_{r=1}^p \bm{e}_{i_r,k_r}$ with $\big((i_1,k_1),\ldots,(i_p,k_p)\big)$ 
being a snake in $\gG^\gT_0\subseteq \wh{Q}^\gT_0$.
For this, the following remark is useful.

\begin{Rem}\normalfont\label{Rem:3-terms}
 Under the transformation $(c_{k-1},c_k,c_{k+1}) \mapsto (c_{k-1}',c_k',c_{k+1}')$ in (\ref{eq:3-move}), we have 
 \begin{align*}
  \text{(i)} \ \ \ \ \ &(1,0,0) \mapsto (0,0,1), \ \ \ \ \ \ \ \ \ \
  \text{(ii)}\ \ \ \ \   (0,0,1) \mapsto (1,0,0),\\
  \text{(iii)}\ \ \ \ \   &(0,1,0) \mapsto (1,0,1), \ \ \ \ \ \ \ \ \ 
  \text{(iv)}\ \ \ \ \  (1,0,1) \mapsto (0,1,0).
 \end{align*} 
 Pictorially, these are expressed as follows:
 \begin{gather*}
  \raisebox{23pt}{(i)} \raisebox{3.1em}{\scalebox{0.6}{\xymatrix@!C=0.1ex@R=0.5ex{
       & \mbox{\LARGE $\star$} \ar@{.}[dr]  && \mbox{\LARGE $\bullet$} \\
       & & \mbox{\LARGE $\bullet$} \ar@{.}[ld] \ar@{.}[rd]\ar@{.}[ur]&& \\
            & & && {\phantom 2}}}} \raisebox{20pt}{{$\Rightarrow$}}
    \raisebox{3.1em}{\scalebox{0.6}{\xymatrix@!C=0.1ex@R=0.5ex{
        & & & & & {\phantom j}  \\
        &   &  & \mbox{\LARGE $\bullet$} \ar@{.}[ul]\ar@{.}[ur] \ar@{.}[dr] && \\
        & &\mbox{\LARGE $\bullet$}\ar@{.}[ur] &  &\mbox{\LARGE $\star$}& }}} 
    \raisebox{23pt}{(ii)} \raisebox{3.1em}{\scalebox{0.6}{\xymatrix@!C=0.1ex@R=0.5ex{
       & \mbox{\LARGE $\bullet$} \ar@{.}[dr]  && \mbox{\LARGE $\star$} \\
       & & \mbox{\LARGE $\bullet$} \ar@{.}[ld] \ar@{.}[rd]\ar@{.}[ur]&& \\
            & & && {\phantom 2}}}} \raisebox{20pt}{{$\Rightarrow$}}
    \raisebox{3.1em}{\scalebox{0.6}{\xymatrix@!C=0.1ex@R=0.5ex{
        & & & & & {\phantom j}  \\
        &   &  & \mbox{\LARGE $\bullet$} \ar@{.}[ul]\ar@{.}[ur] \ar@{.}[dr] && \\
        & &\mbox{\LARGE $\star$}\ar@{.}[ur] &  &\mbox{\LARGE $\bullet$}& }}} \\
   \raisebox{23pt}{(iii)} \raisebox{3.1em}{\scalebox{0.6}{\xymatrix@!C=0.1ex@R=0.5ex{
       & \mbox{\LARGE $\bullet$} \ar@{.}[dr]  && \mbox{\LARGE $\bullet$} \\
       & & \mbox{\LARGE $\star$} \ar@{.}[ld] \ar@{.}[rd]\ar@{.}[ur]&& \\
            & & && {\phantom 2}}}} \raisebox{20pt}{{$\Rightarrow$}}
    \raisebox{3.1em}{\scalebox{0.6}{\xymatrix@!C=0.1ex@R=0.5ex{
        & & & & & {\phantom j}  \\
        &   &  & \mbox{\LARGE $\bullet$} \ar@{.}[ul]\ar@{.}[ur] \ar@{.}[dr] && \\
        & &\mbox{\LARGE $\star$}\ar@{.}[ur] &  &\mbox{\LARGE $\star$}& }}} 
    \raisebox{23pt}{(ii)} \raisebox{3.1em}{\scalebox{0.6}{\xymatrix@!C=0.1ex@R=0.5ex{
       & \mbox{\LARGE $\star$} \ar@{.}[dr]  && \mbox{\LARGE $\star$} \\
       & & \mbox{\LARGE $\bullet$} \ar@{.}[ld] \ar@{.}[rd]\ar@{.}[ur]&& \\
            & & && {\phantom 2}}}} \raisebox{20pt}{{$\Rightarrow$}}
    \raisebox{3.1em}{\scalebox{0.6}{\xymatrix@!C=0.1ex@R=0.5ex{
        & & & & & {\phantom j}  \\
        &   &  & \mbox{\LARGE $\star$} \ar@{.}[ul]\ar@{.}[ur] \ar@{.}[dr] && \\
        & &\mbox{\LARGE $\bullet$}\ar@{.}[ur] &  &\mbox{\LARGE $\bullet$}& }}}
 \end{gather*}
 Therefore, (i) (resp.\ (ii)) is seen as moving to lower right (resp.\ lower left), (iii) as splitting in two points, and (iv) as combining two points into one.
\end{Rem}

For $\flat \in \{<,>, \mrU,\mrD\}$, we write $\gG_0^{\gT,\flat} := \gG_0^\gT \cap \wh{Q}_0^{\gT,\flat}$.
We also write $\gG_0^{\gT,\geq} := \gG_0^{\gT} \setminus \gG_0^{\gT,<}$.
For $(i,k) \in \gG_0^{\gT,\geq}$, define $X^-_{i,k}$ and $X^+_{i,k}$, each of which is an element of $\gG_0^\gt$ (not $\gG_0^\Theta$) or the empty set, as follows:
\[ X^-_{i,k} = \begin{cases}  \emptyset & \text{if} \ (i,k) \in \gG^{\gT,\mrD}_0,\\
                              \big(\frac{1}{2}(\Theta_i + k +2), \frac{1}{2}(\Theta_i+k-2)\big) & \text{if} \ (i,k) \in \gG_0^{\gT,>} \sqcup \gG^{\gT,\mrU}_0,
               \end{cases}
\]
and
\[ X^+_{i,k} = \begin{cases} \emptyset & \text{if} \ (i,k) \in \gG^{\gT,\mrU}_0,\\
                            \big(\frac{1}{2}(\Theta_i-k+2n),\frac{1}{2}(-\Theta_i+k+2n)\big) & \text{if} \ (i,k) \in \gG_0^{\gT,>} \sqcup \gG_0^{\gT,\mrD}.\end{cases}
\]
In addition, for $(i,k), (i',k') \in \gG^{\gT,\geq}_0$ such that $(i',k')$ is in snake position (of twisted type) with respect to $(i,k)$,
define $X_{i,k}^{i',k'}$, which is an element of $\gG_0^\gt$ or the empty set, by
\[ X_{i,k}^{i',k'} = \begin{cases} \emptyset & \text{if} \ (i,k) \in \gG^{\gT,\mrU}_0,\\
                                   \big(\frac{1}{2}(\Theta_i+\Theta_{i'}-k+k'), \frac{1}{2}(-\Theta_i+\Theta_{i'}+k+k')\big) & \text{if} \ (i,k) \in \gG_0^{\gT,>} \sqcup \gG_0^{\gT,\mrD}. \\
                                   \end{cases} 
\]
When $n=7$, these are illustrated as follows. 
Here $\gG^{\gT,<}_0$ is omitted, $(i,k)$ and $(i',k')$ are shown as $\circ$ whose first coordinates are read from the left-hand scale, 
and $X^{\pm}_{i,k}$ and $X_{i,k}^{i',k'}$ are
shown as $\star$ whose first coordinates are read from the right-hand scale:  
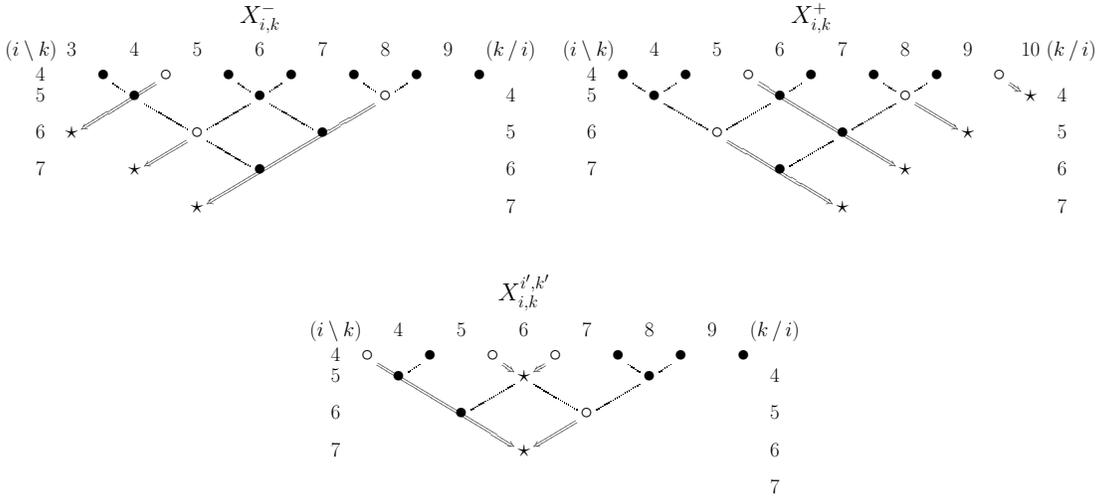
\begin{gather*}
\raisebox{2em}{\scalebox{0.45}{\xymatrix@!C=0.5ex@R=0ex{&&&&&&&\mbox{\huge$X_{i,k}^-$}\\
\mbox{\LARGE$(i\setminus k)$\ \ \ }&\mbox{\LARGE$3$}&&\mbox{\LARGE$4$}&&\mbox{\LARGE$5$}&&\mbox{\LARGE$6$}&&\mbox{\LARGE$7$}&&\mbox{\LARGE$8$}&&\mbox{\LARGE$9$}&&
\mbox{\LARGE $(k\, \slash \, i)$}\\
\mbox{\LARGE$4$}&& \mbox{\huge $\bullet$}\ar@{.}[dr]  && \mbox{\huge $\circ$}\ar@2{->}[dddlll]  &&\mbox{\huge $\bullet$} \ar@{.}[dr] 
&& \mbox{\huge $\bullet$}  && \mbox{\huge$\bullet$} \ar@{.}[dr] & &\mbox{\huge$\bullet$}&& \mbox{\huge $\bullet$}\\
\mbox{\LARGE$5$}& &&\mbox{\huge $\bullet$}\ar@{.}[ddrr]  &&&& \mbox{\huge $\bullet$}  \ar@{.}[ur] \ar@{.}[ddrr] 
&&&&\mbox{\huge $\circ$}\ar@{.}[ur]  \ar@2{->}[ddddddllllll] &&&&\mbox{\LARGE$4$}\\ & {\phantom 2} & &&&&&&&&&&& \\
\mbox{\LARGE$6$}&\mbox{\huge $\star$}&&&& \mbox{\huge $\circ$} \ar@2{->}[ddll]\ar@{.}[uurr]\ar@{.}[ddrr]
&&&& \mbox{\huge $\bullet$}   &&&&& &\mbox{\LARGE$5$}\\& {\phantom 2} \\
\mbox{\LARGE$7$}& &&\mbox{\huge $\star$} &&&&  \mbox{\huge $\bullet$}  &&&&&&&&\mbox{\LARGE$6$}  \\ {\phantom 2} \\
&&&&& \mbox{\huge $\star$}&&&&&&&&&&\mbox{\LARGE$7$}
 }}} \ \ 
\raisebox{2em}{\scalebox{0.45}{\xymatrix@!C=0.5ex@R=0ex{&&&&&&&\mbox{\huge$X_{i,k}^+$}\\
\mbox{\LARGE$(i\setminus k)$\ }&&\mbox{\LARGE$4$}&&\mbox{\LARGE$5$}&&\mbox{\LARGE$6$}&&\mbox{\LARGE$7$}&&\mbox{\LARGE$8$}&&\mbox{\LARGE$9$}&&\mbox{\LARGE$10$}&
\mbox{\LARGE \ \ \ $(k\, \slash \, i)$}\\
\mbox{\LARGE$4$}& \mbox{\huge $\bullet$}\ar@{.}[dr]  && \mbox{\huge $\bullet$}  &&\mbox{\huge $\circ$}  \ar@2{->}[dddddrrrrr] 
&& \mbox{\huge $\bullet$}  && \mbox{\huge$\bullet$} \ar@{.}[dr] & &\mbox{\huge$\bullet$}&&\mbox{\huge$\circ$}\ar@2{->}[dr] \\
\mbox{\LARGE$5$}&&\mbox{\huge $\bullet$}\ar@{.}[ur]\ar@{.}[ddrr]  &&&& \mbox{\huge $\bullet$}  \ar@{.}[ur] 
&&&&\mbox{\huge $\circ$}\ar@2{->}[ddrr]\ar@{.}[ur]  &&  &&\mbox{\huge $\star$}&\mbox{\LARGE$4$}\\ & {\phantom 2} & &&&&&&&&&&&\\
\mbox{\LARGE$6$}&&&& \mbox{\huge $\circ$} \ar@{.}[uurr]\ar@2{->}[ddddrrrr]
&&&& \mbox{\huge $\bullet$} \ar@{.}[uurr]  &&&&\mbox{\huge $\star$} &&&\mbox{\LARGE$5$}\\ {\phantom 2} \\
\mbox{\LARGE$7$}&& &&&&  \mbox{\huge $\bullet$} \ar@{.}[uurr] &&&& \mbox{\huge $\star$}&&&&&\mbox{\LARGE$6$}\\ {\phantom 2} \\
&&&& &&&& \mbox{\huge $\star$}&&&&&&&\mbox{\LARGE$7$}
 }}}
\end{gather*}
\begin{gather*}
\raisebox{2em}{\scalebox{0.45}{\xymatrix@!C=0.5ex@R=0ex{&&&&&&&\mbox{\huge$X_{i,k}^{i',k'}$}\\
&\mbox{\LARGE$(i\setminus k)$}&&\mbox{\LARGE$4$}&&\mbox{\LARGE$5$}&&\mbox{\LARGE$6$}&&\mbox{\LARGE$7$}&&\mbox{\LARGE$8$}&&\mbox{\LARGE$9$}&&
\mbox{\LARGE $(k\, \slash \, i)$}\\
&\mbox{\LARGE$4$}& \mbox{\huge $\circ$}\ar@2{->}[dddddrrrrr]  && \mbox{\huge $\bullet$}  &&\mbox{\huge $\circ$} \ar@2{->}[dr] 
&& \mbox{\huge $\circ$} \ar@2{->}[dl] && \mbox{\huge$\bullet$} \ar@{.}[dr] & &\mbox{\huge$\bullet$}&&\mbox{\huge$\bullet$}\\
 &\mbox{\LARGE$5$}&&\mbox{\huge $\bullet$}\ar@{.}[ur]  &&&& \mbox{\huge $\star$}  \ar@{.}[ddrr] 
&&&&\mbox{\huge $\bullet$}\ar@{.}[ur]&&&&\mbox{\LARGE$4$}  \\  {\phantom 2} & \\
&\mbox{\LARGE$6$}&&&& \mbox{\huge $\bullet$} \ar@{.}[uurr]
&&&& \mbox{\huge $\circ$} \ar@{.}[uurr] \ar@2{->}[ddll] &&&& &&\mbox{\LARGE$5$}\\ {\phantom 2} \\
& \mbox{\LARGE$7$}&& &&&&  \mbox{\huge $\star$}  &&&&&&&&\mbox{\LARGE$6$} \\ {\phantom 2}\\
&&&&&&&&&&&&&&&\mbox{\LARGE$7$}
 }}}
\end{gather*}
For a sequence $\bm{P}=\big((i_1,k_1),\ldots,(i_p,k_p)\big)$ of elements of $\gG^\xi_0$ with $\xi \in \{\gt,\gT\}$,
set $\bm{e}(\bm{P}) = \sum_r \bm{e}_{i_r,k_r} \in \Z_{\geq 0}^{\gG_0^\xi}$.

\begin{Prop}\label{Prop:change_of_words}
 Assume that $\bm{P} = \big((i_1,k_1),\ldots,(i_p,k_p)\big)$ is a snake of twisted type in $\gG^\gT_0 \subseteq \wh{Q}^{\Theta}_0$.  
 Let $1=r_0 \leq r_1 < \cdots < r_{t-1} <r_t < r_{t+1}=p+1$ be the unique increasing sequence satisfying the following for all $a \in [0,t]$:
 \begin{itemize}\setlength{\leftskip}{-15pt}
  \item[{\normalfont(i)}] if $a \in 2\Z$, then $(i_s,k_s) \in \gG^{\gT,<}_0$ holds for all $s \in [r_a,r_{a+1}-1]$, and
  \item[{\normalfont(ii)}] if $a \in 2\Z +1$, then $(i_s,k_s) \in \gG^{\gT,\geq}_0$ holds for all $s \in [r_a,r_{a+1}-1]$.
 \end{itemize}
 For each $a \in [1,t]$ with $a \in 2\Z+1$, set
 \begin{align*}
  \bm{P}^{(a)} &= \big((i_r,k_r),\ldots,(i_{r'-1},k_{r'-1})\big) \in (\gG_0^{\gT,\geq})^{r'-r}, \ \ \ \text{and}\\
    \bm{Q}^{(a)} &= (X_{i_{r},k_{r}}^-, X_{i_{r},k_{r}}^{i_{r+1},k_{r+1}},\ldots,X_{i_{r'-2},k_{r'-2}}^{i_{r'-1},k_{r'-1}}, X_{i_{r'-1},k_{r'-1}}^+) \in (\gG_0^\gt)^{r''} \
 \text{for some $r'' \in \Z_{> 0}$},
 \end{align*}
 where $r=r_a$ and $r'=r_{a+1}$, and $\emptyset$ are ignored.
 Let $\bm{P}^\dagger$ denote the sequence of elements of $\gG_0^\gt$ obtained from $\bm{P}$ by replacing the subsequences 
 $\bm{P}^{(a)}$ with $\bm{Q}^{(a)}$ for all $a \in [1,t]$ such that $a \in 2\Z+1$.
 Then $\bm{P}^\dagger$ is a snake of untwisted type in $\gG_0^\gt$ and we have $\rho\big(\bm{e}(\bm{P})\big) = \bm{e}(\bm{P}^\dagger)$.
\end{Prop}

\begin{proof}
 By tracing the transformations appearing in $\rho$ using Remark \ref{Rem:3-terms}, $\rho\big(\bm{e}(\bm{P})\big) = \bm{e}(\bm{P}^{\dagger})$ is proved directly.
 See the following example, which is the case where $n=7$, 
 \[ \bm{P}=\big((5,4), (5,6), (4,17/2), (4,19/2)\big) \ \ \text{and} \ \ \bm{P}^\dagger = \big((5,3),(5,5),(5,7),(4,10)\big).
 \]
\begin{gather*}
\raisebox{0em}{\scalebox{0.35}{\xymatrix@!C=0.5ex@R=0ex{
\mbox{\Large$(i \setminus k)$}  & \mbox{\LARGE $1$} && \mbox{\LARGE $2$} & & \mbox{\LARGE $3$} & & \mbox{\LARGE$4$} &  & \mbox{\LARGE$5$} &  & \mbox{\LARGE$6$} & & 
\mbox{\LARGE$7$} &  & \mbox{\LARGE$8$}&  & \mbox{\LARGE$9$}&  & \mbox{\LARGE$10$}&  & \mbox{\LARGE$11$}\\
\mbox{\LARGE$1$} &\mbox{\Huge $\bullet$} \ar@{.}[ddrr]&&&& \mbox{\Huge $\bullet$} \ar@{.}[ddrr]  &&&& \mbox{\Huge $\bullet$} \ar@{.}[ddrr] &&&& \mbox{\Huge $\bullet$} \ar@{.}[ddrr]
&&&& \mbox{\Huge $\bullet$}\ar@{.}[ddrr] &&&&\mbox{\Huge $\bullet$}\\ {\phantom 2} \\
\mbox{\LARGE$2$} &&&\mbox{\Huge $\bullet$} \ar@{.}[ddrr]\ar@{.}[uurr]&&&& \mbox{\Huge $\bullet$}\ar@{.}[ddrr]\ar@{.}[uurr]&&&&  \mbox{\Huge $\bullet$} \ar@{.}[ddrr]\ar@{.}[uurr] &&&& \mbox{\Huge $\bullet$}
\ar@{.}[uurr]\ar@{.}[ddrr] &&&&\mbox{\Huge $\bullet$}\ar@{.}[uurr]\\ {\phantom 2}  \\ 
\mbox{\LARGE$3$} &&&&& \mbox{\Huge $\bullet$} \ar@{.}[dr] \ar@{.}[uurr] &&&& \mbox{\Huge $\bullet$} \ar@{.}[dr]\ar@{.}[uurr] &&&& \mbox{\Huge $\bullet$} \ar@{.}[dr] \ar@{.}[uurr]
&&&& \mbox{\Huge $\bullet$} \ar@{.}[uurr]\ar@{.}[dr]&\\
\mbox{\LARGE$4$} &&&&&& \mbox{\Huge $\bullet$}\ar@{.}[dr]  && \mbox{\Huge $\bullet$} \ar@{.}[ur] &&\mbox{\Huge $\bullet$} \ar@{.}[dr] 
&& \mbox{\Huge $\bullet$} \ar@{.}[ur] && \mbox{\Huge$\bullet$} \ar@{.}[dr] & &
\mbox{\Huge $\circ$} \ar@{.}[ur] \ar@2{->}[dddlll]& &\mbox{\Huge $\circ$}\ar@2{->}[dr] \\
\mbox{\LARGE$5$} &&&&&&&\mbox{\Huge $\circ$}\ar@{.}[ur] \ar@2{->}[ddrr] \ar@2{->}[ddll]&&&& \mbox{\Huge $\circ$}  \ar@{.}[ur] \ar@2{->}[ddrr] \ar@2{->}[ddll]
&&&&\mbox{\Huge $\bullet$}  &&&& \\  {\phantom 2} & \\
\mbox{\LARGE$6$} &&&&&  &&&& \mbox{\Huge $\bullet$} \ar@{.}[ddrr]
&&&& \mbox{\Huge $\bullet$}  &&&& \\ {\phantom 2} \\
\mbox{\LARGE$7$} &&&&& && &&&&  \mbox{\Huge $\bullet$} \ar@{.}[uurr] &&&& }}} 
\hspace{-10pt}\raisebox{-4em}{$\Rightarrow$}\raisebox{0em}{\scalebox{0.35}{\xymatrix@!C=0.5ex@R=0ex{
\mbox{\Large$(i\setminus k)$} & \mbox{\LARGE$1$} && \mbox{\LARGE$2$} && \mbox{\LARGE$3$} & & \mbox{\LARGE$4$} & & \mbox{\LARGE$5$} &  
& \mbox{\LARGE$6$} &  & \mbox{\LARGE$7$}&  & \mbox{\LARGE$8$}&  & \mbox{\LARGE$9$}&  & \mbox{\LARGE$10$}&  & \mbox{\LARGE$11$} \\
\mbox{\LARGE$1$} &\mbox{\Huge $\bullet$} \ar@{.}[ddrr]&&&& \mbox{\Huge $\bullet$} \ar@{.}[ddrr]  &&&& \mbox{\Huge $\bullet$} \ar@{.}[ddrr] &&&& \mbox{\Huge $\bullet$} \ar@{.}[ddrr]
&&&& \mbox{\Huge $\bullet$}\ar@{.}[ddrr]&&&&\mbox{\Huge $\bullet$}\\ {\phantom 2} \\
\mbox{\LARGE$2$}&&&\mbox{\Huge $\bullet$} \ar@{.}[ddrr]\ar@{.}[uurr]&& && \mbox{\Huge $\bullet$}\ar@{.}[ddrr]\ar@{.}[uurr]&&&&  \mbox{\Huge $\bullet$} \ar@{.}[ddrr]\ar@{.}[uurr] &&&& \mbox{\Huge $\bullet$}\ar@{.}[ddrr]
\ar@{.}[uurr] &&&& \mbox{\Huge $\bullet$}\ar@{.}[uurr]\\ {\phantom 2} \\ 
\mbox{\LARGE$3$} &&&&& \mbox{\Huge $\bullet$} \ar@{.}[ddrr] \ar@{.}[uurr] &&&& \mbox{\Huge $\bullet$} \ar@{.}[ddrr]\ar@{.}[uurr] &&&& \mbox{\Huge $\bullet$} \ar@{.}[ddrr] \ar@{.}[uurr]
&&&& \mbox{\Huge $\bullet$} \ar@{.}[uurr]\ar@{.}[ddrr]&\\ {\phantom 2} \\
\mbox{\LARGE$4$}&&&&&&&\mbox{\Huge $\bullet$}\ar@{.}[ddrr]\ar@{.}[uurr] &&&& \mbox{\Huge $\bullet$}  \ar@{.}[ddrr]\ar@{.}[uurr] &&&&\mbox{\Huge $\bullet$}\ar@{.}[uurr]
\ar@{.}[ddrr]&&&&\mbox{\Huge $\circ$}\\ {\phantom 2} \\
\mbox{\LARGE$5$} &&&&& \mbox{\Huge $\circ$}\ar@{.}[uurr]\ar@{.}[ddrr] &&&& \mbox{\Huge $\circ$} \ar@{.}[uurr]\ar@{.}[ddrr]
&&&& \mbox{\Huge $\circ$} \ar@{.}[uurr] \ar@{.}[ddrr] &&&&  \mbox{\Huge $\bullet$}\ar@{.}[uurr]\\ {\phantom 2} \\ 
\mbox{\LARGE$6$}&&&&&&& \mbox{\Huge $\bullet$}\ar@{.}[uurr]\ar@{.}[ddrr] &&&&  \mbox{\Huge $\bullet$} \ar@{.}[uurr] \ar@{.}[ddrr]&&&& \mbox{\Huge $\bullet$}\ar@{.}[uurr] &&&& \\
 &&  &&&&  && & && & {\phantom 2} \\
\mbox{\LARGE$7$}&&&&& &&&& \mbox{\Huge $\bullet$}\ar@{.}[uurr] &&&&\mbox{\Huge $\bullet$}\ar@{.}[uurr]&&&&
}}} 
\end{gather*}
It is also checked directly by inspection that $\bm{P}^{\dagger}$ is a snake of untwisted type.
\end{proof}

\begin{Exa}\normalfont
 When $n=15$ and 
 \[ \bm{P}=\big((9,8),(9,10),(8,25/2),(7,15),(8,35/2),(9,20)\big),
 \]
 we have
 \begin{align*}
  \bm{P}^\dagger &= (X_{9,8}^-,X_{9,8}^{9,10}, X_{9,10}^{8,25/2}, (7,15),X_{8,35/2}^{9,20},X_{9,20}^+)\\
                 &=\big((9,7),(9,9),(9,11),(7,15),(9,19),(9,21)\big).
 \end{align*}
\end{Exa}


\subsection{Snake modules associated with twisted height functions}

Fix a strong duality datum $\cD = \{\sL_i\}_{i \in I} \subseteq \scC_\fg$ associated with $\mathfrak{sl}_{n+1}$.
In the sequel, we usually omit $\cD$ and write $S_{i,k}^{\Theta}$ for $S_{i,k}^{\cD,\Theta}$, etc.
All the assertions in this subsection will be stated for $S_{i,k}^{\Theta}$, but no generality is lost by this specific choice of a twisted height function by Lemma \ref{Lem:change_of_height_function}.

By Proposition \ref{Prop:change_of_words}, (\ref{eq:coincidence_of_two_B}) and Lemma \ref{Lem:parametrization}, the following lemma is immediate.

\begin{Lem}\label{Lem:twisted=untwisted}
 Assume that 
 $\bm{P} =\big((i_1,k_1),\ldots,(i_p,k_p)\big)$ is a snake in $\gG^\gT_0$,
 and let $\bm{P}^\dagger$ be the snake in $\gG^{\gt}_0$ given in Proposition \ref{Prop:change_of_words}.
 Then the snake module $\bS^\Theta(\bm{P})$ of twisted type is isomorphic to the snake module $\bS^\theta(\bm{P}^\dagger)$ of untwisted type.
\end{Lem}

\begin{Rem}\normalfont
 We can also show that for any snake $\bm{P}$ of twisted type (not necessarily contained in $\gG^\gT_0$),
 the snake module $\bS^{\Theta}(\bm{P})$ of twisted type is isomorphic to a snake module of untwisted type.
 We will not use this fact in the sequel.
\end{Rem}


\begin{Lem}\label{Lem:twist_cases}
 Let $\bm{P}=\big((i_1,k_1),\ldots,(i_p,k_p)\big)$ be a snake in $\wh{Q}^{\gT}_0$, and $(j,l) \in \wh{Q}^\gT_0$.
 \begin{itemize}\setlength{\leftskip}{-13pt}
  \item[{\normalfont (1)}] Suppose that $(j,l) \prec (i_1,k_1)$.
  \begin{itemize}\setlength{\leftskip}{-30pt}
   \item[{\normalfont (a)}] If $(i_1,k_1)$ is in prime snake position with respect to $(j,l)$, we have $\tfd\big(S^\Theta_{j,l},\bS^\gT(\bm{P})\big)=1$.
   \item[{\normalfont (b)}] Assume that $\Theta_{i_1} \in\{ \Theta_j \pm r\}$ and $k_1 = l+r$ hold for some $r \in \frac{1}{2}\Z_{> 0}$.
    Then we have $\tfd\big(S^\Theta_{j,l},\bS^\Theta(\bm{P})\big)=0$.
   \item[{\normalfont(c)}] If $(j,l) \in \wh{Q}^{\gT,\mrU}_0$ and $(i_1,k_1) \in \wh{Q}^{\gT,>}_0\sqcup \wh{Q}^{\gT,\mrU}_0$, or $(j,l) \in \wh{Q}^{\gT,\mrD}_0$ 
    and $(i_1,k_1) \in \wh{Q}^{\gT,<}_0\sqcup \wh{Q}^{\gT,\mrD}_0$,
    then we have $\tfd\big(S^\gT_{j,l},\bS^\gT(\bm{P})\big)=0$.
 \end{itemize}
  \item[{\normalfont (2)}] Suppose that $(i_p,k_p) \prec (j,l)$.
  \begin{itemize}\setlength{\leftskip}{-30pt}
   \item[{\normalfont (d)}] If $(j,l)$ is in prime snake position with respect to $(i_p,k_p)$, we have $\tfd\big(\bS^\gT(\bm{P}),S^\gT_{j,l}\big)=1$.
   \item[{\normalfont (e)}] Assume that $\Theta_j \in \{\Theta_{i_p} \pm r\}$ and $l = k_p+r$ hold for some $r \in \frac{1}{2}\Z_{>0}$. Then we have
    $\tfd\big(\bS^\gT(\bm{P}),S^\gT_{j,l}\big)=0$.
   \item[{\normalfont (f)}] If $(i_p,k_p) \in \wh{Q}^{\gT,<}_0\sqcup \wh{Q}^{\gT,\mrU}_0$ and $(j,l) \in \wh{Q}^{\gT,\mrU}_0$, or $(i_p,k_p) \in \wh{Q}^{\gT,>}\sqcup \wh{Q}^{\gT,\mrD}_0$
    and $(j,l) \in \wh{Q}^{\gT,\mrD}_0$ ,
    then we have $\tfd\big(\bS^\gT(\bm{P}),S^\gT_{j,l}\big)=0$.
  \end{itemize}
 \end{itemize}
\end{Lem}

\begin{proof}
 (1) By applying Lemma \ref{Lem:Dualizing_label}, if necessary, we may (and do) assume in the proof of each assertion that 
  \begin{itemize}\setlength{\leftskip}{-10pt}
   \item[(a)] $(j,l) \in \wh{Q}^{\gT,<}_0 \sqcup \wh{Q}^{\gT,\mrU}_0$,
   \item[(b)] $\Theta_{i_1} =\Theta_{j}-r$ and $k_1 = l+r$ for some $r$, and in particular, $j \notin \wh{Q}^{\gT,\mrD}_0$,
   \item[(c)] $(j,l) \in \wh{Q}^{\gT,\mrU}_0$ and $(i_1,k_1) \in \wh{Q}^{\gT,>}_0\sqcup \wh{Q}^{\gT,\mrU}_0$,
  \end{itemize}
  respectively.
  In all these three cases, we may further assume by a similar argument as in Lemma \ref{Lem:calculation_of_std} that
  \begin{equation}\label{eq:may_assume}
   l= \Theta_j-d_j \ \ \ \text{and} \ \ \ (i_s,k_s) \in \gG^\gT_0 \ \ \text{for all $s \in [1,p]$},
  \end{equation}
  where we set $d_i = 2 -\gd_{i,n_0}$ for $i \in I$.
  It follows from Lemma \ref{Lem:twisted=untwisted} that 
  \begin{align}
   \tfd\big(S^\gT_{j,l},\bS^\gT(\bm{P})\big) &=\tfd\big(\scD S^\gT_{j^*,\Theta_j-d_j+n}, \bS^\gT(\bm{P})\big)\nonumber\\
                                                      & =\tfd\big(\scD\, \bS^\gt\big(\big((j^*,\Theta_j-d_j+n)\big)^\dagger\big), \bS^\gt(\bm{P}^{\dagger})\big),\label{eq:tfdbigSjl}
  \end{align} 
  and it is directly checked from Proposition \ref{Prop:change_of_words} that 
  \[ \big((j^*,\Theta_j-d_j+n)\big)^{\dagger} = \begin{cases} \big((n,-2+n),(j^*,j-1+n)\big) & \text{if $j <n_0$}, \\ \big((n_0,n_0-1+n)\big) & \text{if $(j,l) = (n_0,n_0-3/2)$},\\
                                                            \big((j^*,j-3+n)\big) & \text{if $j>n_0$}, \end{cases}
  \]
  which implies 
  \[ \scD \,\bS^\gt\big((j^*,\Theta_j-d_j+n)^\dagger\big) \cong \begin{cases}  S^\gt_{1,-3} \nab S^\gt_{j,j-2} & \text{if $j<n_0$} , \\ S^\gt_{n_0,n_0-2} & \text{if $j=n_0$},\\
                                                                         S^\gt_{j,j-4} & \text{if $j>n_0$}.\end{cases}
  \]
  On the other hand, we can also show using the same proposition that 
  \begin{equation}\label{eq:Pdagger}
   \mathrm{FT}(\bm{P}^{\dagger}) = \begin{cases} (i_1, k_1)  & \text{if $(i_1,k_1) \in \gG^{\gT,<}_0$},\\
                                                   (i_1+s,k_1+s+1/2) \ \text{for some $s$} & \text{if $(i_1,k_1) \in \gG^{\gT,\mrD}_0$},\\
                                                   \big(\frac{1}{2}(\Theta_{i_1}+k_1+2),\frac{1}{2}(\Theta_{i_1}+k_1-2)\big) & \text{if $(i_1,k_1) \in  \gG^{\gT,>}_0 
   \sqcup\gG^{\gT,\mrU}_0$},
                                       \end{cases} 
  \end{equation}
  where we denote by $\FT (\bm{P}^\dagger)$ the first term of the sequence $\bm{P}^\dagger$.

  Now let us prove the assertion (a), where we are assuming that $(j,l) \in \wh{Q}^{\gT,<}_0 \sqcup \wh{Q}^{\gT,\mrU}_0$ and $l=\Theta_j-d_j$.
  By the above calculations, we have 
  \begin{equation}\label{eq:tfdbigStw}
   \tfd\big(S^\gT_{j,l},\bS^\gT(\bm{P})\big) = \begin{cases} \tfd\big(S^\gt_{1,-3} \nab S^\gt_{j,j-2}, \bS^\gt(\bm{P}^{\dagger})\big) & \text{if $j < n_0$},\\
                                                                        \tfd\big(S^\gt_{n_0,n_0-2},\bS^\gt(\bm{P}^{\dagger})\big) & \text{if $j = n_0$}.\end{cases}
  \end{equation}
  By the definition of the snake position, we have $(i_1,k_1) \in \gG^{\gT,<}_0\sqcup \gG^{\gT,\mrD}_0$,
  and then we easily see from (\ref{eq:Pdagger}) that 
  \[ \scD^{-1}_\theta(1,-3) = (n,n-2) \not \succeq \FT(\bm{P}^\dagger),
  \] 
  which implies that 
  $\scD^s S^\gt_{1,-3}$ and $\bS^\gt(\bm{P}^\dagger)$ strongly commute for all $s \in \Z_{\geq 0}$ by Lemma \ref{Lem:for_prime} (i).
  Therefore (\ref{eq:tfdbigStw}), together with Lemma \ref{Lem:invariance_by_head}, gives
  \[ \tfd\big(S^\gT_{j,l},\bS^\gT(\bm{P})\big) = \tfd\big(S^\gt_{j,j-2}, \bS^\gt(\bm{P}^{\dagger})\big) \ \ \ \text{ for all $j \leq n_0$}.
  \]
  It is easy to check using (\ref{eq:Pdagger}) that, if $(i_1,k_1)$ is in prime snake position with respect to $(j,l)$ in $\wh{Q}^{\gT}_0$, then
  $\mathrm{FT}(\bm{P}^{\dagger})$ is in prime snake position with respect to $(j,j-2)$ in $\wh{Q}^\gt_0$. 
  Since $\bm{P}^\dagger$ is a snake in $\wh{Q}_0^\gt$ by Proposition \ref{Prop:change_of_words}, the assertion (a) now follows from Lemma \ref{Lem:calculation_of_std} (a).

  Next let us prove (b), where we are assuming that
  \begin{equation}\label{eq:assumption_of_b}
   \Theta_{i_1} = \Theta_{j}-r, \ \ \ k_1 = l+r \ \ \text{for some $r$}, \ \ \ \text{and} \ \ \  l=\Theta_j-d_j.
  \end{equation}
  If $j \leq n_0$, then $i_1 < n_0$ holds and by the same calculation as above we have 
  \[ \tfd\big(S^\gT_{j,l},\bS^\gT(\bm{P})\big) = \tfd\big(S^\gt_{j,j-2}, \bS^\gt(\bm{P}^{\dagger})\big) \ \ \ \text{and} \ \ \ \FT(\bm{P}^\dagger) =(i_1,k_1).
  \]
  The assertion (b) then follows from Lemma \ref{Lem:calculation_of_std} (b).
  If $j>n_0$, on the other hand, we have 
  \[ \tfd\big(S^\gT_{j,l},\bS^\gT(\bm{P})\big) = \tfd\big(S^\gt_{j,j-4}, \bS^\gt(\bm{P}^\dagger)\big).
  \] 
  Noting that $(i_1,k_1) \notin \gG_0^{\gT,\mrD}$, it is easy to show from (\ref{eq:Pdagger}) and (\ref{eq:assumption_of_b}) that $\FT(\bm{P}^\dagger) =(j - r', j-4+r')$ holds for some $r'$,
  and hence the assertion follows from Lemma \ref{Lem:calculation_of_std} (b) in this case as well.

  Finally let us prove (c), where we are assuming that 
  \[ (j,l) =(n_0,n_0-3/2)\in \wh{Q}^{\gT,\mrU}_0 \ \ \  \text{and}  \ \ \ (i_1,k_1) \in \gG^{\gT,>}_0 \sqcup \gG_0^{\gT,\mrU}.
  \]
  As calculated above, we have $S^\gT_{j,l} \cong S^\gt_{n_0,n_0-2}$ and $\FT(\bm{P}^\dagger) = (n_0+s,n_0-2+s)$
  for suitable $s$. 
  Hence the assertion follows from Lemma \ref{Lem:calculation_of_std} (b).

   All the assertions of (2) are proved similarly.
\end{proof}

Let $(i,k),(i',k') \in \wh{Q}^\gT_0$, and suppose that $(i',k')$ is in prime snake position with respect to $(i,k)$.
We define $\bm{Q}_{i,k}^{i',k'}$ and $\bm{R}_{i,k}^{i',k'}$, each of which is a sequence of zero, one or two elements of $\wh{Q}^\gT_0$, as follows:
if $(i,k) \in \wh{Q}^{\gT,<}_0 \sqcup \wh{Q}^{\gT,\mrU}_0$, we set
\begin{align*}
 \bm{Q}_{i,k}^{i',k'} &= \begin{cases} \emptyset & (k'-k = \Theta_i+\Theta_{i'}),\\
                                      \big((\frac{1}{2}(\Theta_i+\Theta_{i'}+k-k'),\frac{1}{2}(\Theta_i-\Theta_{i'}+k+k'))\big) & (k'-k <\Theta_i+\Theta_{i'}),
                         \end{cases}\\[5pt]
 \bm{R}_{i,k}^{i',k'} &= \begin{cases} {\scalebox{1}{$\big((\frac{1}{2}(i+i'-k+k'),\frac{1}{2}(-i+i'+k+k'))\big)$}} & {\scalebox{0.9}{$(i,i'<n_0, k'-k <2n_0-i-i'),$}}\\
                                      {\scalebox{1}{$\big((n_0,-i+k+n_0-\frac{1}{2}), (n_0,i'+k'-n_0+\frac{1}{2})\big)$}} & {\scalebox{0.9}{$(i,i'<n_0, k'-k\geq 2n_0-i-i'),$}}\\
                                      \big((n_0,-i+k+n_0-\frac{1}{2})\big)  & (i<n_0, i'=n_0),\\
                                      \big((n_0,i'+k'-n_0+\frac{1}{2})\big) & (i=n_0, i'<n_0),\\
                                      \emptyset                             & (i=i'=n_0).\end{cases}
\end{align*}
If $(i,k) \in \wh{Q}^{\gT,>}_0 \sqcup \wh{Q}^{\gT,\mrD}_0$, putting $(j,l)=\scD(i,k)$ and $(j',l') = \scD(i',k')$, 
we set $\bm{Q}_{i,k}^{i',k'} = \scD^{-1} \bm{R}_{j,l}^{j',l'}$ and $\bm{R}_{i,k}^{i',k'} =
\scD^{-1}\bm{Q}_{j,l}^{j',l'}$.
These are illustrated as follows, where $(i,k)$ (resp.\ $(i',k')$, $\bm{Q}_{i,k}^{i',k'}$, $\bm{R}_{i,k}^{i',k'}$) are shown as $\circ$ (resp.\ $\bullet$, $\ast$, $\star$):
\begin{gather*}
\raisebox{0em}{\scalebox{0.45}{\xymatrix@!C=0.5ex@R=0ex{
\mbox{\Large$(i \setminus k)$}  & \mbox{\Large $0$} && \mbox{\Large $1$} && \mbox{\Large $2$} & & \mbox{\Large $3$} & & \mbox{\Large$4$} &  & \mbox{\Large$5$} &  & \mbox{\Large$6$} & & 
\mbox{\Large$7$} && \mbox{\Large $8$} && \mbox{\Large $9$} && \mbox{\Large $10$} && \mbox{\Large $11$} && \mbox{\Large $12$} && \mbox{\Large $13$} &&\mbox{\Large $14$} && \mbox{\Large $15$} &&
\mbox{\Large $16$}\\
\mbox{\Large$1$} &&& \mbox{\huge $\ast$} \ar@{-}[ddrr]  &&&&  \ar@{.}[ddrr] &&&&  \ar@{.}[ddrr]
&&&& \ar@{.}[ddrr]&&&& \ar@{.}[ddrr] &&&& \ar@{.}[ddrr]&&&&\mbox{\huge $\ast$}\ar@{-}[ddddrrrr]&&&&\ar@{.}[ddrr]&&  \\ {\phantom 2} \\
\mbox{\Large$2$} &\mbox{\huge $\circ$}\ar@{-}[ddrr]\ar@{-}[uurr]&&&& \mbox{\huge $\bullet$}\ar@{.}[ddrr]\ar@{.}[uurr]&&&&  \mbox{\huge $\ast$} \ar@{-}[ddrr]\ar@{.}[uurr] &&&& \ar@{.}[ddrr]
\ar@{.}[uurr]& & & & \mbox{\huge $\ast$}\ar@{-}[dddrrr]\ar@{.}[uurr]&&&&\ar@{.}[ddrr]\ar@{.}[uurr]&&&& \ar@{.}[ddrr]&&&&\ar@{.}[uurr]&&&&&&&\\ {\phantom 2}  \\ 
\mbox{\Large$3$} &&& \mbox{\huge $\star$} \ar@{.}[dr] \ar@{-}[uurr] &&&& \mbox{\huge $\circ$} \ar@{-}[dr]\ar@{-}[uurr] &&&& \mbox{\huge $\bullet$} \ar@{.}[dr] \ar@{.}[uurr]
&&&& \mbox{\huge $\circ$} \ar@{-}[uurr]\ar@{-}[dr]&&&& \ar@{.}[uurr] &&&& 
\ar@{.}[dr] &&&&\ar@{.}[dr]\ar@{.}[uurr]&&&& \mbox{\huge $\bullet$}\ar@{.}[uurr]\ar@{.}[dr]  &&&&&\\
\mbox{\Large$4$}  &&\ar@{.}[ur]&& \ar@{.}[dr]  &&  \ar@{.}[ur] &&\mbox{\huge $\star$} \ar@{.}[dr] 
&& \mbox{\huge $\star$} \ar@{-}[ur] && \mbox{\huge$\circ$} \ar@{-}[dr] & &
\mbox{\huge $\bullet$} \ar@{.}[ur] \ar@{.}[dl]& &\mbox{\huge$\star$}\ar@{.}[dr]&&\mbox{\huge$\ast$}\ar@{.}[ur]&&
\mbox{\huge$\bullet$}\ar@{.}[dr]&& \mbox{\huge$\circ$}\ar@{-}[uuuuurrrrr]  &&  
\ar@{.}[dr]&& \mbox{\huge$\bullet$}\ar@{.}[ur]  &&  \ar@{.}[dr]&& \mbox{\huge$\star$}\ar@{-}[ur]  &&  \ar@{.}[dr]&&&&&&&&&&&&&&&&&& \\
\mbox{\Large$5$}&\ar@{.}[ddrr]\ar@{.}[ur]&&&&\mbox{\huge $\ast$}\ar@{.}[ur] \ar@{-}[ddrr] \ar@{-}[ddll]&&&&   
\ar@{.}[ur] \ar@{.}[ddrr] \ar@{.}[ddll]
&&&&\mbox{\huge $\star$} \ar@{.}[ddrr] \ar@{-}[ur]&&&&  \ar@{.}[ddrr] &&&&\ar@{.}[ddrr] \ar@{.}[ur]
&&&&\ar@{.}[ddrr] &&&&\ar@{.}[ur]\ar@{.}[ddrr]&&&&&&&&&&\\  {\phantom 2} & \\
\mbox{\Large$6$} &&& \mbox{\huge $\circ$}\ar@{-}[ddrr]&&&& \mbox{\huge $\bullet$} \ar@{.}[ddrr]
&&&&   \ar@{.}[uurr] \ar@{.}[ddrr] &&&& \mbox{\huge $\circ$}\ar@{-}[uuurrr]\ar@{-}[ddddrrrr]  &&&& \ar@{.}[uurr] \ar@{.}[ddrr]
&&&& \ar@{.}[ddrr]&&&&\ar@{.}[uurr] \ar@{.}[ddrr]&&&&\ar@{.}[uurr]\ar@{.}[ddrr]&&&&&&&&&&&\\ {\phantom 2} \\
\mbox{\Large$7$}&\ar@{.}[uurr]&& && \mbox{\huge $\star$}\ar@{-}[uurr]&&&&   \ar@{.}[uurr] &&&& \ar@{.}[uurr]
&&&&\ar@{.}[uurr]&&&&&&&&\ar@{.}[uurr]&&&&\ar@{.}[uurr]&&&&&&&&&&&&&&&&&&&&&&&&&&&&&&&\\  {\phantom { 1}} \\ {\phantom { 1 }}
&&&&&&&&&&&&&&&&&&&\ar@{-}[uuuuuuurrrrrrr]&&&&&&&&&&&&&&&&&&&&&&&&&&&&&&&&&&&&}}} 
\end{gather*}

\begin{Lem}
 Let $(i,k),(i',k') \in \wh{Q}^\gT_0$, and assume that $(i',k')$ is in prime snake position with respect to $(i,k)$.
 Then we have 
 \[ S^\gT_{i,k} \Del S^\gT_{i',k'} \cong \bS^\gT(\bm{Q}_{i,k}^{i',k'}) \otimes \bS^\gT(\bm{R}_{i,k}^{i',k'}).
 \]
\end{Lem}

\begin{proof}
 Since
 \[ \scD(S^\gT_{i,k} \Del S^\gT_{i',k'}) \cong S^\gT_{\scD(i',k')} \nab S^\gT_{\scD(i,k)} \cong S^\gT_{\scD(i,k)} \Del S^\gT_{\scD(i',k')},
 \]
 we may assume by using Lemma \ref{Lem:Dualizing_label} that $(i,k) \in \wh{Q}^{\gT,<}_0 \sqcup \wh{Q}^{\gT,\mrU}_0$.
First let us consider the case $(i,k) \in \wh{Q}^{\gT,\mrU}_0$.
We may further assume by Lemma \ref{Lem:change_of_height_function} that $(i,k)=(n_0,n_0-3/2)$,
which forces $(i',k') \in \gG^{\gT,<}_0 \sqcup \gG^{\gT,\mrD}_0$.
If $(i',k') \in \gG^{\gT,<}_0$, it is proved directly from Proposition \ref{Prop:change_of_words} and Lemma \ref{Lem:product_of_two} that
\begin{align*}
  S^\gT_{i,k} \Del S^\gT_{i',k'} &\cong S^\gt_{n_0,n_0-2} \Del S^\gt_{i',k'} \cong S^\gt_{Q_{n_0,n_0-2}^{i',k'}} \otimes S^\gt_{R_{n_0,n_0-2}^{i',k'}}\\
                                         &\cong \bS^\gT(\bm{Q}_{i,k}^{i',k'}) \otimes \bS^\gT(\bm{R}_{i,k}^{i',k'}),
\end{align*}
and the assertion is proved. The case $(i',k') \in \gG^{\gT,\mrD}_0$ is proved by a similar calculation.

Next let us consider the case $(i,k) \in \wh{Q}^{\gT,<}_0$.
We may assume that $k=i$,
and then by the definition of the prime snake position we have $(i',k') \in \gG^{\gT,<}_0 \sqcup \gG^{\gT,\mrD}_0$.
Now the assertion is proved by a similar argument as above.
\end{proof}

Similarly to Lemma \ref{Lem:being_snake}, the following is proved by inspection.

\begin{Lem}[{\cite[Proposition 3.2]{MR2960028}}]
  Let $p \in \Z_{\geq 2}$, and $\bm{P}=\big((i_1,k_1),\ldots,(i_p,k_p)\big)$ be a prime snake in $\wh{Q}^\gT_0$.
  Set 
  \[ \bm{Q} = \bm{Q}_{i_1,k_1}^{i_2,k_2} * \cdots *\bm{Q}_{i_{p-1},k_{p-1}}^{i_p,k_p} \ \ \ \text{and} \ \ \ \bm{R} = \bm{R}_{i_1,k_1}^{i_2,k_2} * \cdots *\bm{R}_{i_{p-1},k_{p-1}}^{i_p,k_p},
  \]
  where $*$ denotes the concatenation.
  Then $\bm{Q}$ and $\bm{R}$ are snakes with no elements in common.
\end{Lem}

Now we give the main theorem of this section, which is a generalization of \cite[Theorem 3.4]{zbMATH06988693} and \cite[Proposition 3.1, Theorem 4.1]{MR2960028} in type $B$.

\begin{Thm}\label{Thm:Main_twisted}
  Let $\bm{P} = \big((i_1,k_1),\ldots,(i_p,k_p)\big)$ be a snake in $\wh{Q}^\gT_0$.
  \begin{itemize}\setlength{\leftskip}{-15pt}
   \item[{\normalfont(i)}] The simple module $\bS^\gT(\bm{P})$ is real.
   \item[{\normalfont(ii)}] If $\bm{P}$ is prime, then $\bS^\gT(\bm{P})$ is prime.
   \item[{\normalfont(iii)}] Assume that $\bm{P}$ is prime with $p \geq 2$, and set  
  \[ \bm{Q} = \bm{Q}_{i_1,k_1}^{i_2,k_2} * \cdots *\bm{Q}_{i_{p-1},k_{p-1}}^{i_p,k_p} \ \ \ \text{and} \ \ \ \bm{R} = \bm{R}_{i_1,k_1}^{i_2,k_2} * \cdots *\bm{R}_{i_{p-1},k_{p-1}}^{i_p,k_p}.
  \]
  Then $\bS^\gT(\bm{Q})$ and $\bS^\gT(\bm{R})$ strongly commute, and there is a short exact sequence
  \begin{equation}\label{eq:B_ses}
   0 \to \bS^\gT (\bm{Q}) \otimes \bS^\gT (\bm{R}) \to \bS^\gT(\bm{P}_{[1,p-1]}) \otimes \bS^\gT(\bm{P}_{[2,p]}) \to \bS^\gT(\bm{P}) \otimes 
     \bS^\gT(\bm{P}_{[2,p-1]}) \to 0.
  \end{equation}
  \end{itemize}
\end{Thm}

\begin{proof}
  The proof is similar to that of Theorem \ref{Thm:MainA}.

  Using Lemma \ref{Lem:twist_cases}, the assertions (i) and (ii) are proved from Propositions \ref{Prop:reality} and \ref{Prop:primeness}, respectively.
  For (iii), by the same argument as in the proof of Theorem \ref{Thm:MainA}, it suffices to show for $p \geq 3$ that 
  \[ \tfd\big(\bS^\gT(\bm{Q}'),\bS^\gT(\bm{R}_{i_{p-1},k_{p-1}}^{i_p,k_p})\big) = 0=\tfd\big(\bS^\gT(\bm{R}'),\bS^\gT(\bm{Q}_{i_{p-1},k_{p-1}}^{i_p,k_p})\big),
  \]
  where we set 
  \[ \bm{Q}' = \bm{Q}_{i_1,k_1}^{i_2,k_2} * \cdots * \bm{Q}_{i_{p-2},k_{p-2}}^{i_{p-1},k_{p-1}} \ \ \ \text{and} \ \ \ \bm{R}' = \bm{R}_{i_1,k_1}^{i_2,k_2}*\cdots * 
     \bm{R}_{i_{p-2},k_{p-2}}^{i_{p-1},k_{p-1}}.
  \]
  Let us show the former equality (the latter is proved similarly).
  Set $\tilde{\bm{R}} = \bm{R}_{i_{p-1},k_{p-1}}^{i_p,k_p}$, which we may assume not to be the empty set. 
  If $\bm{Q}_{i_{p-2},k_{p-2}}^{i_{p-1},k_{p-1}} \neq \emptyset$, then the equality is proved from Lemma \ref{Lem:twist_cases} (e) and (f).
  Assume that $\bm{Q}_{i_{p-2},k_{p-2}}^{i_{p-1},k_{p-1}} = \emptyset$.
  If $(i_{p-2},k_{p-2}) \in \wh{Q}^{\gT,<}_0 \sqcup \wh{Q}^{\gT,\mrU}_0$, then the same argument as in the proof of Theorem \ref{Thm:MainA} shows that
  \begin{equation}\label{eq:Rnotpreceq}
   R \not\preceq \scD^{-1} Q \ \ \ \text{for all $Q \in \bm{Q}'$ and $R \in \tilde{\bm{R}}$},
  \end{equation}
  and hence the equality holds.
  If $(i_{p-2},k_{p-2}) \in \wh{Q}^{\gT,>}_0 \sqcup \wh{Q}^{\gT,\mrD}_0$, on the other hand, $\bm{Q}_{i_{p-2},k_{p-2}}^{i_{p-1},k_{p-1}} = \emptyset$ implies 
  \[ (i_{p-2},k_{p-2}) \in \wh{Q}^{\gT,\mrD}_0, \ \ \ \ (i_{p-1},k_{p-1}) \in \wh{Q}^{\gT,\mrU}_0 \ \ \ \text{and} \ \ \ \tilde{\bm{R}} \subseteq \wh{Q}^{\gT,\mrU}_0.
  \]
  If $p=3$ or $\bm{Q}_{i_{p-3},k_{p-3}}^{i_{p-2},k_{p-2}} =\emptyset$, then (\ref{eq:Rnotpreceq}) holds as well, and the equality follows.
  If $\bm{Q}_{i_{p-3},k_{p-3}}^{i_{p-2},k_{p-2}} \neq \emptyset$, on the other hand, we have $\bm{Q}_{i_{p-3},k_{p-3}}^{i_{p-2},k_{p-2}} \subseteq \wh{Q}_0^{\gT,<}$
  and therefore the equality follows from Lemma \ref{Lem:twist_cases} (f).
  The proof is complete.
\end{proof}

\begin{Exa}\normalfont
 Assume that $\fg$ is of type $A_3^{(1)}$, and let $\cD = \{\sL_i\}_{i \in [1,3]} \subseteq \scC_\fg$ be the strong duality datum of type $\mathfrak{sl}_4$ defined by 
 \[ \sL_1=L(Y_{1,7}),\ \ \ \sL_2=L(Y_{2,4}),\ \ \text{and} \ \ \sL_3=L(Y_{3,7}).
 \]
 Then we have 
 \begin{align*}
   S_{1,1}^{\Theta} &= L(Y_{1,7}), \ \ S_{2,3/2}^{\Theta} = L(Y_{3,5}), \ \ S_{3,2}^{\Theta}=L(Y_{3,7}Y_{3,5}), \ \ S_{2,5/2}^{\Theta} = L(Y_{3,7}),\\ 
    S_{1,3}^{\Theta} &= L(Y_{1,5}), \ \ S_{2,7/2}^{\Theta} = L(Y_{2,4}),
 \end{align*}
 and the exact sequence (\ref{eq:B_ses}) for $\bm{P} = \big((3,2),(2,9/2),(2,11/2)\big)$ is as follows:
 \begin{align*}
  0 &\to \hd\big(L(Y_{3,7})\otimes L(Y_{1,3}Y_{1,1})\big) \to \hd\big(L(Y_{3,7}Y_{3,5})\otimes L(Y_{1,1})\big) \otimes \hd (L(Y_{1,1}) \otimes L(Y_{1,3})\big)\\
    &\to \hd\big(L(Y_{3,7}Y_{3,5})\otimes L(Y_{1,1}) \otimes L(Y_{1,3})\big) \otimes L(Y_{1,1}) \to 0,
 \end{align*}
 or more explicitly,
 \[ 0 \to L(Y_{3,7}Y_{1,3}Y_{1,1}) \to L(Y_{3,7}Y_{3,5}Y_{1,1}) \otimes L(Y_{2,2}) \to L(Y_{3,7}Y_{3,5}Y_{2,2}Y_{1,1}) \to 0.
 \]
\end{Exa}

\ \\
\noindent\textbf{Acknowledgments.}\ \ \
The author would like to thank Yoshiyuki Kimura, Se-jin Oh, and Hironori Oya for helpful comments.
He is also grateful to the anonymous referees for careful reading the manuscript and valuable comments.
He was supported by JSPS Grant-in-Aid for Scientific Research (C) No.~20K03554.

\def\cprime{$'$} \def\cprime{$'$} \def\cprime{$'$} \def\cprime{$'$}


\end{document}